\documentclass[11pt,reqno]{amsart}
\usepackage{tikz-cd}

\usepackage{times}
\usepackage{amsmath,amsfonts,amstext,amssymb,amsbsy,amsopn,amsthm,eucal}
\usepackage{txfonts}
\usepackage{dsfont}
\usepackage{graphicx}   % for figures
\usepackage{hyperref}
\usepackage{accents}
\usepackage{enumerate}
\usepackage{a4wide,xcolor,mathrsfs}
\usepackage[normalem]{ulem}

\usepackage{epsfig,amsthm}
\usepackage[latin1]{inputenc}

\newcommand{\supp}{\mathop{\rm supp}\nolimits}

\newcommand{\Ch}{{\rm Ch}}

\newcommand{\rank}{\text{rank}}

\newcommand{\mm}{\mathfrak m}

\newcommand{\sfd}{{\sf d}}

\renewcommand{\d}{{\rm{d}}}
\newcommand{\R}{\mathbb{R}}
\newcommand{\N}{\mathbb{N}}

\newcommand{\be}{{\rm b_1}}

\newcommand{\CD}{{\sf CD}}
\newcommand{\RCD}{{\sf RCD}}

\newcommand{\dZ}{\mathds{Z}}

\newcommand{\cB}{\mathcal{B}}

\newcommand{\cH}{\mathcal{H}}

\newcommand{\cL}{\mathcal{L}}

\newcommand{\cM}{\mathcal{M}}

\newcommand{\cP}{\mathcal{P}}

\newcommand{\cR}{\mathcal{R}}

\newcommand{\vare}{\varepsilon}
\newcommand{\eps}{\varepsilon}

\DeclareMathOperator{\diam}{diam}

\newtheorem{theorem}{Theorem}[section]

\newtheorem{proposition}[theorem]{Proposition}
\newtheorem{lemma}[theorem]{Lemma}
\newtheorem{corollary}[theorem]{Corollary}
\theoremstyle{definition}
\newtheorem{definition}[theorem]{Definition}
\theoremstyle{remark}
\newtheorem{remark}{Remark}[theorem]
\theoremstyle{remark}

\theoremstyle{remark}

\theoremstyle{remark}

\theoremstyle{remark}

\theoremstyle{remark}

\theoremstyle{remark}

\newtheorem*{proposition*}{Proposition}

\begin{document}
\title[Upper bound on the  revised first Betti number and torus stability for $\RCD$ spaces]{An upper bound on the revised first Betti number and a torus stability result for $\RCD$ spaces}
\author{Ilaria Mondello}
\address{Universit\'e Paris Est Cr\'eteil, Laboratoire d'Analyse et Math\'ematiques appliqu\'ees.}
\email{ilaria.mondello@u-pec.fr}
\author{Andrea Mondino}
\address{University of Oxford, Mathematical Institute}
\email{Andrea.Mondino@maths.ox.ac.uk}
\author{Raquel Perales}\address{CONACyT Research Fellow at Universidad Nacional Aut\'onoma de M\'exico, Oaxaca.}
\email{raquel.perales@im.unam.mx}

\date{\today}

\maketitle

\begin{abstract}
We prove an upper bound on the rank of the abelianised revised fundamental group  (called ``revised first Betti number'') of a compact $\RCD^{*}(K,N)$ space, in the same spirit of the celebrated Gromov-Gallot upper bound on the first Betti number for a  smooth compact Riemannian manifold with Ricci curvature bounded below.
When the synthetic lower Ricci bound is close enough to (negative) zero and the aforementioned upper bound on the revised first Betti number is saturated (i.e. equal to the integer part of $N$, denoted by $\lfloor N \rfloor$), then we establish a torus stability result stating that the space is $\lfloor N \rfloor$-rectifiable as a metric measure space, and  a finite cover  must be mGH-close to an $\lfloor N \rfloor$-dimensional flat torus; moreover, in case $N$ is an integer, we prove that the space itself is bi-H\"older homeomorphic to a flat torus. This second result extends to the class of non-smooth $\RCD^{*}(-\delta, N)$ spaces a celebrated torus stability theorem by Colding (later refined by Cheeger-Colding).
\end{abstract}

\tableofcontents

%%%%%%%%%%%%%%%
%%%%%%%%%%%%%%%
\newpage

\section{Introduction}\label{s:introduction}
Let us start by recalling that an $\RCD^{*}(K,N)$ space is a (possibly non-smooth) metric measure space $(X,\sfd,\mm)$ with dimension bounded above by $N\in [1,\infty)$ and Ricci curvature bounded below by $K\in \R$, in a synthetic sense (see Section \ref{ss-lowerRic} for the precise notions and the corresponding bibliography). The class of $\RCD^{*}(K,N)$ spaces is a natural non-smooth extension of the class of smooth Riemannian manifolds of dimension $\leq N$ and Ricci curvature bounded below by $K\in \R$, indeed:
\begin{itemize}
\item It contains the class of  smooth Riemannian manifolds of dimension $\leq N$ and Ricci curvature bounded below by $K\in \R$;
\item It is closed under pointed measured Gromov-Hausdorff convergence, so Ricci limit spaces are examples of $\RCD^{*}(K,N)$ spaces; 
\item It includes the class of $\lfloor N \rfloor$-dimensional Alexandrov spaces with curvature bounded below by $K/(\lfloor N \rfloor -1)$, the latter  being the synthetic extension of the class of smooth $\lfloor N \rfloor$-dimensional Riemannian manifolds with sectional curvature bounded below by $K/(\lfloor N \rfloor -1)$;
\item In contrast to the class of smooth Riemannian manifolds, it is closed under natural geometric operations such as quotients, foliations, conical and warped product constructions (provided natural assumptions are met);
\item Several fundamental comparison and structural results known for smooth Riemannian manifolds with Ricci curvature bounded below and for Ricci limits have been extended to $\RCD^{*}(K,N)$ spaces. 
\end{itemize}
It was proved by Wei and the second named author \cite{MondinoWei} (after Sormani-Wei \cite{SormaniWei2001, SormaniWei2004}) that an $\RCD^{*}(K,N)$ space $(X,\sfd,\mm)$ admits a universal cover $(\widetilde X,\sfd_{\widetilde X},\mm_{\widetilde X})$, which is an  $\RCD^{*}(K,N)$ space as well. The group of deck transformations on the universal cover  is called \emph{revised fundamental group of $X$} and denoted by $\bar \pi_1 (X)$ (see Section \ref{SSS:CovSpacesT} for the precise definitions and basic properties).
\medskip

We next discuss the main results of the present paper. Let $(X,\sfd,\mm)$ be a compact $\RCD^*(K,N)$ space and let $\bar \pi_1 (X)$ be its revised fundamental group. 
Set  
$$H:= [ \bar \pi_1(X), \bar \pi_1(X)] \quad \text{ and } \quad \Gamma:= \bar \pi_1(X)/ H$$
respectively the commutator and the abelianised revised fundamental group. 
As a consequence of Bishop-Gromov volume comparison, $\Gamma$ is finitely generated (see Proposition \ref{prop:RevFundGFinGen}, after Sormani-Wei \cite{SormaniWei}) and thus it can be written as
$$\Gamma= \mathbb Z^s \times \mathbb Z_{p_1}^{s_1} \times \cdots \times \mathbb Z_{p_l}^{s_l}.$$
 We define the \emph{revised first Betti number} of $(X,\sfd,\mm)$ as
 $$\be(X) : = \rank(\Gamma)=s.$$
\noindent

The goal of the paper is two-fold:
\begin{itemize}
\item  First, we prove an upper bound for the revised first Betti number of a compact  $\RCD^*(K,N)$ space, generalising to the non-smooth metric measure setting   a classical result of  M.~Gromov \cite{Gromov81} and S.~Gallot \cite{Gallot} originally proved for smooth Riemannian manifolds with Ricci curvature bounded below.  
\item Second, we prove a torus stability/almost rigidity result, roughly stating  that if $(X,\sfd,\mm)$ is a compact  $\RCD^*(-\vare,N)$ space with $\be(X)=\lfloor N \rfloor$, then a finite cover must be measured Gromov-Hausdorff close to a flat $\lfloor N \rfloor$-dimensional torus; if moreover $N$ is an integer, then $(X,\sfd)$ is bi-H\"older homeomorphic to a flat $N$-dimensional torus and $\mm$ is a constant multiple of the $N$-dimensional Hausdorff measure. This extends to the non-smooth $\RCD$ setting  a celebrated result by T.~Colding originally established for smooth Riemannian manifolds with Ricci curvature bounded below \cite[Theorem 0.2]{Col97} and later refined by Cheeger-Colding \cite[Theorem A.1.13]{CC97}; this proved an earlier conjecture by M.~Gromov.
\end{itemize}
\noindent
More precisely, the first main result is the following upper bound on $\be(X)$:

\begin{theorem}[An upper bound on $\be(X)$ for $\RCD^{*}(K,N)$ spaces]\label{lem-GalGro}
There exists a positive function $C(N,t)>0$ with $\lim_{t \to 0}C(N,t)=\lfloor N \rfloor$ such that for any compact $\RCD^*(K,N)$ space  $(X,\sfd,\mm)$ with $\supp(\mm)=X$, $\diam(X) \leq D$, for some $K\in \R, N\in [1,\infty), D>0$, the  revised first Betti number satisfies $\be(X) \leq C(N,KD^2)$.
\\In particular,  for any $N \in  [1,\infty)$ there exists $\eps(N)>0$ such that if $(X,\sfd,\mm)$ is a compact  $\RCD^*(K,N)$ space with
$\diam(X) \leq D$, $KD^2 \geq -\eps(N)$  then $\be(X) \leq \lfloor N \rfloor$. 
\end{theorem}

The upper bound of Theorem \ref{lem-GalGro} is sharp, as a flat $\lfloor N \rfloor$-dimensional torus ${\mathbb T}^{\lfloor N \rfloor}$, is an example of an $\RCD^{*}(0,\lfloor N \rfloor)$ space (thus of an $\RCD^{*}(-\vare,  N)$ space for any $\vare>0$) saturating the upper bound $\be({\mathbb T}^{\lfloor N \rfloor})= \lfloor N \rfloor$.
\\

In order to state the second main result,  let us adopt the standard notation $\vare(\delta|N)$ to denote a real valued function of $\delta$ and $N$ satisfying that $\lim_{\delta\to 0} \vare(\delta|N)=0$, for every fixed $N$. Let us also recall that (see Section \ref{SS:StructRCD} for more details and for the relevant bibliography):
\begin{itemize}
\item We say that $(X,\sfd,\mm)$ has essential dimension equal to $N\in \N$ if $\mm$-a.e. $x$ has a unique tangent space isometric to the $N$-dimensional Euclidean space $\R^{N}$;
\item We say that $(X,\sfd,\mm)$ is $N$-rectifiable as a metric measure space for some  $N\in \N$ if there exists a family of Borel subsets $U_{\alpha}\subset X$ and charts $\varphi_{\alpha}:U_{\alpha}\to \R^{N}$ which are bi-Lipschitz on their image such that $\mm(X\setminus \bigcup_{\alpha} U_{\alpha})=0$ and $\mm\llcorner U_{\alpha}\ll {\mathcal H}^{N}\llcorner U_{\alpha}$, where ${\mathcal H}^{N}$ denotes the $N$-dimensional Hausdorff measure.
\end{itemize}

\begin{theorem}[Torus stability for $\RCD^{*}(K,N)$ spaces]\label{thm:main}
For every $N\in [1,\infty)$ there exists $\delta(N) > 0$ with the following property. Let $(X,\sfd, \mm)$ be a compact $\RCD^{*}(K,N)$ space with $K \diam(X)^2 > -\delta(N)$ and $\be(X) =\lfloor N \rfloor$.
\begin{enumerate}
\item Then $(X,\sfd,\mm)$ has essential dimension equal to $\lfloor N \rfloor$ and it is $\lfloor N \rfloor$-rectifiable as a metric measure space.
\item There exists a finite cover $(X', \sfd_{X'}, \mm_{X'})$ of $(X,\sfd, \mm)$ which is $\vare(\delta|N)$-mGH close to a flat torus of dimension $\lfloor N \rfloor$. 
\item  If in addition $N\in \N$, then $\mm= c {\mathcal H}^N $ for some constant $c>0$ and $(X,\sfd)$ is  bi-H\"older homeomorphic to an $N$-dimensional flat torus.
\end{enumerate}
\end{theorem}

The \emph{torus stability} above should be compared with the \emph{torus rigidity} below, proved by Wei and the second named author \cite{MondinoWei},  extending  to the non-smooth $\RCD^*(0,N)$ setting a classical result of Cheeger-Gromoll \cite{ChGr}. See also Gigli-Rigoni \cite{GiRi} for a related torus rigidity result, where the
maximality assumption on the rank of the revised fundamental group is replaced by the maximality of the rank of harmonic one forms (recall that the rank of the space of harmonic one forms coincides with the first Betti number in the smooth setting).

\begin{theorem}[\cite{MondinoWei} after \cite{ChGr}]\nonumber
Let $(X,\sfd,\mm)$ be a compact $\RCD^*(0,N)$ space for some $N \in [1, \infty)$.  If the revised fundamental group $\bar \pi_1(X)$ contains $\lfloor N \rfloor$ independent generators of infinite order, then $(X,\sfd,\mm)$ is isomorphic as metric measure space  to a flat torus $\mathbb T^{\lfloor N \rfloor}= \R^{\lfloor N \rfloor}/ \Gamma$ for some lattice $\Gamma \subset \R^{\lfloor N \rfloor}$. 
\end{theorem}

\subsection{Outline of the arguments and organisation of the paper}

Our first goal will be to establish the Gromov-Gallot's upper bound on $\be(X)$ stated in Theorem \ref{lem-GalGro}. To that aim:
\begin{itemize}
\item Let $(X,\sfd,\mm)$ be a compact $\RCD^{*}(K,N)$ space.  If $N=1$ then all the results hold trivially (see Remark \ref{rem:N=1}). So we assume that $N\in (1,\infty)$;
\item  Let  $(\widetilde X,\sfd_{\widetilde X},\mm_{\widetilde X})$ be the universal cover of  $(X,\sfd,\mm)$. Recall that $(\widetilde X,\sfd_{\widetilde X},\mm_{\widetilde X})$ is an $\RCD^{*}(K,N)$ space as well, and the revised fundamental group $\bar \pi_{1}(X)$ acts on  $(\widetilde X,\sfd_{\widetilde X},\mm_{\widetilde X})$ by deck transformations (actually $\bar \pi_{1}(X)$ can be identified with the group of deck transformations on $\widetilde X$);
\item Let $H= [\bar \pi_1(X), \bar \pi_1(X)]$ be the commutator of  $\bar \pi_{1}(X)$ and consider the quotient space $\bar X= \widetilde{X} / H$.
$\bar X$ inherits a natural quotient metric measure structure from $\widetilde{X}$,  denoted by $(\bar X,\sfd_{\bar X},\mm_{\bar X})$, which satisfies the $\RCD^{*}(K,N)$ condition as well (see Corollary \ref{lem-barXRCD}). Moreover  $(\bar X,\sfd_{\bar X},\mm_{\bar X})$ is a covering space for $(X,\sfd,\mm)$, with fibres of countable cardinality (corresponding to $\Gamma:=  \bar \pi_1(X)/H$);
\item We will also consider $X':=X/\Gamma'$, where $\Gamma'\cong \mathbb Z^{\be(X)}$ is a suitable subgroup of $\Gamma$. More precisely, fix a point $\bar x\in \bar X$; extending a classical argument of Gromov to the non-smooth $\RCD$ setting,  one can construct $\Gamma'<\Gamma$  isomorphic to $\mathbb Z^{\be(X)}$ such that the distance between $\bar x$ and any element in 
$\Gamma' \bar x$ is bounded above and below uniformly in terms of $\diam(X)$ (see Lemma \ref{lem-Gam'} for the precise statement). The quotient space $(X', \sfd_{X'}, \mm_{X'})$
still satisfies the $\RCD^{*}(K,N)$ condition, it is a covering space for $(X,\sfd,\mm)$, with fibres of finite cardinality (corresponding to the index of $\Gamma'$ in $\Gamma$).  
\end{itemize}
After the above constructions, a counting argument combined with  Bishop-Gromov's volume comparison Theorem in  $(\bar X,\sfd_{\bar X},\mm_{\bar X})$ will give Theorem \ref{lem-GalGro}
at the end of Section \ref{s-betti}.
\\

In order to show Colding's torus stability for $\RCD^{*}(-\delta,N)$ spaces (i.e. Theorem \ref{thm:main}), in Section \ref{s-GHapp} we will construct  $\eps$-mGH approximations from large balls in $\bar X$ to balls of the same radius in the Euclidean space $\R^{\lfloor N \rfloor}$ (see Theorem \ref{thm-GH} for the precise statement). 

This is achieved by an inductive argument with $\lfloor N \rfloor$ steps: in each step we obtain that a ball in $\bar X$ is mGH close to a ball in a product $\R^n \times Y$, where $Y$ is an $\RCD^{*}(0, N -n)$ space. In order to prove the inductive step and pass from $n$ to $n+1$, we show that for $\delta>0$ small enough,  $Y$  must have large diameter, so that the almost splitting theorem  applies to $Y$. Therefore, we get a mGH approximation from a ball in $\bar X$ into $\R^{n+1}\times Y'$. The diameter estimate for $Y$ relies on the volume counting argument described in the previous paragraph and contained in Section \ref{s-betti}. 

The approach above is inspired by Colding's paper \cite{Col97}, however there are some substantial differences: indeed Colding's inductive argument is based on the construction of what are now known as $\delta$-splitting maps, while we only use $\vare$-mGH approximations and the almost splitting theorem; moreover the non-smooth $\RCD^{*}$ setting, in contrast to the smooth Riemannian framework, poses some challenges at the level of regularity, of global/local structure, and of topology. Below we briefly sketch the main lines of arguments; the expert will recognise the differences from \cite{Col97}.

The existence of $\eps$-mGH approximations into the Euclidean space yields the first claim of  Theorem \ref{thm:main}: for $\delta>0$ small enough, $(X,\sfd,\mm)$ has essential dimension equal to $\lfloor N \rfloor$, it is $\lfloor N \rfloor$-rectifiable as a metric measure space and moreover, if $N$ is an integer, the measure coincides with the Hausdorff measure $\cH^N$, up to a positive constant. This will be proved in Theorem \ref{main-delta}  by combining  Theorem \ref{thm-GH} with an $\eps$-regularity result by Naber and the second named author \cite{MondinoNaber}, revisited in the light of the constancy of dimension of $\RCD^{*}(K,N)$ spaces by Bru\'e-Semola \cite{BrueSemola} and a measure-rigidity result by Honda \cite{H19} for non-collapsed $\RCD^{*}(K,N)$ spaces.

When $\{(X_{i}, \sfd_{i}, \mm_{i})\}_{i\in \N}$ is a sequence of spaces as in the assumptions of Theorem \ref{thm:main} with $\delta_{i}\downarrow 0$, Theorem \ref{thm-GH}  yields  pmGH convergence for   $(\bar X_i, \sfd_{\bar X_{i}}, \mm_{\bar X_{i}})$ to the Euclidean space of dimension  $\lfloor N \rfloor$. Then by taking the subgroups  $\mathbb Z^{\lfloor N \rfloor} \cong \Gamma_{i}'<\Gamma_{i}:=\bar \pi_1(X_i)/ H_{i}$ already considered above (i.e. the ones constructed in Lemma \ref{lem-Gam'}, with $k=3$) and using equivariant Gromov-Hausdorff convergence (introduced by Fukaya \cite{Fu}  and further developed by Fukaya-Yamaguchi \cite{FuYa}), we deduce GH-convergence of (a non re-labeled subsequence of) $X_{i}':=\bar{X}_{i}/\Gamma_{i}'$  to a flat torus  of dimension $\lfloor N \rfloor$. This will show the second claim of Theorem \ref{thm:main} (see Proposition \ref{lem-eqGHconv} for more details).

When $N$ is an integer, the measure  of $X_{i}'$  coincides with $\cH^N$ (up to a constant), thanks to the aforementioned result by Honda \cite{H19}.  This fact  allows to apply Colding's volume convergence for $\RCD$ spaces proved by De Philippis-Gigli \cite{DPhG} and get that the GH-convergence obtained above can be promoted to mGH-convergence of $X_{i}'$ to a flat torus.
A recent result by Kapovitch and the second named author \cite{MondinoKapovitch} (which builds on top of Cheeger-Colding's metric Reifenberg theorem \cite{CC97}) states that for $N \in \N$, if a non-collapsed $\RCD^{*}(K,N)$ space is mGH-close enough to a compact smooth $N$-manifold $M$, then it is bi-H\"older homeomorphic to $M$. This implies that for $\delta>0$ small enough as in Theorem \ref{thm:main}, $X':=\bar{X}/\Gamma'$ is bi-H\"older homeomorphic to a flat torus and thus $\bar{X}$ is locally (on arbitrarily large compact subsets) bi-H\"older homeomorphic to $\R^{N}$.  In order to conclude the proof of the third claim of Theorem  \ref{thm:main}, we show that $\Gamma$ is torsion free, yielding that $\Gamma\cong \mathbb Z^{N}$ and thus $X=\bar{X}/\Gamma$ is  bi-H\"older homeomorphic to a flat torus. This last step uses the classical Smith's theory of groups of transformations with finite period.
\\

The paper is organised as follows. Section \ref{s:background} is devoted to recall previous results about $\RCD$ spaces, covering spaces and pointed Gromov-Hausdorff convergence (measured and equivariant) that are used in the rest of the paper. In particular, we show that a metric measure space $(X,\sfd,\mm)$ is $\RCD^*(K,N)$ if and only if any of its regular coverings with countable fibre is an $\RCD^*(K,N)$ space as well. This is essential since in our proofs we often use properties of $\RCD^*$ spaces on the coverings $\widetilde X, \bar X$ and $X'$ of $X$. Section \ref{s-betti} contains the proof of the upper bound for the revised first Betti number and its consequences. In Section \ref{s-GHapp}, we construct by induction $\eps$-mGH approximations between large balls in the covering $\bar X$ and balls in Euclidean space of dimension $\be(X)=\lfloor N \rfloor$. Section \ref{sec:5} is devoted to proving the $\lfloor N \rfloor$ rectifiability, i.e. the first claim of Theorem \ref{thm:main}. In Section \ref{s-proofMain}, we conclude the proof of Theorem \ref{thm:main} by first showing that $X'$ is GH-close to a flat torus  $\mathbb{T}^N$ and then obtaining that, for integer $N$,  $X'$ is bi-H\"older homeomorphic to $\mathbb{T}^N$ and $X=X'$. In the appendix we construct two explicit mGH-approximations that are used in Section \ref{s-GHapp}.

\bigskip

\noindent {\bf Acknowledgments}. 
 I.M. and R.P. wish to thank the Institut Henri Poincar\'e for its hospitality in July 2019 where they met to work on this project. 
\\ A.M. is supported by the European Research Council (ERC), under the European's Union Horizon 2020 research and innovation programme, via the ERC Starting Grant ``CURVATURE'', grant agreement No. 802689.
\\ R.P. wishes to thank the Mexican Math Society and the Kovalevskaya Foundation for the travel support received in November 2018 to visit I.M.  the Summer of 2018.  She also wants to thank the ANR grant:  ANR-17-CE40-0034 ``Curvature bounds and spaces of metrics'' for the support to travel within Europe to 
visit I.M. in July 2019. 
\\The authors thank Daniele Semola for carefully reading a preliminary version of the manuscript and for his comments. 

%%%%%%%%%%%%%%%%%%%%%%%%%%%%%%
%%%%%%%%%%%%%%%%%%%%%%%%%%%%%%

\section{Background}\label{s:background}

In this section we recall some fundamental notions about convergence of metric measure spaces and  about metric measures spaces with a synthetic lower bound on the Ricci curvature which will be used in the paper.

\subsection{Metric measure spaces and pointed metric measure spaces}

A \emph{metric measure space} (m.m.s. for short) is a triple $(X,\sfd,\mm)$ where $(X,\sfd)$ is a complete and separable metric space and  $\mm$ is a locally finite non-negative complete Borel measure on $X$, with  $X=\supp(\mm)$ and $\mm(X)>0$.
\\ A \emph{pointed metric measure space}  (p.m.m.s. for short) is a quadruple  $(X, \sfd,  \mm, \bar x)$ where $(X,\sfd,\mm)$ is a m.m.s. and $\bar x \in X$ is a given reference point.
\\  Two p.m.m.s.  $(X,\sfd,\mm, \bar x)$ and $(X',\sfd',\mm', \bar x')$ are said to be \emph{isomorphic} if there exists an isometry 
$$\varphi : (X,\sfd)\to(X',\sfd') \text{ such that } \varphi _\sharp\mm=\mm'  \text{ and } \varphi(\bar{x})=\bar{x}'.$$
Recall that $(X,\sfd)$ is said to be
\begin{itemize}
\item \emph{proper} if closed bounded sets are compact;
\item \emph{geodesic} if for every pair of points $x,y\in X$ there exists a length minimising geodesic from $x$ to $y$;
\end{itemize}
As we will recall later in this section,  the synthetic Ricci curvature lower bounds used in the paper (i.e. $\CD^{*}(K,N)$ for some $K\in \R,\, N\in [1,\infty)$) imply that $(X,\sfd)$ is proper and geodesic (see Remark \ref{rem:ProperGeod}).
 
\subsection{Gromov-Hausdorff convergence}\label{ss-conv}
 
We first define pointed measured Gromov-Hausdorff (pmGH)  convergence of p.m.m.s.  which will be used in Section 
\ref{s-GHapp}.   For details, see \cite{BuragoBuragoIvanov}, \cite{GMS2013} and  \cite{Villani09}.  Then we define equivariant pointed Gromov-Hausdorff  (EpGH)
convergence and state some results by Fukaya and Fukaya-Yamaguchi which will be employed in Section \ref{s-proofMain}. For  details see \cite{Fu}, \cite{FuYa}. 

\begin{definition}[Definition of pmGH convergence via pmGH approximations]
Let $(X_n,\sfd_n,\mm_n,\bar x_n)$, $n\in\N\cup\{\infty\}$, be a sequence of p.m.m.s. We say that  $(X_n,\sfd_n,\mm_n,\bar x_n)$ converges to 
$(X_\infty,\sfd_\infty,\mm_\infty,\bar x_\infty)$ in the pmGH sense if for any $\vare,R>0$ there exists $N({\vare,R})\in \N$ such that, for each $n\geq N({\vare,R})$, there exists a  Borel map $f^{R,\vare}_n:B_R(\bar x_n)\to X_\infty$ satisfying:
\begin{itemize}
\item $f^{R,\vare}_n(\bar x_n)=\bar x_\infty$;
\item $\sup_{x,y\in B_R(\bar x_n)}|\sfd_n(x,y)-\sfd_\infty(f^{R,\vare}_n(x),f^{R,\vare}_n(y))|\leq\vare$;
\item the $\vare$-neighbourhood of $f^{R,\vare}_n(B_R(\bar x_n))$ contains $B_{R-\vare}(\bar x_\infty)$,
\item $(f^{R,\vare}_n)_\sharp(\mm_n\llcorner{B_R(\bar x_n)})$ weakly converges to $\mm_\infty\llcorner{B_R(x_\infty)}$ as $n\to\infty$, for a.e. $R>0$. 
\end{itemize}
The maps $f^{R,\vare}_n:B_R(\bar x_n)\to X_\infty$ are called  $\varepsilon$-pmGH approximations.
\\ If we do not require the maps $f^{R,\vare}_n$ to be Borel, nor the last item to hold, we say that  the maps $f^{R,\vare}_n$ are  $\vare$-pGH approximations and  that 
the sequence converges in pointed Gromov-Hausdorff (pGH) sense. 
 \end{definition}

We next define equivariant pointed Gromov-Hausdorff  (EpGH) convergence. To this aim,  given a metric space $(X,\sfd)$, we endow its group of isometries $\mbox{Iso}(X)$ with the compact-open topology. In this case,  it is known that the compact-open topology is equivalent to the topology induced by uniform convergence on compact sets (see for example \cite[Theorem 46.8]{Munkres}). When $X$ is proper,  a sequence $(f_n)_{n \in \N}$ of isometries of $X$ converges to $f$ in the compact-open topology if and only if   $(f_n)_{n \in \N}$ converges to $f$ point-wise on $X$. 

\begin{remark}
Given any $x_0 \in X$, denote \begin{equation*}
\sfd_{x_0}(f,g)= \sup  \Big\{ \exp(-  \sfd(x_0,x))\, \underline{\sfd}(f(x),g(x)) \, \big|\,  x \in X \Big\},
\end{equation*}
where $\underline{\sfd}(x,y)=\min\{\sfd(x,y),1\}$.  If $(X,\sfd)$ is proper, one can check that $\sfd_{x_0}$ induces the compact-open topology and that  the group $( \text{Iso}(X), \sfd_{x_0})$ is a proper metric space.

\end{remark}

Let $\cM^p_{eq}$ be the set of quadruples $(X,\sfd, \bar x ,\Gamma)$, where $(X,\sfd, \bar x)$ is a proper pointed metric space and $\Gamma \subset \mbox{Iso}(X)$ is a closed subgroup of isometries. Define the set $\Gamma(r)= \{\gamma \in \Gamma \,|\, \gamma(\bar x) \in B_r(\bar x )\}$. We are now in position to define equivariant pointed Gromov-Hausdorff convergence for elements of $\cM^p_{eq}$. 

\begin{definition}\label{def-eqGH}
Let $(X_n,\sfd_{n},  \bar x_n, \Gamma_n) \in \cM^p_{eq}$, $n=1,2$. An $\varepsilon$-equivariant pGH  approximation is a triple of functions $(f,\phi,\psi)$,
$$f:  B_{\varepsilon^{-1}}(\bar x_1) \to X_2, \quad \phi:  \Gamma_1(\varepsilon^{-1}) \to \Gamma_2, \quad \psi:  \Gamma_2(\varepsilon^{-1}) \to \Gamma_1,$$
that satisfies 
\begin{enumerate}
\item $f(\bar x_1)=\bar x_2$;
\item The $\varepsilon$-neighbourhood of $f(B_{\varepsilon^{-1}}(\bar x_1) )$ contains $B_{\varepsilon^{-1}}(\bar x_2)$;
\item For all $x,y \in B_{\varepsilon^{-1}}(\bar x_1)$
$$|\sfd_{1}(x,y)- \sfd_{2}(f(x),f(y))| < \varepsilon;$$
\item For all $\gamma_1 \in \Gamma_1(\varepsilon^{-1})$ such that $x, \gamma_1x \in B_{\varepsilon^{-1}}(\bar x_1)$, it holds
$$\sfd_{2}(f(\gamma_1x), \phi(\gamma_1)f(x)) < \varepsilon;$$
\item For all $\gamma_2 \in \Gamma_2(\varepsilon^{-1})$ such that $x,  \psi(\gamma_2)x \in B_{\varepsilon^{-1}}(\bar x_1)$, it holds
$$\sfd_{2}(f( \psi(\gamma_2)x), \gamma_2f(x)) < \varepsilon.$$
\end{enumerate}
\end{definition}

Note that we do not assume  $f$ to be continuous, nor $\phi$ and $\psi$ to be homeomorphisms. 

\begin{definition}
A sequence $\{(X_n,\sfd_{n},\bar x_n, \Gamma_n)\}_{n \in \N}$ of spaces in $\cM^p_{eq}$ converges in the equivariant pointed Gromov-Hausdorff (EpGH for short) sense to $(X_\infty,\sfd_\infty,\bar x_\infty, \Gamma_\infty)\in \cM^p_{eq}$ if there exist $\varepsilon_n$-equivariant pGH approximations between $(X_n,\sfd_{n},\bar x_n, \Gamma_n)$ and $(X_\infty,\sfd_\infty,\bar x_\infty, \Gamma_\infty)$ such that $\varepsilon_n \to 0$, as $n\to \infty$.
\end{definition}

\begin{theorem}(Fukaya-Yamaguchi  \cite[Proposition 3.6]{FuYa})\label{thm-GHtoEq}
Let $\{(X_n,\sfd_{n},  \bar x_n, \Gamma_n)\}_{n\in \N}$ be a sequence in $\cM^p_{eq}$ such that $\{(X_n,\sfd_{n},  \bar x_n)\}_{n \in N}$ converges in the pointed Gromov-Hausdorff sense to $(X_\infty,\sfd_\infty,\bar x_\infty)$. Then there exist $\Gamma_\infty$ a closed subgroup of isometries of $X_\infty$ and a subsequence, 
$\{(X_{n_j},\sfd_{{n_j}}, \bar x_{n_j}, \Gamma_{n_j})\}_j \in \cM^p_{eq}$, that converges in equivariant pointed Gromov-Hausdorff sense to $(X_\infty,\sfd_\infty,\bar x_\infty, \Gamma_\infty)\in \cM^p_{eq}$.
\end{theorem}

For a closed subgroup $\Gamma$ in $\mbox{Iso}(X)$ and  $x\in X$, let $\Gamma x\subset X$ denote the orbit of $x$ under the action of $\Gamma$. The space of orbits is denoted by $X/\Gamma$.
Let
\begin{equation}\label{eq-quotdist}
\sfd_{X/\Gamma}( \Gamma x , \Gamma x')=  \inf \Big\{ \, \sfd_X(z,z')  \, \big|\,  z \in \Gamma x, \, z' \in \Gamma x' \Big\}.
\end{equation}  
It is a standard fact that $\sfd_{X/\Gamma}$ defines  a distance on $X/\Gamma$. Indeed, the equivalence between convergence in compact-open topology and point-wise convergence in $X$ implies that the orbits of $\Gamma$ are closed in $x$. Then consider $\Gamma x \neq \Gamma x'$ and assume by contradiction that $\sfd_{X/\Gamma}(\Gamma x, \Gamma x')=0$. Then there exists a sequence of points in $\Gamma x$ converging to a point $y$ in $\Gamma x'$, and since orbits are closed, $y$ belongs to $\Gamma x$ too. Therefore the two orbits coincide, which we assumed not. As a consequence, whenever $\Gamma x \neq \Gamma x'$ we have $\sfd_{X/\Gamma}(\Gamma x, \Gamma x')>0$.

\begin{theorem}(Fukaya  \cite[Theorem 2.1]{Fu})\label{thm-eqGHtoOrb}
Let $\{(X_n,\sfd_{n},  \bar x_n, \Gamma_n)\}_{n\in \N}$ be a sequence in $\cM^p_{eq}$ that converges in equivariant pointed Gromov-Hausdorff sense to   $(X_\infty,\sfd_\infty,\bar x_\infty, \Gamma_\infty)\in \cM^p_{eq}$. Then  $\{(X_n/\Gamma_n,\sfd_{X_n/\Gamma_n}, \Gamma_n\cdot \bar x_n)\}_{n \in N}$ converges in the pointed Gromov-Hausdorff sense to $(X_\infty/\Gamma_\infty,\sfd_{X_\infty/\Gamma_\infty}, \Gamma_\infty  \cdot \bar x_\infty)$.
\end{theorem}

%%%%%%%%%%%%%%%%%%%%%%%%%%%%%%%%%%%%%%%%
%%%%%%%%%%%%%%%%%%%%%%%%%%%%%%%%%%%%%%%%
%%%%%%%%%%%%%%%%%%%%%%%%%%%%%%%%%%%%%%%%

%%%%%%%%%%%%%%%%%%%%%%%%%%%%%%
%%%%%%%%%%%%%%%%%%%%%%%%%%%%%%

\subsection{Synthetic  Ricci curvature lower bounds}\label{ss-lowerRic}

We briefly recall here the definition of $\RCD^*$ spaces, and we refer to \cite{Sturm06I, Sturm06II, Lott-Villani09,  BS2010, AGS11, Gigli12, AGMR, EKS, AMS2013}   for more details about synthetic curvature-dimension conditions and calculus on metric measure spaces. There are different ways to define the curvature-dimension condition, that are now  known to be equivalent in the case of infinitesimally Hilbertian m.m.s. (see for example \cite[Theorem 7]{EKS}). We chose to give here only the definitions of the $\CD^*(K,N)$ condition and infinitesimally Hilbertian m.m.s., since this will be the framework of the paper. 
For $\kappa, s\in \R$, we introduce the generalised sine function 
$$\sin_{\kappa}(s)= \begin{cases}
\frac{\sin(\sqrt{\kappa}s)}{\sqrt{\kappa}} & \mbox{ if } \kappa>0 \\
s & \mbox{ if } \kappa=0 \\
\frac{\sinh(\sqrt{-\kappa}s)} {\sqrt{-\kappa}} & \mbox{ if } \kappa< 0. 
\end{cases}\quad 
$$
For $(t,\theta)\in [0,1]\times \R_{+}$ and $\kappa \in \R$,  the distortion coefficients are defined by
$$
\sigma_\kappa^{(t)}(\theta)=
\begin{cases}
\frac{\sin_\kappa(t\theta)}{\sin_\kappa(\theta)} & \mbox{ if }  \kappa \theta^2\neq 0 \mbox{ and } \kappa \theta^2< \pi^2 \\
t & \mbox{ if } \kappa \theta^2=0 \\
+ \infty & \mbox{ if } \kappa \theta^2 \geq \pi^2.
\end{cases}
\quad
$$For a metric space $(X,\sfd)$, let $\cP_2(X)$ be the space of Borel probability measures $\mu$ over $X$ with finite second moment, i.e. satisfying 
$$\int_X \sfd(x_0,x)^2 \,d\mu(x) < \infty,$$
for some (and thus, for every) $x_0 \in X$.
The \emph{$L^2$-Wasserstein distance} between $\mu_0, \mu_1 \in \cP_2(X)$ is defined by
\begin{equation}\label{eq:defW2}
W_2(\mu_0,\mu_1)^2=\inf_{{\rm q}} \int_{X \times X} \sfd(x,y)^2 d{\rm q}(x,y),
\end{equation}
where ${\rm q}$ is a Borel probability measure on $X \times X$ with marginals $\mu_0, \mu_1$. A measure ${\rm q}\in \cP(X^2)$ achieving the minimum in \eqref{eq:defW2} is called an \emph{optimal coupling}.
 The $L^2$-Wasserstein space $(\cP_2(X), W_2)$ is a complete and separable space, provided $(X,\sfd)$ is so. 
Let  $\cP_{2}(X,\sfd,\mm)\subset \cP_2(X)$ denote the subspace of $\mm$-absolutely continuous measures and  $\cP_{\infty}(X,\sfd,\mm)$ the set of measures in $\cP_2(X,\sfd,\mm)$ with bounded support. 

\begin{definition}\label{def:CD*KN}
Let $K \in \R$ and $N \in [1, \infty)$. A metric measure space $(X,\sfd,\mm)$ satisfies the curvature-dimension condition $\CD^*(K,N)$ if and only if for each $\mu_0,\mu_1 \in \cP_{\infty}(X,\sfd,\mm)$, with $\mu_i=\rho_i m$, $i=0,1$, there exists an optimal coupling $q$ and a $W_2$-geodesic $(\mu_t)_{t \in [0,1]}\subset \cP_{\infty}(X,\sfd,\mm)$ between $\mu_0$ and $\mu_1$ such that for all $t \in [0,1]$ and $N' \geq N$
\begin{equation}\label{eq:defCD*KN}
\int_X \rho_t^{-\frac{1}{N'}}d\mu_t \geq \int_{X \times X} \left(\sigma_{K/N'}^{(1-t)}(\sfd(x_0,x_1))\rho_0(x_0)^{-\frac{1}{N'}}+ \sigma_{K/N'}^{(t)}(\sfd(x_0,x_1))\rho_1(x_1)^{-\frac{1}{N'}}\right) d{\rm q}(x_0,x_1).
\end{equation}
\end{definition}

Given a metric measure space $(X,\sfd,\mm)$, the  Sobolev space $W^{1,2}(X,\sfd,\mm)$ is by definition  the space of $L^2(X,\mm)$ functions having finite Cheeger energy, and it is endowed with the natural norm  $\|f\|^2_{W^{1,2}}:=\|f\|^2_{L^2}+2 \Ch(f)$ which makes it a Banach space.   Here, 
 the Cheeger energy is given by the formula
$$\Ch(f):=\frac 1 2 \int_X |Df|_w^2 \, \d \mm,$$
where $|Df|_w$  denotes the weak upper differential of $f$. 
\\The metric measure space $(X,\sfd,\mm)$ is said to be \emph{inifinitesimally Hilbertian} if the Cheeger energy is a quadratic form (i.e. it satisfies the parallelogram identity) or, equivalently, if the Sobolev space $W^{1,2}(X,\sfd,\mm)$ is a Hilbert space.

\begin{definition}\label{def:RCD*}
Let $K \in \R$ and $N \in [1, \infty)$. We say that a metric measure space $(X,\sfd,\mm)$ is an $\RCD^*(K,N)$ space if it is infinitesimally Hilbertian and it satisfies the $\CD^*(K,N)$ condition. 
\end{definition}

\begin{remark}[The case $N=1$]\label{rem:N=1}
If $(X,\sfd,\mm)$ is a compact $\RCD^{*}(K,N)$ space with $N=1$, then by Kitabeppu-Lakzian \cite{KL} we know that $(X,\sfd,\mm)$ is isomorphic either to a point, or a segment, or a circle. Hence, all the statements of this paper will hold trivially. For instance:

\begin{itemize}
\item The revised first Betti number upper bound $\be(X)\leq 1$ holds trivially;
\item The torus stability holds trivially since  $\be(X)= 1$  only if $(X,\sfd,\mm)$  is isomorphic to a circle.
\end{itemize}
Without loss of generality, we will thus assume $N\in (1,\infty)$ throughout the paper to avoid trivial cases. 
\end{remark}

\begin{remark}[Other synthetic notions: $\CD(K,N)$, $\CD_{loc}(K,N)$, $\RCD(K,N)$]
For $K,N\in \R$, $N\geq 1$ one can consider the $\tau$-distortion coefficients
\begin{equation*}
\tau^{(t)}_{K,N}(\theta):= t^{1/N} \sigma^{(t)}_{K/(N-1)}(\theta)^{(N-1)/N}.
\end{equation*}
Replacing the $\sigma$-distortion coefficients with the $\tau$-distortion coefficients in \eqref{eq:defCD*KN}, one obtains the $\CD(K,N)$ condition. Since $\tau^{(t)}_{K,N}(\theta)\geq  \sigma^{(t)}_{K/N}(\theta)$, the $\CD(K,N)$ condition implies  $\CD^*(K,N)$. Conversely, the $\CD^*(K,N)$ condition  implies $\CD(K^*,N)$ for $K^*=K(N-1)/N$, see Proposition 2.5 (ii) in \cite{BS2010}. 
\\Analogously to Definition \ref{def:RCD*}, one can define the class of $\RCD(K,N)$ spaces as those $\CD(K,N)$ spaces which in addition are infinitesimally Hilbertian. It is clear from the above discussion that  $\RCD(K,N)$ implies $\RCD^*(K,N)$, and that $\RCD^*(K,N)$ implies $\RCD(K^*,N)$. An important property of $\RCD^*(K,N)$ spaces is the \emph{essential non-branching} \cite{RajalaSturm}, roughly stating that every $W_2$-geodesic with endpoints in $\cP_2(X,\sfd,\mm)$ is  concentrated on  a set of non-branching geodesics. This has been recently pushed to full non-branching in \cite{Deng}.
\\The local version of $\CD(K,N)$, called $\CD_{loc}(K,N)$, amounts to require that every point $x\in X$ admits a  neighbourhood $U(x)$  such that for
each pair $\mu_0, \mu_1\in \cP_{\infty}(X,\sfd,\mm)$ supported in $U(x)$ there exists a $W_2$-geodesic from $\mu_0$ to $\mu_1$ (not necessarily supported in $U(x)$) satisfying the $\CD(K,N)$ concavity condition. 
For essentially non-branching spaces, it is not hard to see that $\CD^*(K,N)$ is equivalent to $\CD_{loc}(K,N)$. It is much harder to establish the equivalence in turn with $\CD(K,N)$. This was proved for essentially non-branching spaces with finite total measure in \cite{CavallettiMilman}.  In particular it follows that, for spaces of finite total measure, the conditions $\RCD_{loc}(K,N)$, $\RCD(K,N)$ and $\RCD^*(K,N)$ are all equivalent.
\end{remark}

We state here some well-known properties of $\RCD^*(K,N)$ spaces that we are going to use throughout the paper. First of all, we have the following natural scaling properties:  if $(X,\sfd,\mm)$ is  an $\RCD^*(K,N)$ space, then
\begin{equation*}
\begin{split}
\text{for any $c>0$, $(X,\sfd,c\mm)$ is an $\RCD^*(K,N)$ space}, \\
\text{for any $\lambda>0$, $(X,\lambda \sfd,\mm)$ is an $\RCD^*(\lambda^{-2}K,N)$ space}.
\end{split}
\end{equation*}
\noindent

The following sharp Bishop-Gromov volume comparison was proved in \cite{Sturm06II} for $\CD(K,N)$ spaces, then generalised to non-branching $\CD_{loc}(K,N)$ spaces in \cite{CS12}, and to essentially non-branching $\CD_{loc}(K,N)$ spaces  in \cite{CMIMRN}. In particular it holds for  $\RCD^*(K,N)$ spaces. It will be useful in proving the appropriate upper bound for the revised first Betti number $\be(X)$.

\begin{theorem}[Bishop-Gromov volume comparison]\label{thm-BishopGromov}
Let  $K \in \mathbb R$ and $N\in(1,\infty)$. If $K <0$ then for any 
$\RCD^*(K,N)$ space  $(X,\sfd,\mm)$, all $x \in X$ and all $r\leq R$,
\begin{equation*}
\frac{\mm (B_r(x))}{\mm (B_R(x))} \geq \frac{\int_{0}^{r}\sinh^{N-1} (\sqrt{-K/(N-1)}t) \  dt}{\int_{0}^{R}\sinh^{N-1}(\sqrt{-K/(N-1)}t) \  dt} .
\end{equation*}
If $K \geq 0$ then 
\begin{equation*}
\frac{\mm (B_r(x))}{\mm (B_R(x))} \geq \left(\frac{r}{R}\right)^N .
\end{equation*}
\end{theorem}

\begin{remark}\label{rem:ProperGeod}
The Bishop-Gromov volume comparison implies that $\RCD^*(K,N)$ spaces are locally doubling and thus proper. It is also not hard to check directly from the Definition \ref{def:CD*KN} that $\supp \mm$  (and thus $X$, since we are assuming throughout that $X=\supp \mm$) is a length space. Since a proper length space is geodesic, we have that $\RCD^*(K,N)$ spaces are proper and geodesic. Thus, without loss of generality, we will assume that all the metric spaces in the paper are proper and geodesic.
\end{remark}

The set of $\RCD^*(K,N)$ spaces is compact when endowed with the pointed measured Gromov-Hausdorff topology (\cite{Lott-Villani09, Sturm06I, AGS11, GMS2013, EKS}):

\begin{theorem}[Stability w.r.t. pmGH convergence]\label{thm:stab}
Let $K\in \R$ and $N\in[1,\infty)$ and $C>1$. The set 
$$
\Big\{(X,\sfd,\mm, \bar{x}) \text{ p.m.m.s.  such that $(X,\sfd,\mm)$ is an  $\RCD^*(K,N)$ space and }  \, C^{-1}\leq \mm(B_1(\bar{x}))\leq C \Big\}
$$
endowed with the pmGH topology is compact.
\end{theorem}

Following the terminology of De Philippis-Gigli \cite{DPhG} (after Cheeger-Colding \cite{CC97}),  recall  that an $\RCD^*(K,N)$ space $(X,\sfd, \mm)$ is said
\begin{itemize}
\item  \emph{non-collapsed} if $\mm=\cH^N$ up to a positive constant;  
\item \emph{weakly non-collapsed} if $\mm\ll \cH^N$.
\end{itemize}
It follows from \cite[Theorem 1.12]{DPhG} that whenever $(X,\sfd,\mm)$ is a weakly non-collapsed $\RCD^{*}(K,N)$ space, $N$ is necessarily an integer. Honda \cite[Corollary 1.3]{H19} proved the following additional property of compact weakly non-collapsed spaces:

\begin{theorem}\label{thm-Honda}
Let $K \in \R$ and $N \in \N$. For any compact weakly non-collapsed  $\RCD^*(K,N)$ space $(X,\sfd,\mm)$, there exists $c>0$ such that $m = c \cH^N$.
\end{theorem}

\subsection{Almost Splitting}

We recall some results from \cite{MondinoNaber} that we will use in the proofs, starting from an Abresh-Gromoll inequality on the excess function.  For a metric measure space $(X,\sfd,\mm)$ we consider two points $p,q$ and define the excess function as:
\begin{equation*} 
e_{p,q}(x):=\sfd(p,x)+\sfd(x,q)-\sfd(p,q). 
\end{equation*}
For radii $0 < r_0 < r_1$, let $A_{r_0,r_1}(\{p,q\})$ be the annulus around $p$ and $q$: 
 $$A_{r_0,r_1}(\{p,q\})=  \{  x \in X \, |\,   r_0 < \sfd(p,x) < r_1 \, \lor  r_0 < \sfd(q,x) < r_1 \,\}.$$
We will use the following estimates, contained in  \cite[Theorem 3.7, Corollary 3.8 and Theorem 3.9]{MondinoNaber}.  

\begin{theorem}\label{thm-ExcessEstV2}
Let $(X,\sfd,\mm)$ be an $\RCD^*(K,N)$ space for some $K \in \R$, $N \in (1,\infty)$,  and let $p,q \in X$ be with $\sfd_{p,q}:=\sfd(p,q) \leq 1$. For any $\varepsilon_0 \in (0,1)$ there exists $\bar r = \bar r(K,N,\varepsilon_0)\in (0,1]$ such that if $x \in A_{\varepsilon_{0} \sfd_{p,q}, 2\sfd_{p,q}}(\{p,q\})$ satisfies $e_{p,q}(x)\leq r^2\, \sfd_{p,q}$ for some $ r\in (0, \bar r]$, then 
\begin{itemize}
\item[(i)] The following integral estimate holds: $$\displaystyle \fint_{B_{r\sfd_{p,q}}(x)}e_{p,q}(y)\, d\mm(y) \leq C(K,N,\varepsilon_0)r^2\, \sfd_{p,q}.$$
\item[(ii)] There exists $\alpha=\alpha(N)\in(0,1)$ such that 
\begin{equation}\label{AbGr2}
\sup_{y \in B_{r\sfd_{p,q}}(x)} e_{p,q}(y) \leq C(K,N,\varepsilon_0)r^{1+\alpha}\sfd_{p,q}.
\end{equation}
\item[(iii)] If moreover $x$ is such that the ball $B_{2r\sfd_{p,q}}(x)$ is contained in the annulus $A_{\varepsilon_0 \sfd_{p,q}, 2\sfd_{p,q}}(\{p,q\})$, then there exists $\alpha=\alpha(N)\in (0,1)$ such that 
\begin{equation}\label{gradientEst2}
\fint_{B_{r\sfd_{p,q}}(x)} |De_{p,q}|^2d\mm \leq C(K,N,\varepsilon_0)r^{1+\alpha}. 
\end{equation}
\end{itemize}
\end{theorem}

The almost splitting theorem for $\RCD^*$ spaces states that if there exist $k$ points in $(X,\sfd,\mm)$ that are far enough, and whose excess function and derivatives satisfy the appropriate smallness condition, then the space almost splits $k$ Euclidean factors, meaning that $(X,\sfd,\mm)$ is mGH-close to a product $\R^k \times Y$, for an appropriate $\RCD^*$ metric measure space $(Y,\sfd_Y,\mm_Y)$. More precisely, we follow the notation of \cite[Theorem 5.1]{MondinoNaber}, where $p_i+p_j$ denotes a point and $\sfd^p$ is the distance function $\sfd^p(\cdot)=\sfd(p,\cdot)$:

\begin{theorem}\label{thm-AlmSplitV2}
Let $\varepsilon > 0$, $N \in (1, \infty)$ and $\beta > 2$. 
Then there exists $\delta(\varepsilon,N)>0$ with the following property. Assume that, for some $\delta \leq \delta(\varepsilon,N)$, the following holds: 
\begin{itemize}
\item[(i)] $(X,  d, m)$ is an $\RCD^*(-\delta^{2\beta},N)$ space;
\item[(ii)] there exist points $x$, $\{ p_i, q_i, p_i + p_j\}_{1 \leq i<j \leq k}$ in $X$ for some $k \leq N$, such that 
$$\sfd(p_i,  x), \quad \sfd(q_i,x),  \quad \sfd(p_i + p_j,  x) \geq \delta^{-\beta}, \quad \mbox{for } 1 \leq i<j \leq k,$$
and for all $r \in [1, \delta^{-1}]$:
$$
\sum_{i=1}^{k} \sup_{B_{r}({x})} e_{p_{i},q_{i}}+ \sum_{i=1}^k \fint _{B_{r} ({x})} |D e_{p_i,q_i}|^2 \, d\mm +
\sum_{1\leq i< j\leq k} \fint_{B_{r} ({x} )}  \left|D \left( \frac{\sfd^{p_i}+\sfd^{p_j}} {\sqrt 2} -\sfd^{p_i+p_j}\right) \right|^2 \, d\mm   \leq \delta.
$$
\end{itemize}
Then there exists a p.m.m.s. $(Y,\sfd_Y, \mm_Y, {y})$ such that 
\begin{equation*}
\sfd_{mGH}(B^{X}_{\vare^{-1}}(x), B^{\R^k\times Y}_{\vare^{-1}}( (0^k, {y}))) < \vare.
\end{equation*}
More precisely:

\begin{itemize}
\item[1)] if $N-k<1$ then $Y=\{{y}\}$ is a singleton; 

\item[2)] if $N-k\in [1,+\infty)$ then  $(Y,\sfd_Y, \mm_Y)$ is an $\RCD^*(0,N-k)$-space,
there exist maps $u:  X\supset B_{\vare^{-1}}({x}) \to \R^k$ and $v:  X\supset B_{\vare^{-1}}({x})\to Y$, where $u^i=\sfd(p_i, \cdot)-\sfd(p_i,x)$,  such that the product map
 $$(u,v):X \supset B_{\vare^{-1}}({x})\to  \R^k \times Y \quad \text{is a $\vare$-mGH approximation on its image}.$$ 
\end{itemize}
\end{theorem}

Theorem \ref{thm-AlmSplitV2} was proved in \cite{MondinoNaber} by Naber and the second named author, building on top of Gigli's proof of the Splitting Theorem for $\RCD^{*}(0,N)$ spaces \cite{GigliSplitting}, after Cheeger-Gromoll Splitting Theorem \cite{ChGr} and Cheeger-Colding's Almost Splitting Theorem \cite{CC96}.

\subsection{Structure of $\RCD^*(K,N)$ spaces and rectifiability}\label{SS:StructRCD}
\bigskip
We collect here some known results about the structure of $\RCD^*(K,N)$ spaces,  which extended to the $\RCD^*(K,N)$ setting previous work on Ricci limit spaces \cite{Col97, CC97, CC00a, CC00b, CN}. They will be used in order to prove that for $\eps >0$ small enough, a compact $\RCD^*(-\eps,N)$ space $(X,\sfd,\mm)$ with $\be(X)=\lfloor N \rfloor$ and $\diam(X)=1$ is $\lfloor N \rfloor$-rectifiable and  the measure $\mm$ is absolutely continuous with respect to the Hausdorff measure $\mathcal{H}^{\lfloor N\rfloor}$.

We first recall the notion of $k$-rectifiability for metric and metric measure spaces. 

\begin{definition}[$k$-Rectifiability]\label{def:rectifiable}
Let $k \in \N$. A metric measure space $(X,\sfd,\mm)$ is said to be \emph{$(\mm,k)$-rectifiable as a metric space} if there exists a countable collection of Borel subsets $\{A_i\}_{i \in I}$ such that $\mm(X \setminus \bigcup_{i \in I}A_i) =0$ and there exist bi-Lipschitz maps between $A_i$ and Borel subsets of $\R^k$. A metric measure space $(X,\sfd,\mm)$ is said to be \emph{$k$-rectifiable as a metric measure space} if, additionally, the measure $\mm$ is absolutely continuous with respect to the Hausdorff measure $\mathcal{H}^k$.
\end{definition}

We next recall the definitions of tangent space and of $k$-regular set $\cR^k$. 

\begin{definition}\label{def-TgCones}
Let $(X,\sfd,\mm)$ be an $\RCD^*(K,N)$ space for $N \in (1, \infty)$ and $K \in \R$, and let $x \in X$. A metric measure space $(Y,\sfd_Y,\mm_Y, \bar y)$ is a tangent space of $(X,\sfd,\mm)$ at $x$ if there exists a sequence $r_i\in (0, +\infty)$, $r_i \downarrow 0$ such that $(X, r_i^{-1}\sfd, \mm_{r_i}^x, x)$ converges in the pmGH topology to $(Y,\sfd_Y,\mm_Y, \bar y)$, where
$$\mm_{r}^x=\left(\int_{B_r(x)} \left(1-\frac{\sfd(x,y)}{r}\right)d\mm(y) \right)^{-1}m.$$
The set of all tangent spaces of $(X,\sfd,\mm)$ at $x$ is denoted by $\mbox{Tan}(X,\sfd,\mm,x)$. 
\end{definition} 

\begin{definition}
Let $(X,\sfd,\mm)$ be an $\RCD^*(K,N)$ space for $N \in (1, \infty)$ and $K \in \R$. For any $k \in \N$, the $k$-th regular set $\cR_k$ is given by the set of points $x \in X$ such that tangent space at $x$ is unique and equal to the Euclidean space $(\R^k, \sfd_{\R^k}, c_k \cH^k,0^k)$, with
$$c_k=\left( \int_{B_1(0^k)} (1-|y|) \, d\cL^k(y)\right)^{-1}.$$
\end{definition}

In \cite[Theorem 1.1]{MondinoNaber}  it was proved that for any $\RCD^*(K,N)$ space $(X,\sfd,\mm)$, the  $k$-regular sets $\mathcal{R}_k$ for $k=1, \ldots, \lfloor N \rfloor$ are $(\mm,k)$-rectifiable as a metric spaces and  form an essential decomposition of $X$, i.e.  
$$\mm \Big( X \setminus \bigcup_{k=0}^{\lfloor N \rfloor} \mathcal{R}_k\Big)=0.$$

A subsequent refinement by the independent works \cite{MondinoKell, DPhMR, GigliPasqualetto}  showed  that the measure $\mm$ restricted to $\cR^k$ is absolutely continuous with respect to $\cH^k$. Moreover, in \cite{BrueSemola}, E.~Bru\`e and D.~Semola showed that there exists exactly one regular set $\mathcal{R}_k$ having positive measure. It is then possible to define the essential dimension of an $\RCD^*(K,N)$ space as follows. 

\begin{definition}[Essential dimension] Let $K\in \R, N \in(1, \infty)$ and let $(X,\sfd,\mm)$ be an $\RCD^*(K,N)$ space. The \emph{essential dimension} of $X$ is the unique integer $k \in  \{1,..., \lfloor N \rfloor\}$ such that $\mm(\mathcal{R}_k) > 0.$
\end{definition}

Observe that, as a consequence, any $\RCD^*(K,N)$ space of essential dimension equal to $k$ is $k$-rectifiable as a metric measure space. 

We finally state two theorems that will be used in the final part of the paper,  to show that an $\RCD^*(K,N)$ space with $\be(X)=N\in \N$ and $\diam(X)^2K \geq -\eps$ is mGH-close and bi-H\"older homeomorphic to a flat torus $\mathbb{T}^N$. 

\begin{theorem}[{\cite[Theorem 1.2]{DPhG}}]\label{thm-DPhG}
Let $N \in \N$, $N >1$ and let $(X_i, \sfd_i,\cH^N,x_i)$ be a sequence of non-collapsed $\RCD^*(K,N)$ spaces such that $(X_i,\sfd_i,x_i)$ converges to $(X,\sfd,x)$ in the pointed Gromov-Hausdorff sense.
Then one of the following holds.
\begin{itemize}
\item[(i)] If $\limsup_{i}\cH^N(B_1(x_i))>0$, then $\cH^N(B_1(x_i))$ converges to $\cH^N(B_1(x))$ and $(X_i,\sfd_i, \cH^N,x_i)$ converges in the pmGH sense to $(X,\sfd,\cH^N,x)$.
\item[(ii)] If $\lim_{i  \to \infty}\cH^N(B_1(x_i))=0$, then $\dim_{\cH}(X)\leq N-1$. 
\end{itemize} 
\end{theorem} 

In the following statement, we rephrase Theorem 1.10 of \cite{MondinoKapovitch}:

\begin{theorem}[{\cite[Theorem 1.10]{MondinoKapovitch}}]
\label{thm-MKbihold}
Let $(M,g)$ be a compact manifold of dimension $N$ (without boudary). There exists $\varepsilon =\varepsilon(M) >0$ such that the following holds. If $(X,\sfd,\mm,x)$ is a pointed $\RCD^*(K,N)$ space for some $K \in \R$ satisfying $\sfd_{pmGH}(X,M) < \varepsilon$, then $\mm=c_X  {\mathcal{H}^N}$ for some $c_X >0$ and $(X,\sfd)$ is  bi-H\"older homeomorphic to $M$.
\end{theorem}

\subsection{Covering spaces, universal cover and revised fundamental group}\label{SSS:CovSpacesT}

We first discuss the definition of covering spaces,  universal cover, revised fundamental group, 
and actions of groups of homeomorphisms over topological spaces. Then we focus on length metric measure spaces and see that the $\RCD^*(K,N)$ condition can be lifted to the total space of an $\RCD^*(K,N)$ base,  when having a covering map.

\subsubsection{Covering spaces}\label{SSS:CovSpacesT}
Let us provide some definitions and results related to coverings spaces from \cite{Spanier, Hatcher}. In particular, we state the notion of a group 
acting properly discontinuously as it appears in these references. Note that sometimes this is defined differently.

We say that a topological space $Y$  is a \emph{covering space} for a topological space $X$ if there exists a continuous map $p_{Y,X}: Y \to X$, called \emph{covering map}, with the property that for every point $x \in X$ there exists a neighbourhood $U\subset X$ of $x$ such that  $p_{Y,X}^{-1}(U)$ is  the disjoint union of open subsets of $Y$ and so that the restriction of $p_{Y,X}$ to each of these subsets is homeomorphic to $U$.  By definition,  the covering map is a local homeomorphism.  Two covering spaces $Y, Y'$ of $X$ are said to be equivalent if there exists an homeomorphism between them, $h: Y  \to Y'$, so that $p_{Y', X} \circ h =  p_{Y,X}$.

If $X$ is path-connected then the cardinality of $p_{Y,X}^{-1}(x)$ does not depend on $x \in X$.  We recall that given a topological space $Z$ and $z \in Z$,  the  \emph{fundamental group}  of $Z$, $\pi_1(Z,z)$, is the group of the equivalence classes under based homotopy of the set of closed curves from $[0,1]$ to  $Z$ with endpoints equal to $z$.  Any covering map $p_{Y,X}$ induces a monomorphism $p_{Y,X \sharp} : \pi_1(Y, y_0)  \to  \pi_1(X,p_{Y,X}(y_0))$; moreover,  when 
both $Y$ and $X$ are path-connected, the cardinality of $p_{Y,X}^{-1}(x)$ agrees with the index of  $p_{Y,X \sharp} (\pi_1(Y,y_0))$ in $\pi_1(X, p_{Y,X}(y_0))$.   For $Y$ path-connected, the covering map $p_{Y,X}$ is called \emph{regular} if  $p_{Y,X \sharp} (\pi_1(Y,y_0))$ is a normal subgroup of $\pi_1(X, p_{Y,X}(y_0))$. 

Before defining the group of deck transformations of a covering space, we introduce some terminology of group actions. 

\begin{definition}\label{defs-actions}
A group of homeomorphisms  $G$  of a topological space $Y$ is said to \emph{act effectively} or \emph{faithfully} if  $\bigcap_{y \in Y} \big\{  g  \,|\, g(y)=y \big\} = \{e\}$, where $e$ denotes the identity element of $G$. It acts 
\emph{without fixed points} or \emph{freely} if the only element of $G$ that fixes some point of $Y$ is the identity element.  We say that $G$ acts \emph{discontinuously} if the orbits of $G$ in $Y$ are discrete subsets of $Y$ and we say that $G$ acts \emph{properly discontinuously} if every $y \in Y$  has a neighbourhood $U \subset Y$ so that $U \cap gU = \emptyset$  for all $g \in G\setminus \{e\}$\footnote{This is sometimes defined differently, i.e. $G$ acts properly discontinuously if every $y \in Y$  has a neighbourhood $U \subset Y$ so that $U \cap gU \neq \emptyset$  for finitely many $g \in G$}.
\end{definition}
So,  acting properly discontinuously implies acting discontinuously and without fixed points, and every free action is effective.  

The group of deck transformations of a covering space $Y$ of $X$ is the group of self-equivalences of $Y$:
\begin{equation*}
G(Y\, | \,X) :  = \Big\{  h: Y \to Y  \,|\, h \text{ is an homeomorphism  and  }  p_{Y,X}\circ h=p_{Y,X}\Big\}.
\end{equation*}
By the unique lifting property,  $G(Y\, | \, X)$ acts without fixed points.  Combining this fact with the definition of covering map,
we see that  $G(Y\, | \, X)$ also acts properly discontinuously on $Y$.

If $Y$ is connected and locally path-connected, then $p_{Y,X}$ is regular if 
and only if the group $G(Y\, | \,X)$  acts transitively on each fibre of $p_{Y,X}$. In this case, 
for any $y_0 \in Y$ we have:
\begin{itemize}
\item An  isomorphism of groups: 
\begin{equation*}
G(Y\, | \,X)  \cong   \pi_1(X, p_{Y,X}(y_0)) \Big/ p_{Y,X \sharp}(\pi_1(Y,y_0));
\end{equation*}
\item A bijection between any fibre of $p_{Y,X}$  and $G(Y \, |\, X)$;
\item  A homeomorphism of spaces:
\begin{equation*}
X  \cong Y / G(Y\, | \,X). 
\end{equation*}
\end{itemize}

\begin{definition}[Universal cover of a connected space]
Given a connected topological space  $X$,  a \emph{universal covering space} $\widetilde{X}$ for  $X$ is a connected covering space for $X$ such that for any other connected covering space $Y$  of $X$ there exists a map $f:\widetilde{X}\to Y$ that forms a commutative triangle with the corresponding covering maps, 
i.e. $p_{Y,X}  \circ f =  p_{\widetilde X, X}$. 
\end{definition}
Since we do not require $X$ to be semi-locally simply connected, then $\widetilde X$ might not be simply connected. Thus, the group $G(\widetilde X\, | \,X)$ of deck transformations of $\widetilde{X}$ might not be isomorphic to the fundamental group  of $X$.  However, $G(\widetilde X\, | \,X)$ acts properly discontinuously on $\widetilde X$, transitively on each fibre of $p_{Y,X}$; thus $p_{\widetilde X,X}$ is regular.  Moreover,  any (connected) covering space of $X$  is covered by $\widetilde X$. In particular, universal covering spaces of a connected and locally path-connected space are equivalent.

Recall also  that for a connected topological space  $Y$ and a group $G$ of homeomorphisms of $Y$ acting properly discontinuously on $Y$, the projection map $Y  \to Y/G$ is a  regular covering whose group of deck transformations coincides with $G$, i.e.  $G(Y\, | \,Y/G)= G$.  
\\ We conclude this subsection summarising some results  that will be used later. 

\begin{proposition}\label{prop-X/Hcover}
Let $p:  Y  \to  X$  be a regular covering and let $H  \leq   G(Y\, | \,X)$. 
\begin{itemize}
\item If  $Y$ is connected, then the projection map $Y  \to  Y/H$  is a regular covering map and $G(Y\, | \,Y/H)=H$;
\item If $Y$ is path connected and locally path connected and $H$ is a normal subgroup of $G(Y\, | \,X)$, then the projection map $Y/H  \to X$ is a regular covering map and $G(Y/H \, | \,X)=  G(Y\, | \,X)/H$. 
\end{itemize}
 \end{proposition} 
 
 \begin{proof}
For a covering map, the group of deck transformations acts properly discontinuously on the total space.  Hence, $H$ also acts properly discontinuously on $Y$ and so the first item holds by the paragraph above this proposition.  The second item can be proved in a similar way:  first observe that $Y/H$ is connected because it is the image of the projection map which is continuous, then note that  $G(Y\, | \,X)/H$  acts properly discontinuously on $Y/H$ (see also Exercise 24 in \cite[Chapter 1, Section 1.3]{Hatcher}). 
 \end{proof}

\subsubsection{Coverings of metric spaces and $\RCD^*(K,N)$  spaces}\label{ss:LiftMMS}

We now discuss some definitions and results related to coverings of metric spaces. For more details we refer to \cite{SormaniWei} and \cite{MondinoWei}.

Let  $(X,\sfd_X)$ be a length metric space and $p_{Y,X}:Y  \to X$ be a covering map.  The length and metric structure of $X$ can be lifted to  $Y$ so that the covering map becomes a local isometry.  Explicitly,  denoting by $L_X$  the length structure of $X$,  define the metric $\sfd_Y : Y \times  Y  \to \R$ as 
\begin{equation}\label{eq:deftildedX}
\sfd_Y (y,  y'):= \inf  \Big\{ L_{X}(  p_{ Y, X} \circ  \gamma) \, \big| \, \gamma:[0,1]\to Y , \,
 p_{ Y, X} \circ  \gamma \text{ is Lipschitz and } \gamma(0)=y, \gamma(1)=y'    \Big\}.
 \end{equation}

This lifting process implies that $Y$ is complete whenever $X$ is so. 
In particular, if $X$ is compact,  then $Y$ will be a complete, locally compact length space, thus 
proper \cite[Proposition 2.5.22]{BuragoBuragoIvanov}.

If  $X$ is locally compact and $\mm_X$ is a Borel measure on it, we can lift $\mm_X$ to a Borel measure $\mm_Y$ on $Y$ that is locally isomorphic to $\mm_X$. In order to define  $\mm_Y$,   denote by $\cB(Y)$ the family of Borel subsets of $Y$ and consider the following collection of subsets of $Y$: 
$$
{\Sigma}:=\left\{ \, E   \subset  Y \,\big|\,  p_{Y, X}|_{E}:   E    \to    p_{ Y, X}(E) \text{ is an isometry} \right\}.
$$
Note that $\Sigma$  is stable under intersections and that $Y$ is locally compact given that $p_{Y,X}$ is a local isometry. 
Thus, the smallest $\sigma$-algebra that contains $\Sigma$ equals $\cB(Y)$. For $E \in \Sigma$,  define  $\mm_Y(E):=\mm_{X}\big(p_{Y,X}(E)\big)$ and then extend it to all $\mathcal B(Y)$.

From now on, all the covering spaces will be endowed with this metric and measure. The following result was proved in    \cite{MondinoWei}.

\begin{theorem}\label{thm-UniversalCover}
For any $K \in \mathbb R$  and any $N  \in (1, \infty)$, any $\RCD^*(K,N)$  space admits a universal cover space $(\widetilde X,  \sfd_{\widetilde X},  \mm_{\widetilde X})$
which is itself an $\RCD^*(K,N)$  space. 
\end{theorem}

We now state Sormani-Wei's definition of revised fundamental group \cite{SormaniWei}.

\begin{definition}(Revised fundamental group)
Given  a complete length metric space $(X,\sfd_{X})$ that admits a universal cover  $(\widetilde{X}, \sfd_{\widetilde{X}})$,  the revised fundamental group of $X$, denoted by $\bar\pi_1(X)$, is defined to be the  group of deck transformations $G(\widetilde{X}\, | \, X)$.  
\end{definition}

Recall that the covering map $p_{\widetilde X, X}$  associated to the universal cover space of $X$ is regular and thus $\bar\pi_1(X)$ acts transitively on each fibre of $p_{\widetilde X, X}$  and properly discontinuously on $\widetilde X$ by homeomorphisms; such homeomorphisms are measure-preserving isometries on $\widetilde{X}$, provided $\widetilde{X}$ is endowed with the lifted distance and measure of $X$, as described above.
\smallskip

We conclude this subsection by mentioning two properties that will be used later. First, for a covering map $p_{Y,X}: Y \to X$, one can prove (by lifting geodesics of $X$ to $Y$) that for any $x,x'\in X$  and $y \in Y$ with $y  \in p_{Y,X}^{-1}(x)$ there exists $y' \in p_{Y,X}^{-1}(x')$ such that $\sfd_Y(y, y')= \sfd_X(x, x')$. It follows that if $p_{Y,X}$ is regular, and thus $G(Y \, |\, X)$  acts transitively on its fibres, then for any $y,y'' \in Y$ there exists $h \in G(Y \, |\, X)$ such that 
\begin{equation}\label{eq-distlessDiam}
\sfd_Y( y, h(y''))  \leq \diam(X).
\end{equation}

The second property is that  a quotient space $Y/H$ as in Proposition \ref{prop-X/Hcover} is an  $\RCD^*(K,N)$ space provided either $X$ or $Y$ is an $\RCD^*(K,N)$ space.
We give more details below.  In Theorem \ref{lem-GalGro} we will use this fact to get an upper bound on the revised first Betti number of an $\RCD^*(K,N)$ by passing to a quotient space ($\widetilde X/H$  for $H = [ \bar \pi_1(X), \bar \pi_1(X)]$); this fact will be also useful in Lemma \ref{lem-torus}  to infer that the GH convergence of a sequence of quotient spaces can be promoted to mGH convergence.

\begin{lemma}\label{lem-RCDpasses}
Let $(X,\sfd_X,\mm_X)$ be a compact m.m.s with a regular covering map $p_{Y,X}:  Y \to X$. Assume that $(Y,\sfd_Y, \mm_Y)$ has the structure of m.m.s. so that $p_{Y,X}$ is a surjective local isomorphism of m.m.s.  and that $p_{Y,X}^{-1}(x)$ is at most countable.  Let  $K \in \R$ and $N\in (1,\infty)$. Then $(X,\sfd_X,\mm_X)$ is an $\RCD^*(K,N)$ space  if and only if $(Y,\sfd_Y, \mm_Y)$ is so. 
\end{lemma}

\begin{proof}
We argue along the lines of  \cite[Lemma 2.18]{MondinoWei}.  

Assume that $(X,\sfd_X,\mm_X)$ is an $\RCD^*(K,N)$ space.   Then $(X,\sfd_X)$  is  complete, separable,   proper  and  geodesic. 
Since $p_{Y,X}$  is a regular covering map, we can apply  \cite[Proposition 3.4.16]{BuragoBuragoIvanov}  stating that the length metrics on $X$  are in $1$-$1$  correspondence  with the $G( Y\, |\, X)$-invariant length metrics on $Y$; thus $(Y,\sfd_Y)$  is a length metric space. 
Since $p_{Y,X}$  is a local isometry, we automatically get that  $(Y,\sfd_Y)$ is a complete and locally compact space. Moreover, by our assumption on $p_{Y,X}^{-1}(x)$, 
$(Y,\sfd_Y)$  is separable.   Now every complete locally compact length space is geodesic \cite[Theorem 2.5.23]{BuragoBuragoIvanov}.  
Hence, $(Y,\sfd_Y)$  is a complete, separable and geodesic space.

In order to prove that $(Y,\sfd_Y,\mm_Y)$ is an $\RCD^*(K,N)$, first recall that by \cite[Theorem 3.17]{EKS} we know that $(Y,\sfd_Y,\mm_Y)$ is $\RCD^*(K,N)$ if and only if it is infinitesimally Hilbertian and it satisfies the strong $\text{CD}^{\text{e}}(K,N)$ condition, defined as in Definition 3.1 of  \cite{EKS}. Since $(X,\sfd_X,\mm_X)$ is an $\RCD^*(K,N)$ space, by  \cite[Theorem 3.17, Remark 3.18]{EKS} we infer that $(X,\sfd_X,\mm_X)$ satisfies the strong $\text{CD}^{\text{e}}(K,N)$ condition. Now  \cite[Theorem 3.14]{EKS} says that on a geodesic m.m.s. the strong $\text{CD}^{\text{e}}(K,N)$ condition is equivalent to the strong local $\text{CD}^{\text{e}}_{\text{loc}}(K,N)$ condition, thus in particular $(X,\sfd_X,\mm_X)$ satisfies the strong local $\text{CD}^{\text{e}}_{\text{loc}}(K,N)$ condition. Now each point $y \in Y$ has a compact neighbourhood $U_{y}$  such that $(U_y, \sfd_Y|_{U_{y}\times U_{y}}, \mm_Y\llcorner_{U_{y}})$ is isomorphic as metric measure space to $(p_{Y, X}(U_y),  \sfd_X|_{p_{Y, X}(U_y)\times p_{Y, X}(U_y)}, \mm_X\llcorner_{p_{Y, X}(U_y)})$.  It follows that the strong local $\text{CD}^{\text{e}}_{\text{loc}}(K,N)$ condition satisfied by $(X,\sfd_X,\mm_X)$ passes to the covering $(Y,\sfd_Y, \mm_Y)$. Since $Y$ is geodesic, then by  \cite[Theorem 3.14]{EKS} it also satisfies the strong $\text{CD}^{\text{e}}(K,N)$ condition.  

It remains to show that $(Y, \sfd_{Y}, \mm_{Y})$ is infinitesimally Hilbertian. This follows by a partition of unity on $Y$ made by Lipschitz functions with compact support contained in small metric balls isomorphic to metric balls in $X$, using the fact that the Cheeger energy is a local object (see \cite{AGS11, Gigli12}). 
Indeed, the validity of the parallelogram identity for the Cheeger energy on $Y$  can be checked locally (on each small ball) using a partition of unity. Since such small balls in $Y$ are isomorphic to small balls of $X$ where the Cheeger energy satisfies the parallelogram identity, we conclude that the Cheeger energy on $Y$ satisfies the parallelogram identity as well.  

Thus $(Y,\sfd_Y, \mm_Y)$ is infinitesimally Hilbertian, satisfies the $\text{CD}^{\text{e}}(K,N)$ condition and $\supp(\mm_Y)=Y$. It follows by \cite[Theorem 3.17]{EKS} that $(Y,\sfd_Y,\mm_Y)$ is an  $\RCD^*(K,N)$ space.

The converse implication can be proved with analogous arguments. 
\end{proof} 

\begin{proposition}\label{prop:RevFundGFinGen}
Let $(X,\sfd,\mm)$ be a compact $\RCD^*(K,N)$ space, for some $K \in \R$ and $N\in (1,\infty)$. Then the revised fundamental group $\bar{\pi}_1(X)$ is finitely generated. 
\end{proposition}

\begin{proof}
 By \cite[Proposition 6.4  and Lemma 6.2]{SormaniWei} and Bishop-Gromov volume comparison theorem, for 
 any compact $\RCD^*(K,N)$ space $(X,\sfd,\mm)$, its revised fundamental group $\bar \pi_1(X)$ can be generated by a set of cardinality 
at most $N( \delta_0, \diam(X)) < \infty$, where $\delta_0$  corresponds to the $\delta_0$-cover of $X$ so that $\widetilde X = X^{\delta_0}$
 and $N(\delta_0, \diam(X))$ is the maximal number of balls in $\widetilde X$ of radius $\delta_0$ in a ball of radius $\diam(X)$.
\end{proof}

\begin{corollary}\label{lem-barXRCD}
Let $(X,\sfd,\mm)$ be a compact $\RCD^*(K,N)$ space, for some $K \in \R$ and $N\in (1,\infty)$. Then for any normal subgroup $H$ of the revised fundamental group $\bar{\pi}_1(X)$, the metric measure space $(\widetilde X / H, \sfd_{\widetilde{X}/H}, \mm_{\widetilde X /H})$ is an $\RCD^*(K,N)$ space which is covered by $\widetilde X$ and covers $X$. 
\end{corollary}
 
 \begin{proof}
   Since from Proposition \ref{prop:RevFundGFinGen} we know that $\bar \pi_1(X)$ is finitely generated, then  it is at most countable. Thus the cardinality of each fibre of the covering map is at most countable.  We can thus conclude  using Proposition \ref{prop-X/Hcover} and Lemma \ref{lem-RCDpasses}. 
\end{proof}

\begin{remark}\label{rmrk-X/Hcover}
Since the group of deck transformations $G(Y \, | \, X)$ acts properly discontinuously on $Y$, the semi-metric 
\begin{equation}\label{eq-semi-metric}
\sfd_{Y/H}(Hy , Hy')=  \inf \Big\{ \,\sfd_Y(z,z')  \, \big|\,  z \in Hy , \, z' \in Hy'  \Big\}
\end{equation}   
defined on the quotient space $Y/H=\{Hy\, |\, y \in Y \}$ is actually a metric (this can be seen, for example, using Section \ref{ss-conv}). We also observe that, under the assumptions of Proposition \ref{prop-X/Hcover},  since $Y/H$ is a cover of $X$  it can  also be endowed with the lifted metric of $X$ as defined in 
\eqref{eq:deftildedX}.   
We note that this metric coincides with \eqref{eq-semi-metric}, so we will use the quotient metric of $Y/H$ whenever it is convenient. 
Notice that, in particular,  all the covering maps  appearing in Proposition \ref{prop-X/Hcover} are local isometries. 
\end{remark}

%%%%%%%%%%%%%%%%%%%%%%%%%%%%%%%%
%%%%%%%%%%%%%%%%%%%%%%%%%%%%%%%%
%%%%%%%%%%%%%%%%%%%%%%%%%%%%%%%%%

\section{Upper bound on the revised first Betti number:  $\be \leq \lfloor N \rfloor$}\label{s-betti}

In this section we obtain an upper bound for the revised first Betti number of an $\RCD^*(K,N)$ space with $K \leq 0$ and $N\in (1,\infty)$. In the case of smooth manifolds, the estimate is due to M.~Gromov \cite{Gromov81} and S.~Gallot \cite{Gallot} (compare also  \cite[Section 9.2]{Petersen}).  

We consider a compact geodesic space admitting a universal cover and define its revised first Betti number as the rank of the abelianisation of the revised fundamental group, whenever the abelianisation is finitely generated. Indeed, the fundamental theorem of finitely generated abelian groups states that for any finitely generated abelian group $G$ there exist a rank $s\in \N$, prime numbers $p_i$ and integers $s_i$ such that $G$ is isomorphic to $\mathbb Z^s \times \mathbb Z_{p_1}^{s_1} \times \cdots \times \mathbb Z_{p_l}^{s_l}$.

\begin{definition}
Let $(X,\sfd)$ be a compact geodesic space admitting a universal cover. Let $\bar{\pi}_1(X)$ be its revised fundamental group, set $H:=[\bar \pi_1(X), \bar \pi_1(X)]$ the commutator and  $\Gamma:=\bar{\pi}_1(X)/H$. Then we define the revised first Betti number of $X$ as 
\begin{equation*}
\be(X) : =\begin{cases}
s & \mbox{ if } \Gamma \mbox{ is finitely generated, } \Gamma = \mathbb{Z}^s \times \mathbb Z_{p_1}^{s_1} \times \cdots \times \mathbb Z_{p_l}^{s_l}, \\

\infty & \mbox{ otherwise.}
\end{cases}
\end{equation*}
\end{definition}
From now on, we denote by $\bar{X}$ the quotient space $\widetilde{X}/H$:
\begin{equation}\label{eq:defbarX}
(\bar X, \sfd_{\bar X}, \mm_{\bar X}):= (\widetilde X / H, \sfd_{\widetilde{X}/H}, \mm_{\widetilde{X}/H}).
\end{equation}
By Proposition \ref{prop-X/Hcover}, we know that $\bar X$ is a cover of $X$; moreover,  $\Gamma$ acts on $\bar X$ as an abelian group of isometries. Since the action of $\bar \pi_1(X)$ is properly discontinuous, the same is true for $\Gamma$. In particular, the action is discontinuous and  all the orbits $\Gamma x$, $x \in X$, are discrete. 

The first step in proving the upper bound on the revised first Betti number consists in showing the appropriate analog of a result of Gromov \cite[Lemma 5.19]{Gromov}. In the case of smooth manifolds, compare with \cite[Lemma 2.1, Section 9.2]{Petersen} for $k=1$ and for general $k \in \mathbb N$ with \cite[Lemma 3.1]{Col97}.

\begin{lemma}\label{lem-Gam'}
Let $(X,\sfd)$ be a compact geodesic space that admits a universal cover $(\widetilde X, \sfd_{\widetilde X})$, assume that  $\Gamma:=\bar{\pi}_1(X)/H$ is finitely generated and let $(\bar X, \sfd_{\bar X})$ be as in \eqref{eq:defbarX}.
Then for any $ k \in \mathbb N$ and $x \in \bar X$ there exists a finite index subgroup $\Gamma'=<\tilde\gamma_1,\dots,\tilde\gamma_{\be(X)} >$ of $\Gamma$ isomorphic to $\dZ^ {\be(X)}$ such that for any non trivial element $\tilde\gamma \in \Gamma'$
\begin{equation}\label{eq-lower}
k \diam(X) < \sfd_{\bar X}(\tilde\gamma(x),x)
\end{equation} 
and for all $i=1,...,\be(X)$ 
\begin{equation}\label{eq-upper}
\sfd_{\bar X}(\tilde\gamma_i(x),x) \leq 2k \diam(X).
\end{equation} 
\end{lemma}

\begin{proof}
We first find a subgroup $\Gamma''  \leq \Gamma$  of finite index and generated by elements 
that satisfy \eqref{eq-upper} for $k=1$.   For any $\varepsilon >0$, set $r_\eps=2\diam(X)+\eps$ and let $\Gamma_\varepsilon$ be the subgroup of $\Gamma$  generated by the set
$$\overline{\Gamma(r_\eps)}:=  \{ \gamma \in \Gamma \, | \, \sfd_{\bar X}(\gamma(x), x) \leq 2 \diam (X) + \varepsilon\}.$$
Observe that the previous set is not empty since, because of  \eqref{eq-distlessDiam}, there exists $\gamma \in \Gamma$ such that $\sfd_{\bar X}(\gamma(x),x)\leq \diam(X)$. 
Endow  $\bar X / \Gamma_\varepsilon$ with the quotient topology and the distance $\sfd_{\bar X/\Gamma_\varepsilon}$ induced by
$\sfd_{\bar X}$, c.f. \eqref{eq-quotdist}. Let
$\pi_\varepsilon :  \bar X  \to  \bar X / \Gamma_\varepsilon$ be the covering map.  

\textbf{Step 1}. We claim that $\bar X / \Gamma_\varepsilon \subset \overline{B}_{\diam(X)+\eps}(\pi_\eps(x))$, i.e. that for each $z \in \bar X$ it holds
\begin{equation}\label{eq:diamXGammaeps}
 \sfd_{\bar X / \Gamma_\eps}(\pi_\varepsilon(x), \pi_\varepsilon(z)) \leq \diam(X) + \varepsilon. 
 \end{equation}
By contradiction, assume that  there is $z \in \bar X$ such that 
$\sfd_{\bar X / \Gamma_\eps}(\pi_\varepsilon(x), \pi_\varepsilon(z)) > \diam(X) + \varepsilon$. 
Since  $\bar X / \Gamma_\varepsilon$ is a geodesic  space,  we can take a point in
the geodesic connecting  $\pi_\varepsilon(x)$ to $\pi_\varepsilon(z)$,
that we denote again by  $\pi_\varepsilon(z)$,  so that 
$\sfd_{\bar X / \Gamma_\eps}(\pi_\varepsilon(x), \pi_\varepsilon(z)) = \diam(X) + \varepsilon$.
Since the action of $\Gamma$ in $\bar X$ is discontinuous and hence the same is true for the action of $\Gamma_\vare$, there exist representatives $x,z \in \bar X$  that achieve $\sfd_{\bar X / \Gamma_\eps}(\pi_\varepsilon(x), \pi_\varepsilon(z))$, 
c.f. \eqref{eq-quotdist}, we have:
$$\sfd_{\bar X}(x, z)=\sfd_{\bar X / \Gamma_\eps}(\pi_\varepsilon(x), \pi_\varepsilon(z)) = \diam(X) + \varepsilon.$$
Now, there is $\gamma \in \Gamma$ such that $\sfd_{\bar X}(\gamma(x), z) \leq \diam(X)$, c.f. \eqref{eq-distlessDiam}.  Then, 
$$
\sfd_{\bar X}( \gamma (x),  x) \leq  \sfd_{\bar X}(\gamma( x), z)  + \sfd_{\bar X}(z, x) \leq 2\diam(X) + \varepsilon.
$$
This implies that $\gamma \in \Gamma_\varepsilon$.  Thus, $\pi_\varepsilon( \gamma(x))=\pi_\varepsilon(x)$ and
$$
0=\sfd_{\bar X / \Gamma_\eps}(\pi_\varepsilon(x), \pi_\varepsilon( \gamma(x))) \geq \sfd_{\bar X / \Gamma_\eps}(\pi_\varepsilon(x), \pi_\varepsilon(z))  - \sfd_{\bar X / \Gamma_\eps}(\pi_\varepsilon(z), \pi_\varepsilon( \gamma(x))) \geq \varepsilon,  
$$   
where we used that by definition of the quotient distance $\sfd_{\bar X/\Gamma_\eps}$ and the choice of $\gamma$ we have 
$$\sfd_{\bar X / \Gamma_\eps}(\pi_\varepsilon(z), \pi_\varepsilon( \gamma(x)))\leq \sfd_{\bar X}(z, \gamma(x)) \leq \diam(X).$$ 
This is a contradiction, and thus claim \eqref{eq:diamXGammaeps} is proved.

\textbf{Step 2}. Proof of  \eqref{eq-upper} in case $k=1$.
\\ From step 1 we know that $\bar X / \Gamma_\varepsilon \subset \overline{B}_{\diam(X)+\eps}(\pi_\eps(x))$. Since  $\bar X / \Gamma_\eps$ is proper, we infer that  $\bar X / \Gamma_\varepsilon$ is compact and thus the index of $\Gamma_\varepsilon$ in $\Gamma$ is finite. 

Since the action of $\Gamma$ in $\bar X$ is discontinuous, the set 
\begin{equation*}
\Gamma(3 \diam(X)):=  \{  \gamma  \in \Gamma \, |\,  \sfd(x, \gamma(x)) <  3 \diam(X)\}
\end{equation*}
is finite. Thus, there exists some $\eps_1 <\diam(X)$ such that for all $\eps \leq \eps_1$ the sets $\overline{\Gamma(r_\eps)}$
have bounded cardinality. Since their intersection is not empty, we get that  it coincides with some finite set $\overline{\Gamma(r_{\eps_0})}=\{\gamma_1, \ldots, \gamma_m\}$ for $\eps_0$ small enough, i.e.:
\begin{equation*}
\overline{\Gamma(r_{\eps_0})}= \bigcap_{\varepsilon>0}\overline{\Gamma(r_\eps)}=\overline{\Gamma(2\diam(X))}:= \{\gamma \in \Gamma | \ \sfd(x,\gamma(x))\leq 2\diam(X)\}.
\end{equation*}
Hence, for every element of $\overline{\Gamma(r_{\eps_0})}$ inequality (\ref{eq-upper}) holds with $k=1$.

\textbf{Step 3}. Conclusion by induction. 
\\We are going to select $\be$ elements of $\overline{\Gamma(r_{\eps_0})}$, say $\{\tilde\gamma_1, \ldots, \tilde\gamma_{\be}\}$, in such a way that the subgroup $\Gamma'$ generated by $\{\tilde\gamma_1, \ldots, \tilde\gamma_{\be}\}$ satisfies the conclusions of the lemma.    
First observe that, since the rank of $\Gamma_{\vare_0}$  equals $\be$,  by possibly discarding some elements we can choose a linearly independent subset  of $\overline{\Gamma(r_{\eps_0})}$  with cardinality $\be$. For simplicity, let us denote this subset by $\{\gamma_1, \ldots, \gamma_{\be}\}$. Consider $\Gamma''\subset \Gamma_{\eps_0}$ the subgroup generated by $\{\gamma_1, \ldots, \gamma_{\be}\}$. 
For fixed $k \in \N$, we are going to choose $\{\tilde\gamma_1, \ldots, \tilde\gamma_{\be}\}$ in $\Gamma'' \cap \overline{\Gamma(2\diam(X))}$ and such that both \eqref{eq-lower} and \eqref{eq-upper} hold. In order to do that, we proceed by induction on $j=1, \ldots, \be$ and we choose 
$\tilde{\gamma}_1, \ldots, \tilde{\gamma}_j$ in $\Gamma'' \cap \overline{\Gamma(2\diam(X))}$ such that:
\begin{itemize}
\item[$(a)$]  the subgroup $<\tilde\gamma_1, ... , \tilde\gamma_j>$ has finite index in $< \gamma_1, ... , \gamma_j>$;
\item[$(b)$]   $\tilde\gamma_j=   \gamma=\ell_{1j} \tilde\gamma_1 +   \cdots +\ell_{j-1,j}\tilde\gamma_{j-1}+ \ell_{jj}\gamma_j$ is chosen so that
\begin{equation*}
\begin{split}
l_{jj}  = \max & \left\{ |k|, \ k \in \mathbb{Z} \mbox{ s.t. } \exists \ \ell_{1j}, \ldots, \ell_{j-1,j} \in \mathbb{Z} 
\mbox{ s.t. if  } \gamma= \ell_{1j} \tilde\gamma_1 +   \cdots +\ell_{j-1,j}\tilde\gamma_{j-1}+ k \gamma_j \right. \\
& \left. \quad \text{then }\sfd_{\bar X}(\gamma(x),x) \leq 2k \diam(X)\right\}. 
\end{split}
\end{equation*}
\end{itemize}
Notice that condition $(b)$ ensures that $\Gamma'= <\tilde\gamma_1, \dots,\tilde\gamma_{\be}>$  satisfies (\ref{eq-upper}).  We now show that $\Gamma'$ also satisfies (\ref{eq-lower}). For any element $\gamma \in \Gamma'$  there exists  $j \in \{1, \ldots , \be\}$ such that $\gamma$ can be written as $\gamma= m_1\tilde\gamma_1 +  \cdots + m_j\tilde\gamma_j$ with $m_j \neq 0$. Assume by contradiction that  $\sfd(x,\gamma(x))\leq  k\diam(X)$ and consider 
\begin{align*}
2\gamma(x) = 2m_1\tilde\gamma_1 +  \cdots + 2m_j\tilde\gamma_j = \sum_{i=1}^{j-1} (2m_i+ 2m_{j} \ell_{ij})  \tilde\gamma_i  + 2m_j\ell_{jj} \gamma_j.
\end{align*}
Then by using the triangle inequality and (\ref{eq-upper}) we obtain that 
\begin{equation*}
\sfd_{\bar X}(x,2\gamma(x)) \leq \sfd_{\bar X}(x,\gamma(x)) + \sfd_{\bar X}(\gamma(x),2\gamma(x))  =  2\sfd_{\bar X}(x,\gamma(x))  \leq  2k\diam(X).
\end{equation*}
Then $2\gamma$ satisfies $\sfd_{\bar X}(2\gamma(x),x)\leq 2k \diam(X)$ and $|2m_j \ell_{jj}| > |\ell_{jj}|$, contradicting the way we chose $\ell_{jj}$ and $\tilde\gamma_j$ in $(b)$.  This concludes the proof. 
\end{proof}

\begin{remark} \label{rmk:modTorsion}
With a more careful analysis, a similar version of Lemma \ref{lem-Gam'} holds true if one replaces geodesic space by length spaces. c.f.   \cite{Gromov}. 
Furthermore, the same conclusion holds if we consider $\bar X/ T$ instead of $\bar X$, where $T$ is any torsion subgroup of $\Gamma$. 
 \end{remark}

\begin{remark}\label{rem-Xbar-noncompact}
Note that $\bar X$ is not compact.  Indeed, for any $\bar x \in \bar X$ and corresponding $\Gamma'$ given by Lemma  \ref{lem-Gam'}, the orbit $\Gamma' \bar x= \{\gamma(\bar x), \ | \, \gamma \in \Gamma'\}$ is countable, since $\Gamma'$ acts properly discontinuously on $\bar X$ and $\Gamma'$ is isomorphic to $\mathbb{Z}^{\be}$. 
Now, if by contradiction $\bar X$ is compact,
then $\Gamma' \bar x$ has a converging subsequence $\{\gamma_i(\bar x)\}$.  By using either that the action is properly discontinuous or property \eqref{eq-lower}, it is not difficult to show that $\{\gamma_i(\bar x)\}$ must be a constant sequence starting from $i$ large enough, giving a contradiction.
\end{remark}

\begin{remark}\label{rem-Gamma'closed}
It is not difficult to see that $\Gamma'$ is a  closed discrete group in the compact-open topology. Recall if a sequence of isometries $\gamma_i$ in $\Gamma'$ converges to $\gamma$ in the compact-open topology, then it converges uniformly on every compact subset and in particular for any $\bar x \in \bar X$ we have $\gamma_i(\bar x) \to \gamma(\bar x)$. We know that for any fixed $\bar x \in \bar X$, the only converging sequences in the orbit $\Gamma'\bar x$ are (definitely) constant sequences. Thus there exist $\gamma \in \Gamma'$ and $i_0 \in \N$ such that for all $i \geq i_0$ we have $\gamma_i(\bar x)=\gamma(\bar x)$. Therefore any converging sequence $\gamma_i$ in $\Gamma'$ is constantly equal to an element $\gamma$ of $\Gamma'$, yielding that $\Gamma'$ is closed and discrete.  
\end{remark}

In the following, we consider a compact $\RCD^*(K,N)$ space,  $(X,\sfd,\mm)$ with  $N\in (1,\infty)$. By  Theorem \ref{thm-UniversalCover}, we know that it admits a universal cover space,  $(\widetilde X, \widetilde\sfd, \widetilde\mm)$, 
that satisfies the $\RCD^*(K,N)$ condition. Using the same notation  as in Lemma  \ref{lem-Gam'}, by  Corollary \ref{lem-barXRCD}, the quotient m.m.s. $(\bar X, \bar \sfd,\bar \mm)$ is also an $\RCD^*(K,N)$ space. 
Since by Proposition \ref{prop:RevFundGFinGen} we know that the revised fundamental group $\bar{\pi}_1(X)$ is finitely generated, we infer that the revised first Betti number of $(X,\sfd,\mm)$ is finite. 

We are now ready to prove the first main result of the paper, namely the desired upper bound for $\be(X)$. This is done by combining Lemma \ref{lem-Gam'} with Theorem \ref{thm-BishopGromov} for $(\bar X, \bar \sfd,\bar \mm)$, and generalises to the non-smooth $\RCD$ setting the celebrated upper bound proved in the smooth setting by  M.~Gromov \cite{Gromov81} and S.~Gallot  \cite{Gallot} (see also  \cite[Theorem 2.2, Section 9.2]{Petersen} and \cite[Theorem 5.21]{Gromov}). \\

\noindent \textbf{Proof of Theorem \ref{lem-GalGro}}

Let $(X,\sfd,\mm)$ be a compact $\RCD^*(K,N)$ space with $K  \in \R$, $N \in [1,\infty)$, $\supp(\mm)=X$ and $\diam(X) \leq D$. 
If $N=1$, the claim holds trivially (see Remark \ref{rem:N=1});  thus, we can assume $N\in (1,\infty)$ without loss of generality.
By Theorem 1.2 in \cite{MondinoWei} if  $K > 0$ then $\bar \pi_1(X)$ is finite.  Hence,  $\be(X)= \rank (\Gamma)=0$.  Thus the claim holds trivially. 

  Assume that $K \leq 0$ and take $x \in \bar X=\widetilde X/H$. Recall that both $X$ and $\bar X$ are geodesic spaces (see Section 2), thus we can apply Lemma \ref{lem-Gam'} with $k=1$. Therefore, there exists a subgroup of deck transformations $\Gamma'= <\gamma_1, ..., \gamma_{\be}> \subset \Gamma$
such that
\begin{align}
\diam (X) < \sfd_{\bar X}(\gamma(x),x), &\quad \text{for all non trivial $\gamma \in \Gamma'$},  \label{eq-lower2} \\
 \sfd_{\bar X}(\gamma_i(x),x) \leq 2 \diam(X), &\quad \text{for all } i=1,\ldots, \be. \label{eq-upper2}
 \end{align}
By (\ref{eq-lower2}) all the  open balls $B_{\diam(X)/2}(\gamma(x))$, $\gamma \in \Gamma'$, are mutually disjoint.  Now set, 
\begin{equation*}
I_{r} = \Big\{ \gamma=l_1\gamma_1+ \cdots+l_{\be}\gamma_{\be} \in \Gamma' \,:  \,  |l_1|+\cdots+|l_{\be}|\leq r, \, \,\,l_i\in \mathbb Z \Big\}.
\end{equation*}
By (\ref{eq-upper2}),  since each element of $\Gamma$ is an isometry and applying the triangle inequality,  for all $\gamma \in I_r$ we have $\sfd_{\bar X}(\gamma(x), x) \leq 2r\diam(X)$. Hence, for all $\gamma \in I_r$
\begin{equation*}
B_{\diam(X)/2}(\gamma(x))   \subset B_{2r\diam(X) + \diam(X)/2}(x).
\end{equation*}
Since each element of $\Gamma$ preserves the measure, i.e. $\gamma_\sharp \mm_{\bar X}=\mm_{\bar X}$,  then all the balls $B_{\diam(X)/2}(\gamma(x))$ have the same $\mm_{\bar X}$-measure and thus
\begin{equation*}
|I_r| \, \mm_{\bar X}(B_{\diam(X)/2}(\gamma(x))) \leq   \mm_{\bar X} (B_{2r\diam(X) + \diam(X)/2}(x)).
\end{equation*}
By the definition of $I_r$,  $|I_r|$ is non decreasing with respect to $r$. 
If $r = 1$ then  $\{\gamma_i\}_{i=1}^{\be}  \subset I_r$ and thus
\begin{equation}\label{eq-b1leqI_r}
\be \leq |I_r|  \,\, \text{ for }\,\, r\geq 1.
\end{equation}
For arbitrary  $r \in \mathbb N$, it is easy to check that 
\begin{equation}\label{eq-cardinalityI_r}
|I_r| = (2r +1)^{\be}.
\end{equation}

Now we apply the relative volume comparison theorem, Theorem \ref{thm-BishopGromov}, to obtain an upper bound on the cardinality of $I_r$. 
Since the right hand sides of both equations in  Theorem \ref{thm-BishopGromov}  are non increasing as a function of $K$, we can assume 
that $K<0$. Thus, 
\begin{align}\label{eq-relcom}
|I_r| \leq &   \frac{  \mm_{\bar X} (B_{2r\diam(X) + \diam(X)/2}(x))}{\mm_{\bar X}(B_{\diam(X)/2}(\gamma(x)))}  \nonumber
\leq   \frac{\int_0^{2r\diam(X) + \diam(X)/2} \sinh^{N-1} (\sqrt{-K/(N-1)}s) \, ds} {\int_0^{\diam(X)/2} \sinh^{N-1} (\sqrt{-K/(N-1)}s) \, ds} \nonumber \\ 
=   &   \frac{\int_0^{(2r + 1/2)\sqrt{-K\diam(X)^2/(N-1)}} \sinh^{N-1} (s)\, ds}  {\int_0^{\sqrt{   -K\diam(X)^2/(N-1)}/2  }   \sinh^{N-1}(s) \, ds} \nonumber \\
= :  &   C_r(N, -K\diam(X)^2/(N-1)).
\end{align}
That is, $C_r(N, \cdot) : [0,  -KD^2/(N-1)]  \to  \R$  is the function given by
 \begin{equation*}
 C_r(N, t)= \frac{\int_0^{(2r + 1/2)\sqrt{t}} \sinh^{N-1} (s)\, ds} {\int_0^{\sqrt{t}/2} \sinh^{N-1}(s) \, ds}.
 \end{equation*}
  
By \eqref{eq-b1leqI_r} and since $C_r(N, t)$  is non decreasing as a function of $t$, we have $\be(X)  \leq C_r(N, -KD^2/(N-1))$. 
By using the Taylor expansion of $\sinh$, we calculate that 
$$\lim_{t\to0} C_r(N, t)= \left(\frac{(2r+1/2)}{1/2}\right)^{N}.$$
Thus, for small $t$ we have
\begin{equation}
\label{eq-asymptoticC_r(N,t)}
C_r(N, t)   <   2^N\left(2r+\frac 12\right)^{N}  +  \delta .
\end{equation} 
Now assume by contradiction that there exists a sequence $\eps_i\downarrow 0$  and $\RCD^*(K_i,N)$ metric measure spaces $(X_i,\sfd_i,\mm_i)$ such that 
$$-K_i\diam(X_i)^2 \leq \eps_i, \quad \be(X_i)>N.$$
Thanks to \eqref{eq-cardinalityI_r} and (\ref{eq-relcom}), we know that for any integer $r\geq 1$ we have
$$(2r+1)^{\be(X_i)} \leq C_r(N, -K_i\diam(X_i)^2/(N-1)).$$
For $\eps_i$ small enough, we can apply \eqref{eq-asymptoticC_r(N,t)}, so that for all $r \in \N, r\geq 1$
$$(2r+1)^{\be(X_i)} \leq C_r(N, \eps_i) < 2^N(2r+1/2)^{N} + \delta.$$
Thus for $r$ large enough
$$(2r+1)^{\be(X_i)} \leq 5^N r^N.$$
Now, if $\be(X_i)>N$, it is easily seen that the last inequality fails for  $r=r(N)$ large enough. Hence, we have shown that there exists $\eps(N)>0$ such that if $(X,\sfd,\mm)$ is an $\RCD^*(K,N)$ m.m.s. with $-K\diam(X)^2 \leq \eps(N)$, then $\be(X) \leq N$. Since by definition $\be(X)$ is integer, the last bound is actually equivalent to  $\be(X) \leq \lfloor N \rfloor$. This concludes the proof of the second assertion.
\\In order to prove the first assertion, set
$$C(N,t)=\sup\{ \ \be(X): \text{$(X,\sfd,\mm)$ is an $\RCD^*(K,N)$ m.m.s. with } -K\diam(X)^2=t \ \}$$
and observe that, thanks to \eqref{eq-relcom}, $C(N,t)$ is bounded by $C_r(N,t/(N-1))$. Since it is a bounded supremum of integer numbers, $C(N,t)$ is an integer. Moreover, the flat torus $\mathbb{T}^{\lfloor N \rfloor}$ is an $\RCD^*(0,N)$ space with $\be(\mathbb{T}^N)=\lfloor N \rfloor$, hence $C(N,t) \geq \lfloor N \rfloor$. The previous argument also shows that for $t \leq \eps(N)$, $\be(X) \leq N$, thus for any $t \leq \eps(N)$ we have $ \lfloor N \rfloor \leq C(N,t) \leq N$. This implies that $C(N,t)=\lfloor N \rfloor$ for any $t \leq \eps(N)$. As a consequence, $C(N,t)$ is the desired function tending to $\lfloor N \rfloor$ as $t\to 0$.
\hfill$\Box$

\begin{remark}
In the case of $n$-dimensional manifolds, Gallot proved an optimal bound for $\be(M)$ and expressed the function $C(n,t)$ as $\xi(n,t)n$, where $\xi(n,t)$ is an explicit function tending to one as $t$ tends to zero \cite[Section 3]{Gallot}. 
\end{remark}

\begin{corollary} \label{rmk:disjBalls}
Let $(X,\sfd,\mm)$ be a compact $\RCD^*(K,N)$ space with $N \in (1,\infty)$ and $\diam(X) = 1$.  
Let $(\bar X, \sfd_{\bar X},   \mm_{\bar X})$  be as in \eqref{eq:defbarX}.  
Then for any $\bar x  \in \bar X$, $k \in \N$ and $R >1$,  the open ball $B_{R}(\bar x)$ contains at least  $ \lfloor R  /k \rfloor ^{\be(X)}$ disjoint balls of radius $k/2$.
\end{corollary}

\begin{proof}
By Lemma \ref{lem-Gam'} there exists a subgroup of deck transformations $\Gamma'= <\gamma_1, ..., \gamma_{\be}> \subset \Gamma$
such that
\begin{align*}
k\diam (X) < \sfd_{\bar X}(\gamma(x),x), &\quad \text{for all non trivial $\gamma \in \Gamma'$}, \\
 \sfd_{\bar X}(\gamma_i(x),x) \leq 2k\,  \diam(X), &\quad \text{for all } i=1,\ldots, \be.
 \end{align*}
Arguing as in the proof of Theorem \ref{lem-GalGro} we get that, for any $r \in \mathbb N$, the number of disjoint balls of radius $k/2$ in $B_{k/2+ 2kr}(\bar{x})$ is larger than or equal to the number of elements in 
$$I_{r}= \left\{ \gamma \in \Gamma'  \, | \, \gamma=\ell_1 \gamma_1+\ldots +\ell_{\be(X)}\gamma_{\be(X)}, \quad |\ell_1|+\ldots +|\ell_{\be(X)}| < r \right\}.$$
The cardinality of $I_{r}$ equals $(2r+1)^{\be(X)}$.
Then for $R>1$ write  $\lfloor R \rfloor =k(2r+1)$ and get that $B_{R}(\bar x)$ contains at least  $(2r+1)^{\be(X)} = \lfloor R/k \rfloor^{\be(X)}$ disjoint balls of radius $k/2$.  
\end{proof}

%%%%%%%%%%%%%%%%%%%%%%%%%%%%%%%%%%%%%%%
%%%%%%%%%%%%%%%%%%%%%%%%%%%%%%%%%%%%%%%%%

\section{Construction of mGH-approximations in the Euclidean space}\label{s-GHapp}

This section is devoted to proving Theorem \ref{thm-GH}, which corresponds to the non-smooth $\RCD$ version of  \cite[Lemma 3.5]{Col97}. 
The main goal is to show that if $(X,\sfd,\mm)$ is an $\RCD^*(-\delta, N)$ space with $\delta=\delta(\eps,N)$ small enough, $\diam(X)=1$ and $\be(X)=\lfloor N \rfloor$, then the covering space  $(\bar X, \sfd_{\bar X},   \mm_{\bar X})$  (defined in \eqref{eq:defbarX}) is locally (on suitably large metric balls) mGH close to the Euclidean space $\R^{\lfloor N \rfloor}$.

The proof consists in applying inductively the almost splitting theorem. More precisely,  we show that $(\bar X, \sfd_{\bar X},   \mm_{\bar X})$  is locally (on suitably large metric balls) mGH close to a product $\R^k \times Y_k$ by induction on $k=1, \ldots ,\lfloor N \rfloor$. Since the diameter of the covering space $\bar X$ is infinite (see  Remark \ref{rem-Xbar-noncompact}),  the base case of induction $k=1$ will follow by carefully applying the almost splitting theorem. As for the inductive step, thanks to Corollary \ref{rmk:disjBalls} we will prove a diameter estimate on $Y_k$ that allows us to apply the almost splitting theorem on $Y_k$. We will conclude by deducing the almost splitting of an additional Euclidean factor by constructing an $\eps$-mGH approximation into $\R^{k+1}\times Y_{k+1}$. 

\begin{theorem}\label{thm-GH}
Fix $N \in (1, \infty)$ and $\beta > (2+\alpha) /\alpha$, where $\alpha=\alpha(N)$ is given by Theorem \ref{thm-ExcessEstV2}. For any $\varepsilon \in (0,1)$ 
there exists $ \delta(\varepsilon, N)>0$ such that the following holds. Let $(X,\sfd,\mm)$ be an $\RCD^*(-\delta^{2\beta},N)$ space with $\delta \in (0, \delta(\vare,N)]$, $\be(X)=\lfloor N \rfloor $ and $\diam(X)=1$, and let
$(\bar X, \sfd_{\bar X},   \mm_{\bar X})$  be the covering space as in \eqref{eq:defbarX}. Then there exists $\bar x \in \bar X $ such that
\begin{equation*}
\sfd_{mGH}(B^{\bar X}_{\vare^{-1}}(\bar x), B^{\R^{\lfloor N \rfloor}}_{\vare^{-1}}(0^{\lfloor N \rfloor})) \leq \vare.
\end{equation*}
\end{theorem}

\begin{remark}\label{rem:thm-GH} 
From Theorem \ref{thm-GH} it directly follows that  the point $\bar x \in \bar X$ has a $\lfloor N \rfloor$-dimensional Euclidean tangent cone and it  belongs to the $\lfloor N \rfloor$-regular set $\mathcal{R}_{\lfloor N \rfloor}$. 
\\Indeed, if $(X_i, \sfd_i, \mm_i)$ is a sequence of $\RCD^*(-\delta_i^{2\beta},N)$ spaces with $\delta_i \rightarrow 0$, $\be(X_i)= \lfloor N \rfloor$, $\diam(X_i)=1$ and $\bar x_i$ are as in Theorem \ref{thm-GH}, then the covering spaces  $(\bar X_i,  \sfd_{\bar X_i}, \mm_{\bar X_i}, \bar x_i)$ converge in pointed measured Gromov-Hausdorff sense to the Euclidean space $(\R^{\lfloor N \rfloor}, \sfd_{\R^{\lfloor N \rfloor}}, \mathcal{L}^{\lfloor N \rfloor}, 0^{\lfloor N \rfloor})$. 
The claim follows by applying this observation to a sequence of blow-ups of $\bar{X}$ centred at $\bar x$.
\end{remark}

In order to prove the base case of induction $k=1$, we start by showing the almost splitting of a line for $(\bar X, \sfd_{\bar X},   \mm_{\bar X})$.  This will be a direct consequence of the next proposition, which in turn will follow by combining Theorems \ref{thm-ExcessEstV2} and \ref{thm-AlmSplitV2} with suitable blow-up arguments.

\begin{proposition}\label{prop-CorAlmostSplitting}
Fix $N \in (1, \infty)$ and $\beta > (2+\alpha) /\alpha$, where $\alpha=\alpha(N)$ is given by Theorem \ref{thm-ExcessEstV2}. For any $\varepsilon >0$ there exists $\delta_1 = \delta_1(\varepsilon,  N)>0$, $\delta_1(\varepsilon, N) \rightarrow 0$ as $\vare$ goes to zero, such that for any $\delta \in (0,\delta_1]$ the following holds. Let $(X,\sfd,\mm)$ be an $\RCD^*(-\delta^{2\beta},N)$ m.m.s. such that
$$\mbox{\emph{diam}}(X,\sfd) \geq 2\delta^{-\beta}.$$
Then there exist $x_{\vare}\in X$, a pointed $\RCD^*(0, N-1)$ metric measure space $(Y_{\vare}, \sfd_{Y_{\vare}}, \mm_{Y_{\vare}}, y_{\vare})$ such that
\begin{equation*}
\sfd_{mGH}(B^{X}_{\vare^{-1}}(x_{\vare}), B^{\R \times Y_{\vare}}_{\vare^{-1}}(0, y_{\vare})) \leq \varepsilon.
\end{equation*}
\end{proposition}

\begin{proof} 
 Let $(X,\sfd,\mm)$ be an $\RCD^*(-\delta^{2\beta},N)$ space.  Because of the assumption on the diameter, there exist points $p,q \in X$ such that $\sfd(p,q)=2\delta^{-\beta}$. Define $x_{\vare}$ as the midpoint of a geodesic connecting $p$ and $q$. Consider the rescaled metric $ \sfd_{\delta}=(\delta^{\beta}/2)\, \sfd$. Since $X$ is an $\RCD^*(-\delta^{2\beta},N)$ space, $ (X,  \sfd_{\delta}, \mm)$ is an $\RCD^*(-4,N)$ space. {Observe that $\sfd_\delta(p,q)=1$. With respect to the metric $\sfd_{\delta}$, we have that  $x_{\varepsilon} \in A_{1/4, 2} (\{p,q\})$ }and $e^{\delta}_{p,q}(x_{\vare})=0$. 
 
 \textbf{Step 1}: Estimate on the $\sup$ of the excess.
\\ We can apply Theorem  \ref{thm-ExcessEstV2} and infer that there exist $\bar r=\bar r(N)>0, C=C(N)>0, \alpha = \alpha(N) \in (0,1)$ such that the estimate \eqref{AbGr2} centred at $x_{\eps}$ holds for $\sfd_{\delta}$.  By scaling back to the metric $\sfd$, such an estimate can be written as follows:  
\begin{equation}
\label{eq-est1-scale}
\sup_{y \in B_{ r }(x_{\vare})} e_{p,q}(y) \leq  C(N)\,   \delta^{\alpha \, 
\beta}\,   r^{1+\alpha}, \quad \text{for all } r \in (0, 2\delta^{-\beta} \bar r(N)].
\end{equation}
We aim to choose $\delta>0$ such that \eqref{eq-est1-scale} can be turned into the following: 
\begin{equation}
\label{eq-est1-step1}
\sup_{y \in B_r(x_\vare)} e_{p,q}(y) < \delta/2, \quad \text{for all } r \in [1, \delta^{-1}].
\end{equation}
Hence we first require
\begin{equation}
\label{cond1-step1}
\tag{a}
1 \leq   \delta^{-1}  \leq  2 \delta^{-\beta} \bar r(N),
\end{equation}
so that \eqref{eq-est1-scale} applies to all radii $r\in [1, \delta^{-1}]$. Secondly, we need
$$ C(N) \, \delta^{\alpha\, \beta}  (\delta^{-1})^{1+\alpha} < \delta/2,$$
so the right-hand side of \eqref{eq-est1-scale}  is bounded above by $\delta/2$. That means
\begin{equation}
\label{cond2-step1}
\tag{b}
\delta ^{\beta \alpha - 2-\alpha} < 1/(2C(N)).
\end{equation}

Such a choice is possible since the assumption $\beta > (2+\alpha)/\alpha$ ensures that the exponent on the left-hand side is strictly positive.
By choosing $\delta>0$ sufficiently small so  that both conditions \eqref{cond1-step1} and \eqref{cond2-step1} are satisfied, we obtain from \eqref{eq-est1-scale} that  estimate \eqref{eq-est1-step1} holds.

 \textbf{Step 2}: $L^{2}$-estimate on the gradient of the excess.
 \\Consider again the rescaled metric $ \sfd_{\delta}=(\delta^{\beta}/2)\, \sfd$ and choose $r>0$ so that  $ r \leq \min \left\{ \bar  r(N), 1/4 \right\}$.
Then $B^{X}_{2r} (x_{\vare}) \subset   A_{1/4,  2} (\{p,q\})$ and  estimate \eqref{gradientEst2} of Theorem \ref{thm-ExcessEstV2} holds as well.  By scaling back to the metric $\sfd$, we have:
\begin{equation}
\label{eq-est2-scale}
\fint_{B^{X}_{r}(x_{\vare})}|De_{p,q}|^2 d\mm    \leq  C(N)  \, \delta^{\beta (1+\alpha)}   r^{1+\alpha}, \quad \text{for all } r \leq 2 \delta^{-\beta} \min\left\{  \bar r(N), 1/4\right\}.
\end{equation}
As in step 1, we  aim to choose $\delta>0$ so that \eqref{eq-est2-scale} implies the following:  
\begin{equation}
\label{eq-est2-step1}
\fint_{B^{X}_{r }(x_{\vare})}|De_{p,q}|^2d\mm  \leq \delta/2, \quad \text{for all } r\in[1, \delta^{-1}].
\end{equation}
Hence we require
\begin{equation}
\label{cond3-step1}
\tag{c}
1 \leq   \delta^{-1}  \leq  2 \delta^{-\beta} \min \left\{  \bar  r(N), 1/4 \right\}. 
\end{equation}
In order for the right-hand side of \eqref{eq-est2-scale} to be less than or equal to $\delta/2$ we need 
$$ C(N) \delta^{\beta (1+\alpha)} \delta^{-(1+\alpha)}< \delta/2,$$
that is, 
\begin{equation}
\label{cond4-step1}
\tag{d}
\delta^{(1+\alpha)(\beta -1) -1} < 1/(2C(N)).
\end{equation}
Note that since $\beta >2$, the exponent on the left-hand side is strictly positive. 

 \textbf{Step 3}: Conclusion.  
\\Fix $\delta_0=\delta_0(N)>0$  satisfying inequalities \eqref{cond1-step1}, \eqref{cond2-step1}, \eqref{cond3-step1} and \eqref{cond4-step1}. 
Then for all  $\delta \in (0,\delta_0(N)]$  inequalities \eqref{cond1-step1}, \eqref{cond2-step1}, \eqref{cond3-step1} and \eqref{cond4-step1} hold as well. 
Now, for any $\varepsilon >0$ let $\delta(\vare, N)>0$ be as in Theorem \ref{thm-AlmSplitV2}. 
We define 
$$ \delta_1=\delta_1(\vare, N)=\min\{\delta_0(N), \delta(\vare,N)\}.$$
Then, for all $\delta \in (0,\delta_1]$ inequalities \eqref{eq-est1-step1} and \eqref{eq-est2-step1} hold and we have
$$ \sup_{y \in B^{X}_r(x_{\vare})} e_{p,q}(y) + \fint_{B^{X}_{r }(x_{\vare})}|De_{p,q}|^2 d\mm    \leq \delta, \quad \text{for all } r \in [1, \delta^{-1}].$$
Since $\delta \leq \delta_1 \leq \delta(\vare,N)$, we can apply Theorem \ref{thm-AlmSplitV2} and conclude the proof.
\end{proof}

Proposition \ref{prop-CorAlmostSplitting} can be in particular applied to the covering space $(\bar X, \sfd_{\bar X}, \mm_{\bar X})$. Indeed, thanks to Corollary \ref{lem-barXRCD} it is an $\RCD^*$ space and it is not compact (thus it must have infinite diameter, since it is proper), as it was pointed out in Remark \ref{rem-Xbar-noncompact}.
 This gives the base case of induction, $k=1$. 
 
\begin{corollary}\label{cor-step1}
Let $(X,\sfd,\mm)$ and  $(\bar X, \sfd_{\bar X}, \mm_{\bar X})$  be as in Theorem \ref{thm-GH}.
Then there exist $\bar x_{1, \vare}\in \bar X $ 
and a pointed $\RCD^*(0,N-1)$ space $(Y_{1, \vare}, \sfd_{Y_{1, \vare}}, m_{Y_{1, \vare}}, y_{1, \vare})$ such that
\begin{equation*}
\sfd_{mGH}(B^{\bar X}_{ \vare^{-1}}(\bar x_{1, \vare}), B^{\R \times Y_{1, \vare}}_{ \vare^{-1}}(0, y_{1, \vare})) \leq \varepsilon.
\end{equation*}
\end{corollary}   

Observe that for the base case of induction (i.e. in Corollary \ref{cor-step1}),  we did not use the assumptions on the diameter and revised first Betti number. These assumptions will play a key role in the following, instead. 
Let us state the induction hypothesis.  

\medskip

\noindent {\bf Assumption $\bf A_k$}: Fix $N\in (1,\infty)$ and let $k\in \N$ with $k< \lfloor N \rfloor$. For all $\eta \in (0,1)$  there exists $\delta_k=\delta_k(\eta, N)>0$ such that for all $\delta \in (0, \delta_k]$, the following holds:  if $(X,\sfd,\mm)$ and  $(\bar X, \sfd_{\bar X}, \mm_{\bar X})$  are as in Theorem \ref{thm-GH}, then there exists $\bar x_{k, \eta} \in \bar X$ and a pointed $\RCD^*(0,N-k)$ space $(Y_{k, \eta},\sfd_{Y_{k, \eta}},m_{Y_{k, \eta}}, y_{k, \eta})$ such that
\begin{equation*}
\label{inductionSplit}
\sfd_{mGH}(B^{\bar X}_{ \eta^{-1}}(\bar x_{k,  \eta}), B^{\R^k \times Y_{k, \eta}}_{ \eta^{-1}}(0^k, y_{k, \eta})) \leq \eta.
\end{equation*}

In order to prove $\bf A_{k+1}$ given $\bf A_{k}$ we aim to apply Proposition \ref{prop-CorAlmostSplitting} to the space $Y_{k,\eta}$:  in this way, $Y_{k,\eta}$ will almost split a line, thus $\bar X$ will almost split an additional Euclidean factor, yielding $\bf A_{k+1}$. To this aim, the following diameter estimate will be key.

\begin{lemma}\label{lem-diamEst}
Assume that $\bf A_{k}$ holds. For any $\eta \in (0,1)$, let $\delta_k(\eta, N)>0$, $(X,\sfd,\mm)$ and  $(\bar X, \sfd_{\bar X}, \mm_{\bar X})$ be as in $\bf A_{k}$  and $(Y_{k, \eta},\sfd_{Y_{k, \eta}},m_{Y_{k, \eta}}, y_{k, \eta})$ be the corresponding $\RCD^*(0, N-k)$ p.m.m.s..
\\Then there exist $c_{N}\in (0,1)$ and $\eta_0(N)>0$ such that for all $\eta \in (0,\eta_0(N)]$ and for all $\delta \in (0, \delta_k(\eta, N)]$, it holds:
\begin{equation*}
\mbox{\emph{diam}}(B^{Y_{k, \eta}}_{ \eta^{-1}}(y_{k, \eta}))>  c_{N} \eta^{-1}. 
\end{equation*}
\end{lemma}

\begin{proof} 
We argue by contradiction. Assume there exist a sequence $\eta_i \downarrow 0$, corresponding $\delta_i=\delta_k(\eta_i, N) \rightarrow 0$ and pointed $\RCD^*(-\delta_i^{2\beta}, N)$ spaces $(\bar X_i,\sfd_{\bar X_{i}}, \mm_{\bar X_{i}}, \bar x_i)$ for which there exists pointed $\RCD^*(0,N-k)$ spaces $(Y_i, \sfd_{Y_i}, \mm_{Y_i},y_i)$ such that
\begin{equation}
\label{eq:Claim12}
\sfd_{mGH}(B^{\bar X_{i}}_{\eta_{i}^{-1}}(\bar{x_i}), B_{\eta_i^{-1}}^{\R^k\times Y_i}((0^k,y_i))) \leq \eta_i  \quad \mbox{ and } \lim_{i \rightarrow +\infty}   {\eta_i}  
 \diam(B^{Y_{i}}_{\eta_i^{-1}}( y_i))=0.
\end{equation}
Let $i$ be sufficiently large so that $\eta_i < 1$. By Corollary \ref{rmk:disjBalls} we know that $B^{\bar X_{i}}_{\eta_i^{-1}}(\bar{x}_i)$ contains at least $(  \lfloor \eta_i^{-1} \rfloor)^b$ disjoint balls of radius $1/2$, at positive mutual distance. Using \eqref{eq:Claim12} we infer that,  for $i$ large enough, the ball $B_{\eta_i^{-1}}^{\R^k\times Y_i}((0^k,y_i))$ in $\R^k \times Y_i$ also contains at least $(  \lfloor  \eta_i^{-1} \rfloor)^b$ disjoint balls of radius $1/2$.

Rescale the metric of $\mathbb R^k \times Y_i$ by a factor ${\eta_i}$ and denote the resulting space as $(\mathbb R^k \times Y_i)^{{\eta_i }}$. Then for  large enough $i$ the ball of radius 1 in $(\R^k \times Y_i)^{{\eta_i}}$ centred at $(0^k,y_i)$ contains at least $(\  \lfloor  \eta_i^{-1} \rfloor )^b$ disjoint balls of radius $ \tfrac {\eta_i} 2$ at positive mutual distance. Furthermore, since $ {\eta_i}  \diam(B^{Y_{i}}_{\eta_i^{-1}}(y_i))$ tends to zero as $i$ tends to infinity, when taking the Gromov-Hausdorff limit of such balls we obtain: 
\begin{equation*}
\lim_{i  \to \infty}\sfd_{GH}(B_1^ {(\R^k  \times Y_i)^{{\eta_i}}}(0^k,y_i), B_1^ {(\R^k)^{ {\eta_i}}}   (0))=0.
\end{equation*}
As a consequence, for large enough $i$,  $B_1^ {(\R^k)^{\eta_i}}(0)$ contains at least $(  \lfloor  \eta_i^{-1} \rfloor )^b$ disjoint balls of radius $\tfrac {\eta_i} 2$. Denote by $\omega_k$ the volume of $B_1^{\R^k}(0)$. Since we only rescaled the metric of $\mathbb R^k \times Y_i$ by a factor ${\eta_i}$, the mass of $B_1^{(\R^k)^{\eta_i }}(0)$ equals $\omega_k (\eta_i^{-1})^{k}$ and the mass of a ball of radius $\tfrac {\eta_i} 2$ in  $B_1^{(\R^k)^{\eta_i }}(0)$ equals $\omega_k (  1/2)^{k}$.  Hence, 
\begin{equation}\label{eq:contrlemdiamAk}
\omega_k  (\lfloor  \eta_i^{-1} \rfloor)^b (1/2)^{k} \leq \omega_k (\eta_i^{-1})^{k}. 
\end{equation}
However, since $1 \leq k < \lfloor N\rfloor=b$ and $\eta_i \to 0$, the estimate \eqref{eq:contrlemdiamAk} cannot hold for $i$ sufficiently large. 
\end{proof}

\begin{remark}
Notice that we used the hypotheses $\diam(X)=1$ and $b:=\be(X)=\lfloor N\rfloor$ to have a given number of disjoint balls of radius $1/2$ in a ball in $\bar X$ of radius larger than 1.
\end{remark}

We next combine Lemma \ref{lem-diamEst} and Proposition \ref{prop-CorAlmostSplitting} in order to prove that the space $Y_{k,\eta}$ almost splits a line, for $\eta>0$ small enough depending on $\eps>0$.

\begin{proposition}\label{prop-AlmostSplittingY}
Assume that $\bf A_{k}$ is satisfied. Then for any $\vare \in (0,1)$ there exists $\eta(\vare,N)>0$ such that the following holds. For any $\eta \in (0, \eta(\vare,N)]$, let  $(Y_{k, \eta},\sfd_{Y_{k, \eta}},\mm_{Y_{k, \eta}}, y_{k, \eta})$  be the pointed m.m.s. given by $\bf A_k$. Then there exist $y\in B^{Y_{k, \eta}}_{\eta^{-1}/2}(y_{k,\eta})$ and a pointed $\RCD^*(0, N-k-1)$ space $(Y', \sfd_{Y'}, \mm_{Y'}, y')$ such that 
\begin{equation*}
B^{Y_{k, \eta}}_{\vare^{-1}}(y) \subseteq B^{Y_{k, \eta}}_{\eta^{-1}}(y_{k,\eta}) \,  \text{  and  }  \, \sfd_{mGH}(B^{Y_{k, \eta}}_{\vare^{-1}}(y), B^{\R \times Y'}_{\vare^{-1}}(0,y'))\leq \vare.
\end{equation*}
\end{proposition}

\begin{proof}
Define \begin{equation*}
\eta(\vare,N) = \min\left\{\frac{\vare}{2}, \eta_0(N), \frac{c_N}{2}\delta_1(\vare,N)^{\beta}\right\},
\end{equation*}
where $\delta_1(\vare,N)>0$ is the quantity given by Proposition \ref{prop-CorAlmostSplitting} and $c_N, \eta_0(N)$ are defined in Lemma \ref{lem-diamEst}. Then by assumption $\bf A_k$ and Lemma \ref{lem-diamEst}, for any $\eta \in(0, \eta(\vare,N)]$ and for all $\delta \in (0, \delta_k(\eta,N)]$, if $(X,\sfd,\mm)$ is an $\RCD^*(-\delta^{2\beta},N)$ space as in assumption $\bf A_k$, then there exist $\bar{x}_{k,\eta} \in \bar X$ and a pointed $\RCD^*(0,N-k)$ space $(Y_{k,\eta}, \sfd_{Y_{k,\eta}}, \mm_{Y_{k,\eta}}, y_{k,\eta})$ such that 
\begin{align}
& \sfd_{mGH}(B^{\bar X}_{\eta^{-1}}(\bar x_{k,\eta}), B^{\R^k \times Y_{k,\eta}}_{\eta^{-1}}(0^k,y_{k,\eta})) \leq \eta \nonumber \\
& \diam (B^{Y_{k, \eta}}_{\eta^{-1}}(y_{k,\eta}))> c_{N} \, \eta^{-1}.  \label{eq:diamY_keta}
\end{align}
Let $\xi>0$ be such that $c_N\, \eta^{-1}=2\xi^{-\beta}$. Our choice of $\eta(\vare,N)$ ensures that for any $\eta \in (0,\eta(\vare,N)]$ we have $\xi \in (0, \delta_1(\vare,N)]$. Therefore we can apply Proposition \ref{prop-CorAlmostSplitting}  to $Y_{k,\eta}$ and get that there exist $y \in Y_{k,\eta}$ and a pointed $\RCD^*(0, N-k-1)$ space $(Y', \sfd_{Y'}, \mm_{Y'}, y')$ such that 
\begin{equation*}
\sfd_{mGH}(B^{Y_{k, \eta}}_{\vare^{-1}}(y), B^{\R \times Y'}_{\vare^{-1}}(0,y'))\leq \vare.
\end{equation*}
It remains to show that $y\in B^{Y_{k, \eta}}_{\eta^{-1}/2}(y_{k,\eta})$ and that $B^{Y_{k, \eta}}_{\vare^{-1}}(y) \subseteq B^{Y_{k, \eta}}_{\eta^{-1}}(y_{k,\eta})$. 
\\From the proof of Proposition \ref{prop-CorAlmostSplitting},  we know that $y$ is a midpoint of a geodesic between two points $p,q$ at distance equal to $c_{N} \eta^{-1}$. Since $Y_{k,\eta}$ is a geodesic space and $c_{N}\in (0,1)$, it is easily seen that \eqref{eq:diamY_keta} implies that there exists a point  $q \in B^{Y_{k, \eta}}_{\eta^{-1}}(y_{k,\eta})$ such that $\sfd_{Y_{k,\eta}}(q, y_{k,\eta}) = c_N \eta^{-1}$. 
Then, in the proof of Proposition \ref{prop-CorAlmostSplitting}  we can chose $p=y_{k,\eta}$, $q \in Y_{k,\eta}$ with $\sfd_{Y_{k,\eta}}(q, y_{k,\eta}) = c_N \eta^{-1}$  and $y$ a midpoint of a geodesic between $p$ and $q$. Therefore $\sfd_{Y_{k,\eta}}(y, y_{k,\eta}) = c_N \eta^{-1}/2  $. Now, for any point $z \in B^{Y_{k, \eta}}_{\vare^{-1}}(y)$ we have  
\begin{equation*}
\sfd_{Y_{k, \eta}}(z,y_{k,\eta}) \leq \sfd_{Y_{k, \eta}}(z,y) + \sfd_{Y_{k, \eta}}(y,y_{k,\eta}) < \vare^{-1}+ c_N \eta^{-1}/2.
\end{equation*}
Moreover, our choice of $\eta \leq \eta(\vare,N) \leq \vare/2$ ensures that $\vare^{-1} \leq \eta^{-1}/2$. Therefore, for any $z \in B^{Y_{k, \eta}}_{\vare^{-1}}(y)$, we have $\sfd_{Y_{k, \eta}}(z,y_{k,\eta}) < \eta^{-1}$, as desired.  
\end{proof}

We are now in position to prove Theorem \ref{thm-GH}. 

\begin{proof}[ \bf{Proof of Theorem \ref{thm-GH} }]
We proceed by induction. For $k=1$,  ${\bf A_1}$ follows from Corollary \ref{cor-step1}.
Now assume that ${\bf A_k}$ holds for some $1 \leq k <\lfloor N \rfloor$ and let us show ${\bf A_{k+1}}$. Denote by $C_1, C_{2}>0$ the constants appearing in Propositions \ref{prop-restrmGH} and \ref{prop-prodmGH} respectively and define $C:=\max\{C_1C_2, 2C_1\}$.  Fix $\vare\in (0,1)$ and let
$$
\vare_1 := \min\{ 1/2,  1/C \} \, \vare, \quad \eta_{1}:= \min\{\vare_{1}/4, \eta(\vare_{1}, N)\}, 
$$
 where $\eta(\vare_1,N)>0$ is given by Proposition  \ref{prop-AlmostSplittingY}.

With these choices, if $\delta \in (0, \delta_{k}(\eta,N))$ and $(X,\sfd_{X}, \mm_{X})$ is an $\RCD^*(-\delta^{2\beta},N)$ space that satisfies $\bf A_k$, then there exist $\bar x_{k,\eta_{1}}\in \bar X$, a pointed $\RCD^*(0,N-k)$ space $(Y_{k,\eta_{1}}, \sfd_{Y_{k,\eta_{1}}}, \mm_{Y_{k,\eta_{1}}}, y_{k,\eta_{1}})$ and an $\eta_1$-mGH approximation: 
\begin{equation*}
\phi : B^{\bar X}_{\eta_1^{-1}}(\bar x_{k,\eta_{1}}) \to B_{\eta_1^{-1}}^{\R^k\times Y_{k,\eta_{1}}}(0^k,y_{k,\eta_{1}}). 
\end{equation*}
Moreover, by Proposition \ref{prop-AlmostSplittingY}, there exist $y \in B^{Y_{k,\eta_{1}}}_{\eta_1^{-1}/2}(y_{k,\eta_{1}})$ with $B^{Y_{k,\eta_{1}}}_{\eps_1^{-1}}(y) \subset B^{Y_{k,\eta_{1}}}_{\eta_1^{-1}}(y_{k,\eta_{1}})$, an $\RCD^*(0,N-k-1)$ space $(Y',\sfd_{Y'},\mm_{Y'},y')$ and an $\eps_1$-mGH approximation
\begin{equation*}
\varphi': B^{Y_{k,\eta_{1}}}_{\eps_1^{-1}}(y) \rightarrow B^{\R \times Y'}_{\vare_1^{-1}}(0,y').
\end{equation*}
 Since $\eta_{1}\in (0,\vare_{1})$, the inclusion  $B^{Y_{k,\eta_{1}}}_{\eps_1^{-1}}(y) \subset B^{Y_{k,\eta_{1}}}_{\eta_1^{-1}}(y_{k,\eta_{1}})$ ensures that   $B^{\R^k \times Y_{k,\eta_{1}}}_{\eps_1^{-1}}((0^k,y))\subset B^{\R^k \times Y_{k,\eta_{1}}}_{\eta_1^{-1}}((0^k,y_{k,\eta_{1}}))$.

Therefore, there exists $\bar x_{k+1,\eta_{1}}$ in $B^{\bar X}_{\eta_1^{-1}}(\bar x_{k,\eta_{1}})$ such that
\begin{equation*}
\sfd_{\R^k \times Y_{k,\eta_{1}}}((0^k,y), \phi(\bar x_{k+1,\eta_{1}})) \leq \eta_1.
\end{equation*}
We aim to show that 
\begin{equation}\label{eq:claimThm4.1}
\sfd_{mGH}(B^{\bar X}_{\vare^{-1}}(\bar x_{k+1,\eta_{1}}), B^{\R^{k+1}\times Y'}_{\vare^{-1}}(0^{k+1},y'))\leq \vare.
\end{equation}
We first claim that \begin{equation}\label{eq:claimBxkxk+1}
B^{\bar X}_{\eta_{1}^{-1}/4}(\bar x_{k+1,\eta_{1}}) \subset B^{\bar X}_{\eta_1^{-1}}(\bar x_{k,\eta_{1}}).
\end{equation}

 Indeed, since $\phi$ is an $\eta_1$-mGH approximation, by the definition of $\bar x_{k+1,\eta_{1}}$ and $y$  we have 
 \begin{align*}
\sfd_{\bar X}(\bar x_{k+1,\eta_{1}}, \bar x_{k,\eta_{1}}) &\leq \sfd_{\R^{k}\times Y_{k,\eta_{1}}}\big(\phi(\bar x_{k+1,\eta_{1}}), (0^{k},y_{k,\eta_{1}})\big)+\eta_1 \\
& \leq \sfd_{\R^{k}\times Y_{k,\eta_{1}}}\big(\phi(\bar x_{k+1,\eta_{1}}),(0^{k},y)\big)+\sfd_{\R^{k}\times Y_{k,\eta_{1}}}\big( (0^{k},y),(0^{k},y_{k,\eta_{1}}) \big) +\eta_1 \leq 2\eta_1 +\frac{1}{2}\eta_1^{-1}. 
\end{align*}

The claim  \eqref{eq:claimBxkxk+1} follows by triangle inequality.\\ 

As a consequence, since $\vare^{-1} \leq  \eta_{1}^{-1}/4$, by Proposition \ref{prop-restrmGH} we can construct  a $(C_1\eta_1)$-mGH approximation out of $\phi$: 
\begin{equation*}
\phi_1: B^{\bar X}_{\vare^{-1}}(\bar x_{k+1,\eta_{1}}) \to B_{\vare^{-1}}^{\R^k \times Y_{k,\eta_{1}}}((0^k,y)).
\end{equation*}
Thanks to Proposition \ref{prop-prodmGH}, there exists a $(C_2\eps_1)$-mGH approximation: 
$$
\varphi: B_{\eps_1^{-1}}^{\R^k}(0^k) \times B^{Y_{k, \eta_{1}}}_{\eps_1^{-1}}(y)\to B_{\eps_1^{-1}}^{\R^k}(0^k) \times B_{\eps_1^{-1}}^{\R\times Y'}((0,y')).
$$
Since the ball centred at $(0^k,y)$ of radius $\eps^{-1} \leq \eps_1^{-1}/\sqrt{2}$ is included in the previous product of balls, we can use again  Proposition \ref{prop-restrmGH} to construct a  $(C_1C_2 \eps_1)$-mGH approximation out of $\varphi$:
\begin{equation*}
\varphi_1 : B^{\R^k \times Y_{k,\eta_{1}}}_{\eps^{-1}}((0^k,y)) \rightarrow B^{\R^{k+1}\times Y'}_{\eps^{-1}}((0^{k+1},y')).
\end{equation*}
The composition of $\varphi_1$ with $\phi_1$ then gives a $(2C_1\eta_1+C_1C_2\vare_1)$-mGH approximation:
\begin{equation*}
f =\varphi_1 \circ \phi_1 : B_{\vare^{-1}}(\bar x_{k+1,\eta_{1}}) \rightarrow B^{\R^{k+1}\times Y'}_{\vare^{-1}}((0^{k+1},y')). 
\end{equation*}
Thanks to our choices of $C, \eps_1$ and $\eta_1$, the map $f$ is an $\eps$-mGH approximation and the claim \eqref{eq:claimThm4.1} is proved.

Finally, set $Y_{k+1,\vare}:=Y'$ and $y_{k+1, \eps}:=y'$. We have proven that given $\bf A_k$, for any $\eps \in (0, \eps(N)]$ there exists $\delta_{k+1}:=\delta_{k}(\vare,N)$ such that for any $\delta \in (0, \delta_{k+1}(\eps, N))$ and any  $\RCD^*(-\delta^{-2\beta},N)$ space $(X,\sfd,\mm)$ with $\diam(X)=1$ and $\be(X)=\lfloor N \rfloor$, there exist $\bar{x}_{k+1, \eps} \in \bar X$ and a pointed $\RCD^*(0,N-k-1)$ space $(Y_{k+1,\vare}, \sfd_{Y_{k+1,\vare}}, \mm_{Y_{k+1,\vare}},y_{k+1,\vare})$ 
such that
\begin{equation*}
\sfd_{mGH}(B^{\bar X}_{\eps^{-1}}(\bar x_{k+1, \eps}), B^{\R^{k+1}\times Y_{k+1, \eps}}_{\eps^{-1}}((0^{k+1},y_{k+1, \eps})) \leq \eps.
\end{equation*}
This shows that for any integer $0< k < \lfloor N \rfloor$, $\bf A_k$ implies $\bf A_{k+1}$. 
\end{proof}

\section{Proof of Theorem \ref{thm:main}, first claim}\label{sec:5}

In this section we prove the first part of the main Theorem \ref{thm:main}, by combining Theorem \ref{thm-GH} with the structure theory of $\RCD^{*}(K,N)$ spaces \cite{MondinoNaber, MondinoKell, GigliPasqualetto, BrueSemola}.   More precisely we show the following result, which in turn immediately implies the first claim of Theorem \ref{thm:main} by a standard scaling argument.

\begin{theorem}\label{main-delta}
For any $\eps \in (0,1)$ and $N \in (1, \infty)$ there exists $\delta(\eps,N)>0$ such that for all $\delta \in (0, \delta(\eps,N)]$, any $\RCD^*(-\delta,N)$ space $(X,\sfd,\mm)$ with $\be(X)=\lfloor N \rfloor$ and $\diam(X)=1$ has essential dimension equal to $\lfloor N \rfloor$ and it is $\lfloor N \rfloor$-rectifiable as a metric measure space. Moreover, if $N \in \N$, there exists $c>0$ such that $\mm=c \mathcal{H}^{\lfloor N \rfloor}$. 
\end{theorem} 
 
In \cite[Theorem 6.8]{MondinoNaber}, the authors proved that for any $\eps>0$ there exists $\delta>0$ such that if $(X,\sfd,\mm)$ is an $\RCD^*(-\delta, N)$ space and a ball of radius $\delta^{-1}$ is $\delta$-mGH close to a Euclidean ball of the same radius in $\R^{\lfloor N \rfloor }$, then there exists a subset (of large measure) $U_\eps$ of the unit ball which is $(1+\eps)$ bi-Lipschitz to a subset of $\R^{\lfloor N \rfloor }$. In order to construct $U_\eps$ and the bi-Lipschitz map into $\R^{\lfloor N \rfloor }$, they showed the existence of a function $u$ on the unit ball which,  restricted to any ball $B_s(x)$ centred at a point $x$ of $U_\eps$, is an $(\eps s)$-mGH isometry. We summarise these results in the following statement.

\begin{theorem}[{\cite[Theorem 6.8]{MondinoNaber}}]\label{thm-Ubilip}
For every $N \in (1, \infty)$ there exists $\delta_{0}=\delta_{0}(N)>0$  with the following property.  Let $(X, \sfd, \mm)$ be an $\RCD^*(-\delta,N)$ space for some $\delta\in [0, \delta_{0})$ and assume that for some $x_{0} \in X$
 it holds 
\begin{equation*}
\sfd_{mGH}( B^{X}_{\delta^{-1}}(x_{0}) ,  B^{\R^{\lfloor N \rfloor }}_{\delta^{-1}}(0^{\lfloor N \rfloor })) \leq \delta.
\end{equation*}
Then there exists a Borel subset $U_{\vare} \subset B_1(\bar x)$  such that
\begin{itemize}
\item [1.]$\mm(B_1(x_{0}) \setminus U_{\vare}) \leq \vare$; 
\item[2.]  $U_{\vare}$ is $(1+\vare)$ bi-Lipschitz to a subset of $\R^{\lfloor N \rfloor }$; 
\item[3.] For all $x \in U_{\vare}$ and for all $r \in (0,1]$ such that $B^{X}_r(x) \subset B^{X}_1(x)$, we have
$$\sfd_{mGH}(B_r^{X}(x), B^{\R^{\lfloor N \rfloor }}_{r}(0^{\lfloor N \rfloor }))\leq \eps r.$$
\end{itemize}
In particular, for any $x \in U_{\vare}$ and for any tangent cone $(Y,\sfd_Y, \mm_Y)$ at $x$ we have
$$\sfd_{mGH}(B^Y_1(x), B_1^{\R^{\lfloor N \rfloor }}(0^{\lfloor N \rfloor }) )\leq \eps.$$
\end{theorem}

The third property is contained in the proof of \cite[Theorem 6.8]{MondinoNaber}. Thanks to the constancy of the dimension of $\RCD^*(K,N)$ spaces proved by Bru\'e-Semola \cite{BrueSemola}, the following holds. 

\begin{corollary}\label{cor_EpsReg_EssDim}
For every $N \in (1, \infty)$ there exists $\delta_{0}=\delta_{0}(N)>0$  with the following property.  Let $(X, \sfd, \mm)$ be an $\RCD^*(-\delta,N)$ space for some $\delta\in [0, \delta_{0})$ and assume that for some $x_{0} \in X$
 it holds 
\begin{equation}
\label{eq_hypEpsReg}
\sfd_{GH}( B^{X}_{\delta^{-1}}(x_{0}), B^{\R^{\lfloor N \rfloor }}_{\delta^{-1}}(0^{\lfloor N \rfloor }) ) \leq \delta.
\end{equation}  
Then the essential dimension of $X$ is equal to $\lfloor N \rfloor $ and $(X,\sfd,\mm)$ is $\lfloor N \rfloor $-rectifiable as a metric measure space. 
\end{corollary}

\begin{proof}
By the definition of the dimension of $\RCD$ spaces, we know that there exists a unique $n \in \N$, with $n\leq \lfloor N \rfloor$, such that the $n$-th regular stratum $\mathcal{R}_n$ has positive measure. Therefore, by definition of $\mathcal{R}_n$ for $\mm$-a.e. $x \in X$, tangent cones at $x$ are unique and equal to the Euclidean space $(\R^n, \sfd_{\R^n}, \mathcal{L}^n)$. Now assume by contradiction that $n < \lfloor N \rfloor$. Because of Theorem \ref{thm-Ubilip}, \eqref{eq_hypEpsReg} implies the existence of a set $U_\eps$ satisfying properties 1 to 3, with $\mm(U_\eps)>0$. As a consequence, there exists $x \in U_\eps$ with unique tangent cone equal to $\R^n$. Property 3 then implies that the unit ball in $\R^n$ is $\eps$-GH close to the unit ball in $\R^{\lfloor N \rfloor }$, which is impossible for $n< \lfloor N \rfloor$ and $\eps>0$ sufficiently small. Therefore, $\lfloor N \rfloor $ is the essential dimension of $(X,\sfd,\mm)$ and $(X,\sfd,\mm)$ is $\lfloor N \rfloor $-rectifiable as a metric measure space. 
\end{proof}

The combination of Corollary \ref{cor_EpsReg_EssDim} and Theorem \ref{thm-GH} yields the following result.  

\begin{corollary}\label{cor-XbarRect}
For any $\eps \in (0,1)$ and $N \in (1, \infty)$ there exists $\delta(\eps,N)>0$ with the following property. If $(X,\sfd,\mm)$ is an $\RCD^*(-\delta,N)$ space  with $\be(X)=\lfloor N \rfloor$ and $\diam(X)=1$, then the covering space $(\bar X, \sfd_{\bar X},  \mm_{\bar X})$ has essential dimension equal to $\lfloor N \rfloor$ and it is $\lfloor N \rfloor$-rectifiable as a metric measure space. 
\end{corollary}

\begin{proof}
Fix $\eps\in (0,1)$ and let $\beta>0$ be as in Theorem \ref{thm-GH}. Let $\eta(\eps,N)$ be given in Corollary \ref{cor_EpsReg_EssDim} and set $\eps_1=\eta(\eps,N)$. Then by Theorem \ref{thm-GH}, there exists $\delta_1(\eps_1,N)>0$ such that for any $\delta \in (0, \delta_1(\eps_1,N)]$ and for any $\RCD^*(-\delta^{2\beta},N)$ space $(X,\sfd,\mm)$ with $\be(X)=\lfloor N \rfloor$ and $\diam(X)=1$, there exists $\bar x \in \bar X$ such that 
$$\sfd_{mGH}(B^{\bar X}_{\eps_1^{-1}}(\bar x), B^{\R^{\lfloor N \rfloor }}_{\eps_1^{-1}}(0^{\lfloor N \rfloor}))\leq \eps_1.$$
As a consequence, $(\bar X, \sfd_{\bar X}, \mm_{\bar X})$ satisfies the assumptions of Corollary \ref{cor_EpsReg_EssDim}, thus it has essential dimension equal to $\lfloor N \rfloor$ and it is $\lfloor N \rfloor$-rectifiable as a metric measure space. It suffices then to choose $\delta(\eps,N)=\delta_1(\eps_1,N)^{\frac{1}{2\beta}}$.         
\end{proof}

We are now in position to prove Theorem \ref{main-delta}. 

\begin{proof}[Proof of Theorem \ref{main-delta}]
Let $\bar p : \bar X \rightarrow X$ be the covering map and denote by $\cR_{\lfloor N \rfloor}(\bar X)$ the $\lfloor N \rfloor$-th regular set of $\bar X$. 
Recall that $\mm_{\bar X}(\bar X \setminus \cR_{\lfloor N \rfloor}(\bar X))=0$. Let $B_{r}^{\bar X}(\bar x)$ be a sufficiently small ball in $\bar X$ such that
$$
\bar p|_{B_{r}^{\bar X}(\bar x)}: B_{r}^{\bar X}(\bar x)\to B_{r}^{X}(\bar p (\bar x))
$$
is an isomorphism of metric measure spaces. Since for $\mm_{\bar{X}}$-a.e. $\bar x'\in B_{r}^{\bar X}(\bar x)$ the tangent cone is unique and equal to $\R^{\lfloor N \rfloor}$, the same is true for $\mm$-a.e. $x'\in  B_{r}^{X}(\bar p (\bar x))$ and thus the regular set $\cR^{\lfloor N \rfloor}$ of $X$ has positive $\mm$-measure. Therefore $(X,\sfd,\mm)$ has essential dimension equal to $\lfloor N \rfloor$ and it is $\lfloor N \rfloor$-rectifiable as a metric measure space. In particular, $\mm \ll \mathcal{H}^{\lfloor N\rfloor}$.

If $N$ is an integer, then  $\mm \ll \mathcal{H}^{N}$ and $(X,\sfd,\mm)$ is a compact weakly non-collapsed $\RCD^{*}(-\delta,N)$ space. Corollary 1.3 in \cite{H19} ensures that for any compact weakly non-collapsed $\RCD^*(-\delta, N)$ space, there exists $c>0$ such that $\mm=c\mathcal{H}^{ N}$, thus concluding the proof. 
\end{proof}

\section{Proof of Theorem \ref{thm:main}, second and third claims}
\label{s-proofMain}

Now we are in position to conclude the proof of Theorem  \ref{thm:main}.  Given a sequence  of $\RCD^*(-K_i,N)$  spaces $(X_i,\sfd_i,\mm_i)$ with $K_i < 0$ tending to zero, $\diam(X_i)=1$ and $\be(X_i)=\lfloor N \rfloor$, the proof consists in applying the results of equivariant pointed Gromov-Hausdorff convergence  as in Section \ref{ss-conv} to the sequence  $(\bar X_i, \sfd_i, \bar x_i)$ and subgroups $\Gamma'_i$  as in Lemma \ref{lem-Gam'}, in order to obtain equivariant convergence (up to a subsequence) to $(\R^b,  \sfd_{\R^b},  0, \mathbb{Z}^b)$.  Then we will conclude that the quotients $\bar X_i /  \Gamma'_i$  mGH converge to a flat torus, which, by applying Theorem \ref{thm-MKbihold}, will imply that for large $i$  the quotients are bi-H\"older homeomorphic to this torus. In the last step we show that $\bar X_i /  \Gamma'_i= X_i$.

We start with the following lemma. 

\begin{lemma}\label{prop-AbealianGroup}
Let $(X_i,  \sfd_i,  x_i,  \Gamma_i) \in \cM^p_{eq}$ be a sequence of spaces that converge in equivariant pGH sense to 
$(X_\infty ,\sfd_\infty , x_\infty,  \Gamma_\infty) \in \cM^p_{eq}$. Assume  $\Gamma_i$ is an abelian group, for each $i\in \N$. Then $\Gamma_\infty$ is an  abelian group as well. 
\end{lemma}

\begin{proof}
Given arbitrary $\gamma_{\infty1}, \gamma_{\infty2} \in \Gamma_\infty$, we will show that they commute. For that,  
by hypothesis there exist $\varepsilon_i$-equivariant pGH approximations  $(f_i,\phi_i,\psi_i)$:
$$f_i:  B^{X_{\infty}}_{\varepsilon_i^{-1}}(x_\infty) \to  X_i, \quad \phi_i:  \Gamma_\infty(\varepsilon_i^{-1}) \to \Gamma_i, \quad \psi_i:  \Gamma_i(\varepsilon_i^{-1}) \to \Gamma_\infty,$$
satisfying the conditions of Definition \ref{def-eqGH} and so that $\varepsilon_i \to 0$. 
\\Take an arbitrary point  $z_\infty \in X_\infty$.   By the triangle inequality and for $i$ large enough such that $z_\infty, \gamma_{\infty 1}z_\infty, \gamma_{\infty 1}\gamma_{\infty 2}z_\infty \in B^{X_{\infty}}_{\varepsilon_i^{-1}}(x_\infty)$ and $\gamma_{\infty 1}\gamma_{\infty 2} \in \Gamma_{\infty}(\varepsilon_i^{-1})$, we get
\begin{align*}
\sfd_{i}  (f_i(\gamma_{\infty1} \gamma_{\infty2} z_\infty), \phi_i(\gamma_{\infty1}) \phi_i(\gamma_{\infty2}) f_i(z_\infty)) \leq &
\sfd_{i}  (f_i(\gamma_{\infty1} \gamma_{\infty2} z_\infty), \phi_i(\gamma_{\infty1}) f_i( \gamma_{\infty2}z_\infty))  \\
 & + \sfd_{i}  ( \phi_i(\gamma_{\infty1}) f_i( \gamma_{\infty2}z_\infty), \phi_i(\gamma_{\infty1}) \phi_i(\gamma_{\infty2}) f_i(z_\infty)).
\end{align*}

Applying (4) of Definition \ref{def-eqGH} and that  $\phi_i(\gamma_{\infty1})$ is an isometry, we see that each term in the right hand side of the previous inequality is bounded above by $\varepsilon_i$. We conclude that  
$$\sfd_{i}  (f_i(\gamma_{\infty1} \gamma_{\infty2} z_\infty), \phi_i(\gamma_{\infty1}) \phi_i(\gamma_{\infty2}) f_i(z_\infty)) \leq 2\varepsilon_i.$$
The same estimate holds reversing the roles of $\gamma_{\infty1}$ and $\gamma_{\infty2}$,  that is:
$$\sfd_{i}  (f_i(\gamma_{\infty2} \gamma_{\infty1} z_\infty), \phi_i(\gamma_{\infty2}) \phi_i(\gamma_{\infty1}) f_i(z_\infty)) \leq 2\varepsilon_i.$$
By the triangle inequality and using that $\Gamma_i$ is abelian, so that $\phi_i(\gamma_{\infty2}) \phi_i(\gamma_{\infty1})= \phi_i(\gamma_{\infty1}) \phi_i(\gamma_{\infty2})$, we get:
$$\sfd_{i}  (f_i(\gamma_{\infty1} \gamma_{\infty2} z_\infty), f_i(\gamma_{\infty2} \gamma_{\infty1} z_\infty)) \leq 4 \varepsilon_i.$$
From (3) of Definition \ref{def-eqGH}, we also have:
$$|\sfd_{\infty}(\gamma_{\infty1}\gamma_{\infty2} z_\infty,   \gamma_{\infty2}\gamma_{\infty1} z_\infty)- \sfd_i(f_i(\gamma_{\infty1}\gamma_{\infty2} z_\infty), f_i(\gamma_{\infty2}\gamma_{\infty1} z_\infty))| < \varepsilon_i.$$
Therefore, when taking the limit as $i\to \infty$ we obtain $\sfd_{\infty}(\gamma_{\infty1}\gamma_{\infty2} z_\infty,   \gamma_{\infty2}\gamma_{\infty1} z_\infty)=0$. 
\\Since $z_\infty\in X_{\infty}$ is an arbitrary point,  we conclude that $\gamma_{\infty1}$ and $\gamma_{\infty2}$ commute. 
\end{proof}

We are now ready to prove the  key result of this section, which directly gives the second claim of Theorem \ref{thm:main} by a standard compactness/contradiction argument.

\begin{proposition}\label{lem-eqGHconv}
Let $N\in (1,\infty)$ and let $(X_i,\sfd_i,\mm_i)$ be a sequence of $\RCD^*(-K_i,N)$ spaces with $\be(X_i)= \lfloor N \rfloor$, $\diam(X_i)=1$ and
$K_i>0  $   such that $K_i \downarrow 0$.   Fix some  $\bar x_i \in \bar X_i$ and let  $\Gamma'_i$ be  as in Lemma \ref{lem-Gam'}, for $k=3$. 
\\Then any Gromov-Hausdorff limit of $X_i'=  \bar X_i/\Gamma'_i$ is isometric to an $\lfloor N \rfloor$-dimensional flat torus.  
\end{proposition}

\begin{remark}\label{rmrk-RicDiam}
In Proposition \ref{lem-eqGHconv} we require $\diam(X_i)=1$ instead of the bound $K_i\diam(X_i)^{2} \downarrow 0$. To show that the latter condition is {\bf not} enough, consider a sequence $X_i$ of manifolds with $K_i=  i$ and $\diam(X_i)=i^{-1}$. Then $K_i\diam(X_i)^{2} \downarrow 0$ but any GH limit of this sequence collapses due to $\diam(X_i) \to 0$.   We could also consider manifolds $X_i$ with $K_i=i^{-3}$ and $\diam(X_i)=i$ then $K_i\diam(X_i)^{2} \downarrow 0$ and any GH converging subsequence has a limit space with infinite diameter.  Hence, it is necessary to have two sided uniform bounds on $\diam(X_i)$ and for simplicity we set them equal to 1.
\end{remark}

\begin{proof}[Proof of Proposition \ref{lem-eqGHconv}]
Set $b:=\lfloor N \rfloor= \be(X_i)$. For simplicity of notation, we will not relabel subsequences. 
By Theorem \ref{thm-GH} and Remark \ref{rem:thm-GH}, the sequence  $(\bar X_i,  \sfd_{ \bar X_i}, \bar x_i)$ converges in pointed Gromov-Hausdorff sense to 
$(\mathbb R^b,  \sfd_{\R^b}, 0^b)$. 
By Gromov's compactness Theorem and stability of the $\RCD^{*}(0,N)$ condition, there exists an $\RCD^{*}(0,N)$ space $(X,\sfd_{X}, \mm_{X})$ with $\diam(X)=1$ such that $X_{i}\to X$ in mGH sense, up to a subsequence.  
From Remark \ref{rem-Gamma'closed} we know that, for any $i\in \N$, the groups $\Gamma_i'$ given by Lemma \ref{lem-Gam'} are closed. Thus, by Theorem \ref{thm-GHtoEq} there exist a group of isometries  of $\R^b$, $\Gamma'_\infty$, and a subsequence  $(\bar X_i,  \sfd_{ \bar X_i}, \bar x_i, \Gamma'_i)$ that converges 
in the equivariant  pointed Gromov-Hausdorff sense to $(\mathbb R^b, \sfd_{\mathbb R^b},0^{b}, \Gamma'_\infty)$. 
Moreover, $\R^{b}$ is the universal cover of $X$, and $\Gamma'_\infty$ is contained in the corresponding group of deck transformations. 

We will show that $\mathbb R^b/\Gamma'_\infty$ is a flat torus. To this aim, we prove that $\Gamma'_\infty$ is isomorphic to $\mathbb Z^b$. 
\\

{\bf Step 1}. We claim that
\begin{equation}\label{eq:ClaimStep1FT}
\sfd_{\R^b} (\gamma_\infty y_\infty,  y_\infty) \geq  1, \quad \text{for all } y_\infty \in \mathbb R^b \;  \text{ and for all }  \gamma_\infty \in \Gamma'_\infty, \;  \gamma_\infty\neq {\rm id}.
\end{equation}
Let $(f_i,\phi_i,\psi_i)$ be equivariant
$\varepsilon_i$-pGH approximations,    $\varepsilon_i \to 0$, as in Definition \ref{def-eqGH}:
$$
f_i:  B^{\R^{b}}_{\varepsilon_i^{-1}}(0^{b}) \to \bar X_i, \quad \phi_i:  \Gamma'_\infty(\varepsilon_i^{-1}) \to \Gamma'_i \quad \psi_i:  \Gamma'_i(\varepsilon_i^{-1}) \to \Gamma'_\infty.
$$
To prove \eqref{eq:ClaimStep1FT}, we first show that the claim holds for all non trivial $\gamma_i \in \Gamma'_i$ and all $y_i \in \bar X_i $, $i \in \mathbb N$.  Then a convergence argument will show that the claim holds.

Since $\diam(X_i)=1$, for all $i\in \mathbb N$ and $y_i \in \bar X_i$ there exists $\gamma \in \Gamma_i'$ such that $\sfd_{\bar X_i}(\gamma \bar x_i,y_i) \leq 1$. Moreover, by Lemma \ref{lem-Gam'} for any $\gamma' \in \Gamma_i' \setminus\{{\rm id}\}$, we have $3 < \sfd_{\bar X_i}(\gamma' \bar x_i,\bar x_i)$. 
Then, by the triangle inequality, 
\begin{align*}
3 <  \sfd_{\bar X_i}(\gamma' \bar x_i,\bar x_i) = & \sfd_{\bar X_i}(\gamma' \gamma \bar x_i, \gamma \bar x_i) \\
\leq  & \sfd_{\bar X_i}(\gamma' \gamma \bar x_i, \gamma' y_i)   +  \sfd_{\bar X_i}(\gamma' y_i,  y_i)  + \sfd_{\bar X_i}(y_i, \gamma \bar x_i) \\
\leq & 2 + \sfd_{\bar X_i}( \gamma'  y_i,  y_i).
\end{align*}
Therefore:
\begin{equation}\label{eq:claimXiStep1}
\sfd_{\bar X_i}(\gamma' y_i,  y_i) > 1, \quad \text{for all  $\gamma' \in \Gamma'_i\setminus\{{\rm id}\}$ and $y_i\in \bar X_i$}.   
\end{equation}
Now let $\gamma_\infty \in \Gamma'_\infty\setminus \{\rm id\}$ and $y_\infty \in \mathbb R^b$. For $i$ large enough, $\gamma_\infty y_\infty, y_\infty \in \Gamma'_\infty(\varepsilon_i^{-1})$ and then 
by (3) of Definition \ref{def-eqGH} it holds:
\begin{equation}\label{eq1claim2}
\sfd_{\R^b}(\gamma_\infty y_\infty, y_\infty) > -\vare_i + \sfd_{\bar X_i}(f_i(\gamma_\infty y_\infty), f_i(y_\infty)).
\end{equation}
By (4) of Definition \ref{def-eqGH},  we also have:
\begin{equation}\label{eq2claim2}
\sfd_{\bar X_i}(f_i(\gamma_\infty y_\infty),  \phi_i(\gamma_\infty)f_i(y_\infty)) < \varepsilon_i.
\end{equation}
Combining  \eqref{eq1claim2}, the triangle inequality and \eqref{eq2claim2} we get
\begin{equation}
\label{eq3claim2}
\begin{split}
\sfd_{\R^b}(\gamma_\infty y_\infty, y_\infty) > & -\vare_i +  \sfd_{\bar X_i}(f_i(y_\infty),  \phi_i(\gamma_\infty)f_i(y_\infty)) - \sfd_{\bar X_i}(f_i(\gamma_\infty y_\infty),  \phi_i(\gamma_\infty)f_i(y_\infty))\\
>   &  \sfd_{\bar X_i}(f_i(y_\infty),  \phi_i(\gamma_\infty)f_i(y_\infty)) - 2\vare_i. 
\end{split}
\end{equation}
If we show that  $\phi_i(\gamma_\infty)\neq {\rm id}$ then we have that $\sfd_{\bar X_i}(\phi_i(\gamma_{\infty})f_i(y_\infty),f_i(y_\infty))>1$ and by passing to the limit we will be able to conclude the proof of the claim. We are going to prove that $\sfd_{\bar X_i}(\phi_i(\gamma_{\infty})f_i(y_\infty),f_i(y_\infty))>0$, so that $\phi_i(\gamma_{\infty})\neq {\rm id}$. By the triangle inequality, arguing as in \eqref{eq1claim2}  and using \eqref{eq2claim2} we get
\begin{equation}\label{eq4claim2}
\begin{split}
\sfd_{\bar X_i}(f_i(y_\infty),  \phi_i(\gamma_\infty)f_i(y_\infty)) \geq   &
\sfd_{\bar X_i}(f_i(y_\infty),  f_i(\gamma_\infty y_\infty)) - \sfd_{\bar X_i}(  f_i(\gamma_\infty y_\infty),   \phi_i(\gamma_\infty)f_i(y_\infty) ) \\
\geq   &
\sfd_{\R^b}(y_\infty,  \gamma_\infty y_\infty) -  2 \vare_i.
\end{split}
\end{equation}
Since by hypothesis $\gamma_\infty$ is a non trivial isometry and elements in the deck transformations do not fix points, we have $\sfd_{\R^b}(y_\infty,  \gamma_\infty y_\infty) > 0$. 
Thus by \eqref{eq4claim2} for sufficiently large $i$, $\sfd_{\bar X_i}(f_i(y_\infty),  \phi_i(\gamma_\infty)f_i(y_\infty)) >0$.  This shows that $\phi_i(\gamma_\infty)$ is non trivial and thus  $\sfd_{\bar X_i}(f_i(y_\infty),  \phi_i(\gamma_\infty)f_i(y_\infty)) >1$. Therefore, as $i \to \infty$, inequality \eqref{eq3claim2} implies the claim \eqref{eq:ClaimStep1FT}.
\\

{\bf Step 2}. We show that  $\Gamma'_{\infty} \cong \dZ^{b}$. 
\\From Lemma  \ref{lem-Gam'} we know that $\Gamma'_i \cong \dZ^{b}$. Let $\{\gamma_{ij}\}_{j=1}^b$ be a set of generators for $\Gamma'_i$.

 By the Arzel\'a-Ascoli theorem there exist a  subsequence  $(\bar X_{i_k},  \sfd_{\bar X_{i_k}}, \bar x_{i_k}, \Gamma'_{i_k})$ 
and corresponding subsequences of isometries $\{\gamma_{i_k1}\}_{k=1}^\infty, \dots, \{\gamma_{i_k b}\}_{k=1}^\infty$ that converge to 
$\gamma_{\infty 1}, \dots, \gamma_{\infty b}   \in \Gamma'_\infty$, respectively. We are going to show that $\{\gamma_{\infty j}\}_{j=1}^b$ are independent generators of $\Gamma'_{\infty}$ and that they have infinite order. 

To simplify notation consider that the whole sequence converges.  Given $\gamma_\infty  \in  \Gamma_\infty'$, notice that  $\phi_i(\gamma_\infty)   \to  \gamma_\infty$ in Arzel\'a-Ascoli sense.  Indeed, for all  $z \in \R^b$ and  $z_i  \in \bar X_i$  such that $\sfd_{\bar X_i}(f_i(z), z_i)  \to  0$,
by using the triangle inequality and (4) in Definition \ref{def-eqGH}, and since $\phi_i(\gamma_\infty)$ is an isometry, we have
\begin{align*}
\sfd_{\bar X_i}( \phi_i(\gamma_\infty) z_i,  f_i( \gamma_\infty  z))  \leq &
\sfd_{\bar X_i}( \phi_i(\gamma_\infty) z_i,  \phi_i(\gamma_\infty) f_i(z))  +    \sfd_{\bar X_i}(\phi_i(\gamma_\infty) f_i(z),   f_i( \gamma_\infty  z)) \\
\leq  &   \sfd_{\bar X_i}( z_i,  f_i(z) ) +  \vare_i   \to  0.
\end{align*} 
Moreover, for any $\gamma_\infty \in  \Gamma'_\infty$, there exist $s_1, \ldots s_b \in \mathbb{Z}$ such that $\phi_i(\gamma_\infty)= \gamma_{i 1}^{s_1} \cdots  \gamma_{i b}^{s_b}$. Then we know that the left hand side of the previous equation converges to $\gamma_\infty$, while the right hand side converges to  $\gamma_{\infty 1}^{s_1} \cdots  \gamma_{\infty b}^{s_b}$.   Thus, any $\gamma_\infty \in \Gamma'_{\infty}$  can be written as a composition of elements in $\{\gamma_{\infty j} \}_{j=1}^b$. 

We next show that $\{\gamma_{\infty j} \}_{j=1}^b$ are independent and have infinite order. Let $(s_{1},\ldots, s_{b})\in \mathbb Z^{b} \setminus \{(0, \ldots, 0)\}$. We claim that  $\gamma_{\infty 1}^{s_1} \cdots  \gamma_{\infty b}^{s_b}\neq {\rm id}$. From the previous arguments, we know that $\gamma_{i 1}^{s_1} \cdots  \gamma_{i b}^{s_b} \to \gamma_{\infty 1}^{s_1} \cdots  \gamma_{\infty b}^{s_b}$ as $i\to \infty$.
Since $\{\gamma_{i j} \}_{j=1}^b$ are independent generators of $\Gamma_{i}'\cong \dZ^{b}$, we have that $\gamma_{i 1}^{s_1} \cdots  \gamma_{i b}^{s_b}\neq {\rm id}$.
Hence, from \eqref{eq:claimXiStep1} it follows that  
 \begin{equation*} 
 1  <  \sfd_{\bar X_i}( \gamma_{i 1}^{s_1} \cdots  \gamma_{i b}^{s_b}\,  f_{i}(z), \,  f_{i}(z))   \to  \sfd_{\R^b}( \gamma_{\infty 1}^{s_1} \cdots  \gamma_{\infty b}^{s_b} \, z, \,  z), \quad \text{for all $z\in \R^{b}$, }
 \end{equation*}
  and thus   $\gamma_{\infty 1}^{s_1} \cdots  \gamma_{\infty b}^{s_b}\neq {\rm id}$.

In conclusion, by the fundamental theorem of finitely generated abelian groups, we infer that $\Gamma_\infty'\cong \dZ^{b}$. Thus, $\R^b / \Gamma_\infty'$ is a $b$-dimensional flat torus. The proposition follows now by Theorem \ref{thm-eqGHtoOrb}.
 \end{proof}

\begin{corollary}\label{lem-torus}
For all $N \in \mathbb N$, $N > 1$, there exists $\varepsilon(N) > 0$ with the following property.
Let $(X,\sfd, \mathcal H^N)$ be a compact $\RCD^{*}(K,N)$ space with $K \diam^2 (X) > - \varepsilon(N)$ and $\be(X) =N$.
Then $X':= \bar X/ \Gamma'$ is bi-H\"older homeomorphic to an $N$-dimensional flat torus, where   $\Gamma'$ is given  by Lemma \ref{lem-Gam'}, for $k=3$.
\end{corollary}

\begin{proof}

Suppose by contradiction that there is no such $\vare(N)>0$. Then there exists a sequence of 
compact $\RCD^*(K,N)$ spaces $(X_i,\sfd_i, \mathcal H^N)$ with  $K_i \diam^2 (X_i) > - \varepsilon_i$, $\be(X_i) =N$, $\vare_i  \to  0$ such that none of the $X_i$  is
 bi-H\"older homeomorphic to a flat torus of dimension $N$.  Consider the rescaled spaces $(X_i', \sfd_i', \cH^N):=(X_i,  \diam(X_i)^{-1} \sfd_i, \mathcal H^N )$. Clearly $X_{i}'$ has diameter equal to 1 and it is an $\RCD^*( K_i \diam^2 (X_i, \sfd_i), N)$   space with $\be(X_{i}')=N$.  Thus we can apply Proposition \ref{lem-eqGHconv}  and infer that any GH-limit is a flat torus $\mathbb T^N$. 
  
Moreover, from Theorem \ref{thm-DPhG} (i) we have that $(X_i', \sfd_i', \cH^N)$ converges in mGH sense to  $(\mathbb T^N, \sfd_{\mathbb T^{N}}, \mathcal H^N)$. For $i$ large enough so that $\sfd_{mGH}(X'_i, \mathbb{T}^N) \leq \eps(\mathbb{T}^N)$, we can apply Theorem  \ref{thm-MKbihold} and get that $X_i'$ is bi-H\"older homeomorphic to $\mathbb T^N$. When scaling back to the original metric,  the same conclusion holds.  This is a contradiction.

\end{proof}

We can now conclude the proof of the main theorem. 

\begin{proof}[Proof of the third claim of  Theorem  \ref{thm:main}, i.e. when $N \in \N$]
If $N=1$, the claim holds trivially (see Remark \ref{rem:N=1});  thus, we can assume $N\geq 2$ without loss of generality.
\\ From Corollary \ref{lem-torus}, we know that $(\bar{X}, \sfd_{\bar X})$ is locally (on arbitrarily large compact subsets)  bi-H\"older homeomorphic to $\R^N$ (thus in particular it has the integral homology of a point) and $\mm_{\bar X}$ is a constant multiple of the $N$-dimensional Hausdorff measure $\mathcal H^{N}$.
By construction, we also know that the abelianised revised fundamental group $\Gamma:=\bar{\pi}_1(X)/H$ acts by deck transformations on $\bar{X}:=\widetilde{X}/H$ and that $X=\bar{X}/\Gamma$. Thus, summarising:
\begin{equation}\label{eq:GammaNoFixPoints}
\begin{split}
&\text{ $(\bar X,\sfd_{\bar X})$ is a topological manifold with the integral homology of a point} \\
& \qquad \qquad  \text{and  the action of $\Gamma$ on $\bar X$ has no fixed points.}
\end{split}
\end{equation} 
In order to prove that $(X,\sfd)$ is bi-H\"older homeomorphic to a flat torus and that $\mm$ is a constant multiple of $\mathcal H^{N}$, it is  enough to prove that $\Gamma \cong \mathbb Z^N$. Since $\Gamma$ is a finitely generated abelian group (recall Proposition \ref{prop:RevFundGFinGen}), it is sufficient to show that $\Gamma$ has no subgroup isomorphic to $\mathbb Z/ p \mathbb Z$ with $p$ prime. This follows from \eqref{eq:GammaNoFixPoints}: indeed,  from Smith theory (see for instance \cite[Chap. 3]{bredon}), if $\mathbb Z/ p \mathbb Z$, with $p$ prime, acts on a topological manifold with the $\mod p$ homology of a point then the set of fixed points is non empty. \end{proof}

%%%%%%%%%%%%%%%%%%%%
%%%%%%%%%%%%%%%%%%%%
%%%%%%%%%%%%%%%%%%%%

\section{Appendix: some basic properties of mGH approximations}

%%%%%%%%%%%%%%%
%%%%%%%%%%%%%%%
For the reader's convenience, in this appendix we recall some well known properties of mGH approximations used in the paper.

\begin{proposition}[Restriction of mGH approximations]\label{prop-restrmGH}
Fix $K\in \R$, $N\in (1,\infty)$ and $V>0$. 
Then there exists a constant $C=C(K,N,V)>0$ with the following properties.
Let   $(X,\sfd_{X},\mm_{X})$ and $(Y,\sfd_{Y},\mm_{Y})$  be $\CD^{*}(K,N)$ spaces. Assume that $V^{-1}\leq \mm_{X}(B^{X}_{R}(x))\leq V$   and that there exists an $\varepsilon$-mGH approximation
 $$\phi :  B^{X}_{R}(x) \to  B^{Y}_{R}(y), \; \text{ with $\phi(x)=y$}. $$  
Let  $r\in (0,\vare)$ and $y' \in Y$  with $\sfd_Y(y',y)\leq  R-r + 2\varepsilon$, so that
$B^Y_{r}(y')  \subset  B^Y_{R}(y)$  and thus  we can choose  $x'  \in B^{X}_{R}(x)$ such that 
\begin{equation}\label{eq-xk+1}
\sfd_{Y}(\phi(x'),  y') < \varepsilon
\end{equation} 
and  $B^{X}_{r}(x') \subset B^{X}_{R}(x)$.
Then the function   $\varphi: B^{X}_{r}(x') \to B^{Y}_{r}(y')$ given by 
\begin{equation}\label{eq:defRestr}
 \varphi(z)=
 \begin{cases}
\phi(z)   & \text{if }   \phi(z) \in B^{Y}_{r}(y')\\
w \;  \text{ for some  $w \in \partial B^{Y}_{r}(y')$ with }  \, \sfd_{Y}(w, \phi(z))= \sfd_{Y}(B^{Y}_{r}(y'), \phi(z))  & \text{ otherwise.} 
\end{cases}
\end{equation}
is a $C\vare$-mGH approximation. 
\end{proposition}

\begin{proof}
 Before calculating the distortion of $\varphi$ we see that for all $z  \in B^{X}_{r}(x')$  we have 
 \begin{equation}\label{eq-out}
\sfd_{Y}( \varphi(z), \phi(z)) \leq  2\vare.
\end{equation}
Indeed, for any  $z \in B^{X}_{r}(x')$, using that $\phi$ is an $\vare$-GH approximation and  the  definition of 
$x'$ in  \eqref{eq-xk+1}, we get
\begin{align*}
\sfd_{ Y} (\phi(z), y') \leq & \sfd_{Y} (\phi(z),  \phi(x')) + \sfd_{ Y} (\phi(x'), y')
 \leq    \sfd_{X} (z, x')  + 2 \vare <  r  +  2\vare. 
\end{align*}
Hence, if  $\phi(z) \notin B^{Y}_{r}(y')$ then  $\sfd_{Y}( \varphi(z), \phi(z)) \leq 2\vare$.  
The other case is trivial. 

\textbf{Step 1}. Control of the distortion of $\varphi$.
\\Let  $z,z' \in  B^{X}_{r}(x')$ such that $\varphi(z)=w$ and $\varphi(z')=w'$. 
Then by  \eqref{eq-out} and using that $\phi$ is a $\vare$-GH approximation, we get 
\begin{align*}
\sfd_{Y}(\varphi(z), \varphi(z'))  \leq   &  \sfd_{Y}(w,\phi(z))  + \sfd_{Y}(\phi(z), \phi(z'))  + \sfd_{Y}( \phi(z'), w') \\
  \leq   &  2\vare + \{ \sfd_{X}(z, z')  +  \vare\}  + 2\vare  \\
\leq & 5\vare +   \sfd_{X}(z, z').  
\end{align*}
In a similar way, we can get $   \sfd_{X}(z, z') \leq    5\vare +  \sfd_{Y}(\varphi(z), \varphi(z'))$.
So,  $dist(\varphi) \leq 5\vare$.

\textbf{Step 2}. Almost surjectivity of  $\varphi$.
\\Next, we show that for any  $w \in B^{Y}_{r}(y')$ there exists  $z'\in B^{X}_{r}(x')$ such that $\sfd_{Y} (w, \varphi(z') )  \leq 7\vare$.
\\Let  $w \in B^{Y}_{r}(y')$. 
Since  $B^{Y}_{r}(y') \subset B^{Y}_{R}(y)$ and $\phi$  is an $\vare$-GH approximation,
there exists  $z\in B^{X}_{R}(x)$ such that $\sfd_{Y} (w, \phi(z) )  \leq \vare$. 
If $z \in  B^{X}_{r}(x')$ we set $z'=z$.  By  \eqref{eq-out} we get
\begin{align*}
\sfd_{Y}(\varphi(z') , w) \leq   \sfd_{Y}(\varphi(z') , \phi(z')  ) +  \sfd_{Y}(\phi(z'), w) \leq 2\vare +  \vare= 3\vare.
 \end{align*}
If $z \notin  B^{X}_{r}(x')$,   let $z' \in \partial B^{X}_{r}(x')$ be a  closest point to $z$. 
Then, by   \eqref{eq-out} and using that  $\phi$ is a $\vare$-GH approximation, we get 
 \begin{align*}
\sfd_{Y}(\varphi(z') , w) \leq  &  \sfd_{Y}(\varphi(z') , \phi(z')  ) +   \sfd_{ Y}(\phi(z'), \phi(z)) +  \sfd_{Y}(\phi(z), w)  \\
 \leq  & 2\vare +  \{ \sfd_{X}(z', z) +\vare \}+  \vare.
 \end{align*}
We next estimate $\sfd_{X}(z', z)$. For this,  by  the  definition of $z'$ it is enough to estimate $\sfd_{X}(z, x')$. We have: 
\begin{align*}
\sfd_{X}( z, x')  \leq &    \sfd_{Y}( \phi(z),   \phi(x')) +  \vare \\
  \leq &  \{  \sfd_{Y}( \phi(z),   w)    +    \sfd_{Y}(  w, y' )     +   \sfd_{Y}( y' ,   \phi(x') ) \} +  \vare  \\
    \leq &  \{  \vare   +   r  +  \vare \} +  \vare  = 3 \vare +  r.
\end{align*}
Thus, $\sfd_{X}(z', z) \leq 3 \vare$ and $\sfd_{Y}(\varphi(z') , w) \leq  7\vare$. 

\textbf{Step 3}. Control of the measure  distortion.
\\Using that $\phi$ is an $\vare$-GH approximation and the definition \eqref{eq:defRestr} of $\varphi$, it is clear that  
\begin{equation}\label{eq:varphiequivphi}
\varphi\equiv \phi \text{ on $B_{r-2\vare} (x). $}
\end{equation}
From the Bishop-Gromov volume comparison, we have that there exists  $\bar C=\bar{C}(K,N,V)>0$ such that
\begin{equation}\label{eq:varphiequivphi2}
 \mm_{X}(B^{X}_{r}(x)\setminus B^{X}_{r-2\vare}(x))\leq \bar{C}(K,N,V)\,  \vare.
\end{equation}
The combination of \eqref{eq:varphiequivphi}, \eqref{eq:varphiequivphi2} with the fact that $\phi$ is a $\vare$-GH approximation gives (together with steps 1 and 2) that $\phi$ is a $C\vare$-GH approximation for some $C=C(K,N,V)>0$.
\end{proof}

\begin{remark} 
\label{rem-restr}
Observe that the previous argument also shows that if $\phi : B^{X}_{R}(x) \to  B^{Y}_{R}(y)$ is an $\vare$-GH approximation and $r < R$, the restriction $\varphi: B^X_r(x) \to B^Y_r(y)$ defined in \eqref{eq:defRestr}  is a $7\vare$-GH approximation. The dependence of $C$ on $K,N$ and $V$ comes  only in estimating the distortion of the measure.
\end{remark}

%%%%%%%%%%%%%%%%%%%%%%%%
%%%%%%%%%%%%%%%%%%%%%%%%
%%%%%%%%%%%%%%%%%%%%%%%%%

%$B^{X}_{r/\sqrt{2}}(x) \times B^{Y}_{r/\sqrt{2}}(y) \subset  B_{r}^{X\times Y} ((x,y))\subset B^{X}_{\sqrt{2} r}(x) \times B^{Y}_{\sqrt{2} r}(y)$. Indeed in the proof of Theorem 3.3 we can apply Proposition \ref{prop-prodmGH}, get a mGH approximation defined on a product of balls of radius $r$, and then restrict is to get a mGH approximation to a ball of radius $r/\sqrt{2}$ in the ball in the product.

\begin{proposition}[Product with an Euclidean factor]\label{prop-prodmGH}
There exists a universal constant $C>0$ with the following properties.
Let $(Y, \sfd_{Y}, \mm_{Y})$ and $(Y',\sfd_{Y'}, \mm_{Y'})$ be metric measure spaces.
 Let   $$\phi:\bar{B}^Y_{r}(y) \rightarrow \bar{B}^{Y'}_{r}(y')$$  be an $\varepsilon$- mGH approximation with $\phi(y)=y'$ and $\vare\in (0,1)$.
 \\ Define
 $\varphi:  \bar{B}^{\R^k}_{r}(0^{k}) \times \bar{B}^{Y}_{r}(y)\to  \bar{B}^{\R^k}_{r}(0^{k}) \times \bar{B}^{Y'}_{r}(y') $ 
 by 
$$ \varphi(a,z)= (a, \phi(z)), \quad \text{ for all } (a,z)\in \bar{B}^{\R^k}_{r}(0^{k}) \times \bar{B}^{Y}_{r}(y).$$
Then   $\varphi$ is  a  $C\varepsilon$-mGH approximation.
\end{proposition}

\begin{proof}

\textbf{Step 1}. We first show that  $\varphi: \bar{B}^{\R^k}_{r}(0^{k}) \times \bar{B}^{Y}_{r}(y)\to  \bar{B}^{\R^k}_{r}(0^{k}) \times \bar{B}^{Y'}_{r}(y')$ is  a  $3\varepsilon$-GH approximation.
\\ To this aim, note that since $\phi$ is an $\varepsilon$-GH approximation
\begin{align}
|\sfd_{\R^{k}\times Y'}(\varphi(a_{1},z_{1}), \varphi(a_{2},z_{2}))^{2}-\sfd_{\R^{k}\times Y}((a_{1},z_{1}),(a_{2},z_{2}))^{2}| & =
 | \sfd_{Y'}^2 (\phi(z_{1}), \phi(z_{2})) -  \sfd_{Y}^2(z_{1}, z_{2})|    \label{eq:dprod} \\
&\leq   \varepsilon^{2}+2 \varepsilon\, \sfd_{Y}(z_{1},z_{2}).  \nonumber
\end{align}
In case $\sfd_{Y}(z_{1},z_{2})\leq \vare$,   by  \eqref{eq:dprod} and since $\phi$ is an $\varepsilon$-GH approximation we have
\begin{align}\label{eq:estdYleqeps}
|\sfd_{\R^{k}\times Y'}(\varphi(a_{1},z_{1}), \varphi(a_{2},z_{2}))-\sfd_{\R^{k}\times Y}((a_{1},z_{1}),(a_{2},z_{2}))| & =  
\frac{  |\sfd_{\R^{k}\times Y'}(\varphi(a_{1},z_{1}), \varphi(a_{2},z_{2}))^{2}-\sfd_{\R^{k}\times Y}((a_{1},z_{1}),(a_{2},z_{2}))^{2}|}
{ \sfd_{\R^{k}\times Y'}(\varphi(a_{1},z_{1}), \varphi(a_{2},z_{2})) +   \sfd_{\R^{k}\times Y}((a_{1},z_{1}),(a_{2},z_{2}))}  \nonumber \\
&  \leq  \frac{
 | \sfd_{Y'}^2 (\phi(z_{1}), \phi(z_{2})) -  \sfd_{Y}^2(z_{1}, z_{2})| }{ 
  \sfd_{Y'} (\phi(z_{1}), \phi(z_{2})) +  \sfd_{Y}(z_{1}, z_{2}) }  
  \nonumber \\
  &   \leq | \sfd_{Y'}(\phi(z_{1}), \phi(z_{2})) -  \sfd_{Y}(z_{1}, z_{2})| \leq \vare.
\end{align}
If instead $\sfd_{Y}(z_{1},z_{2})\geq \varepsilon$,  proceeding as in \eqref{eq:estdYleqeps}  we obtain
\begin{align}\label{eq:estdYgeqeps}
|\sfd_{\R^{k}\times Y'}(\varphi(a_{1},z_{1}), \varphi(a_{2},z_{2}))-\sfd_{\R^{k}\times Y}((a_{1},z_{1}),(a_{2},z_{2}))| & \leq 
\frac{   \varepsilon^{2}+2 \varepsilon\, \sfd_{Y}(z_{1},z_{2}) }
{ \sfd_{\R^{k}\times Y'}(\varphi(a_{1},z_{1}), \varphi(a_{2},z_{2}))+\sfd_{\R^{k}\times Y}((a_{1},z_{1}),(a_{2},z_{2})) } \nonumber \\
& \leq  \frac{   \varepsilon^{2}+2 \varepsilon\, \sfd_{Y}(z_{1},z_{2}) } 
{  \sfd_{Y}(z_{1},z_{2}) } \nonumber \\
&   \leq 3 \varepsilon.
\end{align}
Combining \eqref{eq:estdYleqeps} with \eqref{eq:estdYgeqeps}, we obtain the claim. 

\textbf{Step 2}. Control of the measure distortion. 
\\In order to obtain the closeness of the measures $\varphi_{\sharp}\left({\mathcal L}^{k}\otimes \mm_{Y} \llcorner \bar{B}^{\R^k}_{r}(0^{k}) \times \bar{B}^{Y}_{r}(y)\right)$ and ${\mathcal L}^{k}\otimes \mm_{Y'}\llcorner  \bar{B}^{\R^k}_{r}(0^{k}) \times \bar{B}^{Y'}_{r}(y')$, it is enough to notice that for each $\psi_{1}\in C(\R^k), \, \psi_{2}\in C(Y')$ with $\int_{\bar{B}^{\R^k}_{r}(0^{k})} \psi_{1} \,d{\mathcal L}^{k} =1$ it holds
\begin{align*}
&\left|\int   \psi_{1} \otimes \psi_{2} \ d \varphi_{\sharp}\left({\mathcal L}^{k}\otimes \mm_{Y} \llcorner \bar{B}^{\R^k}_{r}(0^{k}) \times \bar{B}^{Y}_{r}(y)\right)- \int   \psi_{1} \otimes \psi_{2}\, d \left({\mathcal L}^{k}\otimes \mm_{Y'}\llcorner  \bar{B}^{\R^k}_{r}(0^{k}) \times \bar{B}^{Y'}_{r}(y')\right)  \right| \nonumber \\
&\quad = \left|\int  \psi_{2} \ d \varphi_{\sharp}\left(m_{Y} \llcorner \times \bar{B}^{Y}_{r}(y)\right)- \int \psi_{2}\, d \left(\mm_{Y'}\llcorner  \bar{B}^{Y'}_{r}(y')\right)  \right|,
\end{align*}
where we used Fubini-Tonelli's Theorem. 
\end{proof}

%%%%%%%%%%%%%%%%%%%%%%%%%%%
%%%%%%%%%%%%%%%%%%%%%%%%%%%
%%%%%%%%%%%%%%%%%%%%%%%%%%%%
%%%%%%%%%%%%%%%%%%%%%%%%%%%%

\bibliographystyle{amsalpha}
\bibliography{ColdingRCD.bib}

\providecommand{\bysame}{\leavevmode\hbox to3em{\hrulefill}\thinspace}
\providecommand{\MR}{\relax\ifhmode\unskip\space\fi MR }
\providecommand{\MRhref}[2]{%
  \href{http://www.ams.org/mathscinet-getitem?mr=#1}{#2}
}
\providecommand{\href}[2]{#2}
\begin{thebibliography}{AGMR15}

\bibitem[AGMR15]{AGMR}
Luigi Ambrosio, Nicola Gigli, Andrea Mondino, and Tapio Rajala,
  \emph{{R}iemannian {R}icci curvature lower bounds in metric measure spaces
  with $\sigma$-finite measure}, Trans. Amer. Math. Soc. \textbf{367} (2015),
  no.~7, 4661--4701.

\bibitem[AGS14]{AGS11}
Luigi Ambrosio, Nicola Gigli, and Giuseppe Savar\'{e}, \emph{Metric measure
  spaces with {R}iemannian {R}icci curvature bounded from below}, Duke Math. J.
  \textbf{163} (2014), no.~7, 1405--1490.

\bibitem[AMS19]{AMS2013}
Luigi Ambrosio, Andrea Mondino, and Giuseppe Savar\'{e}, \emph{Nonlinear
  diffusion equations and curvature conditions in metric measure spaces}, Mem.
  Amer. Math. Soc. \textbf{262} (2019), no.~1270, v+121.

\bibitem[BBI01]{BuragoBuragoIvanov}
Dmitri Burago, Yuri Burago, and Sergei Ivanov, \emph{A course in metric
  geometry}, Graduate Studies in Mathematics, vol.~33, American Mathematical
  Society, Providence, RI, 2001.

\bibitem[Bre72]{bredon}
Glen~E. Bredon, \emph{Introduction to compact transformation groups}, Pure and
  Applied Mathematics, vol.~46, Academic Press, 1972.

\bibitem[BS10]{BS2010}
Kathrin Bacher and Karl-Theodor Sturm, \emph{Localization and tensorization
  properties of the curvature-dimension condition for metric measure spaces},
  J. Funct. Anal. \textbf{259} (2010), no.~1, 28--56.

\bibitem[BS20]{BrueSemola}
Elia Bru\`e and Daniele Semola, \emph{Constancy of the dimension for
  $\mathsf{RCD}{(K,N)}$ spaces via regularity of {L}agrangian flows}, Comm.
  Pure Appl. Math. \textbf{73} (2020), 1141--1204.

\bibitem[CC96]{CC96}
Jeff Cheeger and Tobias~H. Colding, \emph{Lower bounds on {R}icci curvature and
  the almost rigidity of warped products}, Ann. of Math. (2) \textbf{144}
  (1996), no.~1, 189--237.

\bibitem[CC97]{CC97}
\bysame, \emph{On the structure of spaces with {R}icci curvature bounded below.
  {I}}, J. Differential Geom. \textbf{46} (1997), no.~3, 406--480.

\bibitem[CC00a]{CC00a}
\bysame, \emph{On the structure of spaces with {R}icci curvature bounded below.
  {II}}, J. Differential Geom. \textbf{54} (2000), no.~1, 13--35.

\bibitem[CC00b]{CC00b}
\bysame, \emph{On the structure of spaces with {R}icci curvature bounded below.
  {III}}, J. Differential Geom. \textbf{54} (2000), no.~1, 37--74.

\bibitem[CG72]{ChGr}
Jeff Cheeger and Detlef Gromoll, \emph{The splitting theorem for manifolds of
  nonnegative {R}icci curvature}, J. Differential Geometry \textbf{6}
  (1971/72), 119--128.

\bibitem[CM]{CavallettiMilman}
Fabio Cavalletti and Emanuel Milman, \emph{The globalization theorem for the
  {C}urvature-{D}imension condition}, Invent. Math. \textbf{226}, 1--137 (2021). 
  
\bibitem[CM18]{CMIMRN}
Fabio Cavalletti and Andrea Mondino, \emph{Almost {E}uclidean isoperimetric
  inequalities in spaces satisfying local {R}icci curvature lower bounds},
  Intern. Math. Res. Not. IMRN \textbf{2020} (2018), no.~5, 1481--1510.

\bibitem[CN12]{CN}
Tobias~H. Colding and Aaron Naber, \emph{Sharp {H}\"{o}lder continuity of
  tangent cones for spaces with a lower {R}icci curvature bound and
  applications}, Ann. of Math. (2) \textbf{176} (2012), no.~2, 1173--1229.

\bibitem[Col97]{Col97}
Tobias~H. Colding, \emph{Ricci curvature and volume convergence}, Ann. of Math.
  (2) \textbf{145} (1997), no.~3, 477--501.

\bibitem[CS12]{CS12}
Fabio Cavalletti and Karl-Theodor Sturm, \emph{Local curvature-dimension
  condition implies measure-contraction property}, J. Funct. Anal. \textbf{262}
  (2012), no.~12, 5110--5127. \MR{2916062}

\bibitem[Den20]{Deng}
Qin Deng, \emph{H{\"o}lder continuity of tangent cones in $\mathsf{RCD}{(K,N)}$
  spaces and applications to non-branching}, Preprint arXiv:2009.07956 (2020).

\bibitem[DPG18]{DPhG}
Guido De~Philippis and Nicola Gigli, \emph{Non-collapsed spaces with {R}icci
  curvature bounded from below}, J. \'{E}c. polytech. Math. \textbf{5} (2018),
  613--650.

\bibitem[DPMR17]{DPhMR}
Guido De~Philippis, Andrea Marchese, and Filip Rindler, \emph{On a conjecture
  of {C}heeger}, Measure theory in non-smooth spaces, Partial Differ. Equ.
  Meas. Theory, De Gruyter Open, Warsaw, 2017, pp.~145--155.

\bibitem[EKS15]{EKS}
Matthias Erbar, Kazumasa Kuwada, and Karl-Theodor Sturm, \emph{On the
  equivalence of the entropic curvature-dimension condition and {B}ochner's
  inequality on metric measure spaces}, Invent. Math. \textbf{201} (2015),
  no.~3, 993--1071.

\bibitem[Fuk86]{Fu}
Kenji Fukaya, \emph{Theory of convergence for {R}iemannian orbifolds}, Japan.
  J. Math. (N.S.) \textbf{12} (1986), no.~1, 121--160.

\bibitem[FY92]{FuYa}
Kenji Fukaya and Takao Yamaguchi, \emph{The fundamental groups of almost
  non-negatively curved manifolds}, Ann. of Math. (2) \textbf{136} (1992),
  no.~2, 253--333.

\bibitem[Gal83]{Gallot}
Sylvestre Gallot, \emph{Bornes universelles pour des invariants
  g\'eom\'etriques}, S\'eminaire de th\'eorie spectrale et g\'eom\'etrie
  \textbf{1} (1982-1983), 1--20 (fr), talk:2.

\bibitem[Gig13]{GigliSplitting}
Nicola Gigli, \emph{The splitting theorem in non-smooth context}, Preprint
  arXiv:1302.5555 (2013).

\bibitem[Gig15]{Gigli12}
\bysame, \emph{On the differential structure of metric measure spaces and
  applications}, Mem. Amer. Math. Soc. \textbf{236} (2015), no.~1113, vi+91.

\bibitem[GMS15]{GMS2013}
Nicola Gigli, Andrea Mondino, and Giuseppe Savar\'{e}, \emph{Convergence of
  pointed non-compact metric measure spaces and stability of {R}icci curvature
  bounds and heat flows}, Proc. Lond. Math. Soc. (3) \textbf{111} (2015),
  no.~5, 1071--1129.

\bibitem[GP]{GigliPasqualetto}
Nicola Gigli and Enrico Pasqualetto, \emph{Behaviour of the reference measure
  on $\mathsf{RCD}$ spaces under charts}, arXiv:1607.05188, to appear in Comm.
  Anal. Geom.

\bibitem[GR18]{GiRi}
Nicola Gigli and Chiara Rigoni, \emph{Recognizing the flat torus among
  {$\mathsf{RCD}^*(0,N)$} spaces via the study of the first cohomology group},
  Calc. Var. Partial Differential Equations \textbf{57} (2018), no.~4.

\bibitem[Gro81]{Gromov81}
Mikhael Gromov, \emph{Structures m\'{e}triques pour les vari\'{e}t\'{e}s
  {R}iemanniennes}, Textes Math\'{e}matiques [Mathematical Texts], vol.~1,
  CEDIC, Paris, 1981, Edited by J. Lafontaine and P. Pansu.

\bibitem[Gro07]{Gromov}
Misha Gromov, \emph{Metric structures for {R}iemannian and non-{R}iemannian
  spaces}, english ed., Modern Birkh\"{a}user Classics, Birkh\"{a}user Boston,
  Inc., Boston, MA, 2007, Based on the 1981 French original, With appendices by
  M. Katz, P. Pansu and S. Semmes, Translated from the French by Sean Michael
  Bates.

\bibitem[Hat02]{Hatcher}
Allen Hatcher, \emph{Algebraic topology}, Cambridge University Press,
  Cambridge, 2002.

\bibitem[Hon20]{H19}
Shouhei Honda, \emph{New differential operator and non-collapsed $\mathsf{RCD}$
  spaces}, Geometry \& Topology \textbf{24} (2020), no.~4, 2127--2148.

\bibitem[KL16]{KL}
Yu~Kitabeppu and Sajjad Lakzian, \emph{Characterization of low dimensional
  {$\mathsf{RCD}^*(K, N)$} spaces}, Analysis and Geometry in Metric Spaces
  \textbf{4} (2016), no.~1, 187--215.

\bibitem[KM18]{MondinoKell}
Martin Kell and Andrea Mondino, \emph{On the volume measure of non-smooth
  spaces with {R}icci curvature bounded below}, Ann. Sc. Norm. Super. Pisa Cl.
  Sci. (5) \textbf{18} (2018), no.~2, 593--610.

\bibitem[KM21]{MondinoKapovitch}
Vitali Kapovitch and Andrea Mondino, \emph{On the topology and the boundary of
  {$N$}-dimensional {$\mathsf{RCD}(K, N)$} spaces}, Geometry \& Topology
  \textbf{25} (2021), no.~1, 445--495.

\bibitem[LV09]{Lott-Villani09}
John Lott and C\'{e}dric Villani, \emph{Ricci curvature for metric-measure
  spaces via optimal transport}, Ann. of Math. (2) \textbf{169} (2009), no.~3,
  903--991.

\bibitem[MN19]{MondinoNaber}
Andrea Mondino and Aaron Naber, \emph{Structure theory of metric measure spaces
  with lower {R}icci curvature bounds}, J. Eur. Math. Soc. (JEMS) \textbf{21}
  (2019), no.~6, 1809--1854.

\bibitem[Mun00]{Munkres}
James~R. Munkres, \emph{Topology}, Prentice Hall, Inc., Upper Saddle River, NJ,
  2000.

\bibitem[MW19]{MondinoWei}
Andrea Mondino and Guofang Wei, \emph{On the universal cover and the
  fundamental group of an {${\rm RCD}^*(K,N)$}-space}, J. Reine Angew. Math.
  \textbf{753} (2019), 211--237.

\bibitem[Pet16]{Petersen}
Peter Petersen, \emph{Riemannian geometry}, third ed., Graduate Texts in
  Mathematics, vol. 171, Springer, Cham, 2016.

\bibitem[RS14]{RajalaSturm}
Tapio Rajala and Karl-Theodor Sturm, \emph{Non-branching geodesics and optimal
  maps in strong {$\mathsf{CD}(K,\infty)$}-spaces}, Calc. Var. Partial
  Differential Equations \textbf{50} (2014), no.~3-4, 831--846.

\bibitem[Spa66]{Spanier}
Edwin~H. Spanier, \emph{Algebraic topology}, Springer-Verlag, New York, 1966.

\bibitem[Stu06a]{Sturm06I}
Karl-Theodor Sturm, \emph{On the geometry of metric measure spaces. {I}}, Acta
  Math. \textbf{196} (2006), no.~1, 65--131.

\bibitem[Stu06b]{Sturm06II}
\bysame, \emph{On the geometry of metric measure spaces. {II}}, Acta Math.
  \textbf{196} (2006), no.~1, 133--177.

\bibitem[SW01]{SormaniWei2001}
Christina Sormani and Guofang Wei, \emph{Hausdorff convergence and universal
  covers}, Trans. Amer. Math. Soc. \textbf{353} (2001), 3585--3602.

\bibitem[SW04a]{SormaniWei}
\bysame, \emph{The covering spectrum of a compact length space}, J.
  Differential Geom. \textbf{67} (2004), no.~1, 35--77.

\bibitem[SW04b]{SormaniWei2004}
\bysame, \emph{Universal covers for {H}ausdorff limits of noncompact spaces},
  Trans. Amer. Math. Soc. \textbf{356} (2004), no.~3, 1233--1270.

\bibitem[Vil09]{Villani09}
C\'{e}dric Villani, \emph{Optimal transport}, Grundlehren der Mathematischen
  Wissenschaften [Fundamental Principles of Mathematical Sciences], vol. 338,
  Springer-Verlag, Berlin, 2009, Old and new.

\end{thebibliography}
\end{document}